\documentclass[a4paper,twoside, reqno]{amsart}

\usepackage{amsmath}
\usepackage{amsthm}
\usepackage{amscd}
\usepackage{amssymb}
\usepackage{array}
\usepackage{hyperref}
\usepackage[pdftex]{graphicx}
\usepackage[latin1]{inputenc}
\usepackage{latexsym}
\usepackage{mathdots}
\usepackage{multirow}
\usepackage{stmaryrd}
\usepackage[all]{xy}

\newtheorem{theorem}{Theorem}[subsection]
\newtheorem{lemma}[theorem]{Lemma}
\newtheorem{corollary}[theorem]{Corollary}
\newtheorem{definition}[theorem]{Definition}
\newtheorem{proposition}[theorem]{Proposition}
\newtheorem{remark}[theorem]{Remark}
\newtheorem{example}[theorem]{Example}
\newtheorem{deflem}[theorem]{Definition + Lemma}

\newcommand{\Ad}[0]{\operatorname{Ad}}

\newcommand{\C}[0]{\mathcal{C}}
\newcommand{\core}[0]{\operatorname{core}}

\newcommand{\diam}[0]{\operatorname{diam}}
\renewcommand{\dim}[0]{\operatorname{dim}}
\newcommand{\dist}[0]{\operatorname{dist}}
\newcommand{\ev}[0]{\operatorname{ev}}
\newcommand{\her}[0]{\operatorname{her}}
\newcommand{\id}[0]{\operatorname{id}}
\renewcommand{\Im}[0]{\operatorname{Im}}
\newcommand{\image}[0]{\operatorname{im}}

\newcommand{\Lim}[0]{\operatorname{Lim}}
\newcommand{\M}[0]{\mathbb{M}}

\newcommand{\order}[0]{\operatorname{order}}
\newcommand{\prim}[0]{\operatorname{Prim}}
\newcommand{\rank}[0]{\operatorname{rank}}
\renewcommand{\Re}[0]{\operatorname{Re}}

\newcommand{\topdim}[0]{\operatorname{topdim}}
\newcommand{\tr}[0]{\operatorname{tr}}
\newcommand{\Z}[0]{\mathbb{Z}}

\title[Semiprojectivity for subhomogeneous $C^*$-algebras]{A characterization of semiprojectivity \\ for subhomogeneous $C^*$-algebras}
\date{July 14, 2014}

\author{Dominic Enders}
\thanks{This work was supported by the SFB 878 {\itshape Groups, Geometry and Actions} and the Danish National Research Foundation through the {\itshape Centre for Symmetry and Deformation} (DNRF92).}

\subjclass[2010]%
{Primary
46L05 
 ; Secondary
46L80, 
46L85, 
54C55, 
54F50, 
}

\address{Dominic Enders
\newline Department of Mathematical Sciences, University of Copenhagen
\newline Universitetsparken 5, DK-2100 Copenhagen \O, Denmark}
\email{d.enders@math.ku.dk}

\begin{document}

\begin{abstract}

We study semiprojective, subhomogeneous $C^*$-algebras and give a detailed description of their structure. In particular, we find two characterizations of semiprojectivity for subhomogeneous $C^*$-algebras: one in terms of their primitive ideal spaces and one by means of special direct limit structures over one-dimensional NCCW complexes. These results are obtained by working out several new permanence results for semiprojectivity, including a complete description of its behavior with respect to extensions by homogeneous $C^*$-algebras.
\end{abstract}

\maketitle

\section{Introduction}

The concept of semiprojectivity is a type of perturbation theory for $C^*$-algebras which has become a frequently used tool in many different aspects of $C^*$-algebra theory. 
Due to a certain kind of rigidity, semiprojective $C^*$-algebras are technically important in various situations.
In particular, the existence and comparison of limit structures via approximate interwinings, which is an integral part of the Elliott classification program, often relies on perturbation properties of this type. This is one of the reasons why direct limits over semiprojective $C^*$-algebras, e.g., AF- or A$\mathbb{T}$-algebras, are particularly tractable and one therefore constructs models preferably from semiprojective building blocks. The most popular of those are without doubt the non-commutative CW-complexes (NCCWs) introduced by Eilers, Loring and Pedersen. These are in fact semiprojective in dimension one (\cite{ELP98}, but see also \cite{End14}). In this paper, we study semiprojectivity for general subhomogeneous $C^*$-algebras and see whether there exist more interesting examples, i.e., besides the one-dimensional NCCW complexes (1-NCCWs), that could possibly serve as useful building blocks in the construction of ASH-algebras. In Theorem \ref{thm structure}, we give two characterizations of semiprojectivity for subhomogenous $C^*$-algebras: an abstract one in terms of primitive ideal spaces and a concrete one by means of certain limit structures. These show that it is quite a restriction for a subhomogeneous $C^*$-algebra to be semiprojective, though many examples beyond the class of 1-NCCWs exist. On the other hand, a detailed study of the structure of these algebras further reveals that they can always be approximated by 1-NCCWs in a very strong sense, see Corollary \ref{cor approx nccw}, and hence essentially share the same properties.

The work of this paper is based on the characterization of semiprojectivity for commutative $C^*$-algebras, which was recently obtained by S{\o}rensen and Thiel in \cite{ST12}. They showed that a commutative $C^*$-algebra $\C(X)$ is semiprojective if and only if $X$ is an absolute neighborhood retract of dimension at most 1 (a 1-ANR), thereby confirming a conjecture of Blackadar and generalizing earlier work of Chigogidze and Dranishnikov on the projective case (\cite{CD10}). Their characterization further applies to trivially homogeneous $C^*$-algebras, i.e. to algebras of the form $\C(X,\M_n)$. In a first step, we generalize their result to general homogeneous $C^*$-algebras. The main difficulty, however, is to understand which ways of 'gluing together' several homogeneous $C^*$-algebras preserve semiprojectivity, or more precisely: Which extensions of semiprojective, homogeneous $C^*$-algebras are again semiprojective? Conversely, is semiprojectivity preserved when passing to a homogeneous subquotient? These questions essentially ask for the permanence behavior of semiprojectivity along extensions of the form $0\rightarrow\C_0(X,\M_n)\rightarrow A\rightarrow B\rightarrow 0$. While it is known that the permanence properties of semiprojectivity with respect to extensions are rather bad in general, we are able to work out a complete description of its behavior in the special case of extensions by homogeneous ideals, see Theorem \ref{thm 2 out of 3}. With this permanence result at hand, it is then straightforward to characterize semiprojectivity for subhomogeneous $C^*$-algebras in terms of their primitive ideal spaces. In particular, it is a necessary condition that the subspaces corresponding to a fixed dimension are all 1-ANRs. Combining this with the structure result for one-dimensional ANR-spaces from \cite{ST12}, we further obtain a more concrete description of semiprojective, subhomogeneous $C^*$-algebras by identifying them with certain special direct limits of 1-NCCWs.\\

\noindent
This paper is organized as follows. In section \ref{section preliminaries}, we briefly recall some topological definitions and results that will be used troughout the paper. We further remind the reader of some facts about semiprojectivity, subhomogeneous $C^*$-algebras and their primitive ideal spaces. We then start by constructing a lifting problem which is unsolvable for strongly quasidiagonal $C^*$-algebras. This lifting problem then allows us to extend the results of \cite{ST12} from the commutative to the homogeneous case.

Section \ref{section semiprojectivity} contains a number of new contructions for semiprojective $C^*$-algebras. We first introduce a technique to extend lifting problems, a method that can be used to show that in certain situations semiprojectivity passes to ideals. After that, we introduce a class of maps which give rise to direct limits that preserve semiprojectivity. Important examples of such maps are given and discussed.

Section \ref{section extensions} is devoted to the study of extensions by homogeneous $C^*$-algebras, i.e. extensions of the form $0\rightarrow\C_0(X,\M_n)\rightarrow A\rightarrow B\rightarrow 0$. In \ref{section retract maps}, we define and study a certain set-valued retract map $R\colon \prim(A)\rightarrow 2^{\prim(B)}$ associated to such an extension. We discuss regularity concepts for $R$, i.e. continuity and finiteness conditions, and show how regularity of $R$ relates to lifting properties of the corresponding Busby map and, by that, to splitting properties of the extension itself. In particular, we identify conditions under which regularity of $R$ implies the existence of a splitting map $s\colon B\rightarrow A$ with good multiplicative properties. After that, we verify the required regularity properties for $R$ in the case of a semiprojective extension $A$. In section \ref{section existence limit structures} it is shown how certain limit structures for the space $X$ give rise to limit structures for the extension $A$, again provided that the associated retract map $R$ is sufficiently regular. Putting all these results together in \ref{section keeping track}, we find a '2 out of 3'-type statement, Theorem \ref{thm 2 out of 3}, which gives a complete description for the behavior of semiprojectivity along extensions of the considered type. 

In section \ref{section main}, we use this permanence result to work out two characterizations of semiprojectivity for subhomogeneous $C^*$-algebras. These are presented in Theorem \ref{thm structure}, the main result of this paper. Based on this, we find a number of consequences for the structure of these algebras, e.g. information about their $K$-theory and dimension. Further applications, such as closure and approximation properties, are discussed in \ref{section applications}. We finish by illustrating how this also gives a simple method to exclude semiprojectivity and show that the higher quantum permutation algebras are not semiprojective.

\section{Preliminaries}\label{section preliminaries}
\subsection{The structure of 1-dimensional ANR-spaces}\label{section 1-ANR}
We are particularly interested in ANR-spaces of dimension at most one. The structure of these spaces has been studied and described in detail in \cite[section 4]{ST12}. Here we recall the most important notions and results. More information about ANR-spaces can be found in \cite{Bor67}. For proofs and further reading on the theory of continua, we refer the reader to Nadler's book \cite{Nad92}. 

\begin{definition}\label{def anr}
A compact, metric space $X$ is an absolute retract (abbreviated AR-space) if every map $f\colon Z\rightarrow X$ from a closed subspace $Z$ of a compact, metric space $Y$ extends to a map $g\colon Y\rightarrow X$,
i.e. $g\circ\iota=f$ with $\iota\colon Z\rightarrow Y$ the inclusion map:
\[\xymatrix{
  & Y \ar@{-->}[dl]_g \\
  X & Z \ar[l]^f \ar[u]_\iota
}\]
\noindent If every map $f\colon Z\rightarrow X$ from a closed subspace $Z$ of a compact, metric space $Y$ extends to a map $g\colon V\rightarrow X$ on a closed neighborhood $V$ of $Z$
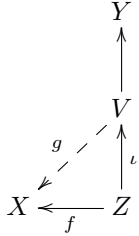
\[\xymatrix{
& Y \\ 
& V \ar@{-->}[dl]_g \ar[u] \\ 
X & Z \ar[l]^f \ar[u]_\iota
}\]
then $X$ is called an absolute neighborhood retract (abbreviated ANR-space).
\end{definition}

A compact, locally connected, metric space is called a Peano space. A connected Peano space is called a Peano continuum. Now given an ANR-space $X$, we can embed it into the Hilbert cube $\mathcal{Q}$ and obtain a retract from a neighborhood of $X$ in $\mathcal{Q}$ onto $X$. Hence an ANR-space inherits all local properties of the Hilbert cube which are preserved under retracts. These properties include local connectedness, so that all ANR-spaces are Peano spaces. The converse, however, is not true in general. But as we will see, it is possible to identify the ANR-spaces among all Peano spaces, at least in the one-dimensional case.

A closed subspace $Y$ of a space $X$ is a retract of $X$ if there exists a continuous map $r\colon X\rightarrow Y$ such that $r_{|Y}=\id_Y$. If the retract map $r\colon X\rightarrow Y$ regarded as a map to $X$ is homotopic to the identity, then $Y$ is called a deformation retract of $X$. It is a strong deformation retract if in addition the homotopy can be chosen to fix the subspace $Y$.
The following concept of a core continuum is due to Meilstrup. It is crucial for understanding the structure of one-dimensional ANR-spaces. 

\begin{deflem}[\cite{Mei05}]\label{def core}
Let $X$ be a non-contractible one-dimensional Peano continuum. Then there exists a unique strong deformation retract which contains no further proper deformation retract. We call it the core of $X$ and denote it by $\core(X)$. 
\end{deflem}

As in \cite{ST12}, we define the core of a contractible, one-dimensional Peano continuum to be any fixed point. Many questions about one-dimensional Peano continua can be reduced to questions about their cores. This reduction step uses a special retract map, the so-called first point map: 

\begin{deflem}[{\cite[4.14-16]{ST12}}]\label{def fpm}
Let $X$ be a one-dimensional Peano continuum and $Y$ a subcontinuum with $\core(X)\subset Y$. For each $x\in X\backslash Y$ there is a unique point $r(x)\in Y$ such that $r(x)$ is a point of an arc in $X$ from $x$ to any point of $Y$. Setting $r(x)=x$ for all $x\in Y$, we obtain a map $r\colon X\rightarrow Y$. This map is called the first point map, it is continuous and a strong deformation retract from $X$ onto $Y$. 
\end{deflem}

The following follows directly from the proof of \cite[Lemma 4.14]{ST12}.

\begin{lemma}\label{lemma arc to core}
Let $X$ be a one-dimensional Peano continuum, $Y\subseteq X$ a subcontinuum containing $\core(X)$ and $r\colon X\rightarrow Y$ the first point map onto $Y$. Then the following is true:
\renewcommand{\labelenumi}{(\roman{enumi})}
\begin{enumerate} 
  \item For every point $x\in X\backslash Y$ there exists an arc from $x$ to $r(x)\in Y$ which is unique up to reparametrization.
  \item If $\alpha$ is a path from $x\in X\backslash Y$ to $y\in Y$, then $r(\image(\alpha))\subseteq\image(\alpha)$.
\end{enumerate} 
\end{lemma}

The simplest example of a one-dimensional Peano space is a graph, i.e. a finite, one-dimensional CW-complex. The order of a point $x$ in a graph $X$ is defined as the smallest number $n\in\mathbb{N}$ such that 
for every neighborhood $V$ of $x$ there exists an open neighborhood $U\subseteq V$ of $x$ with $|\partial U|=|\overline{U}\backslash U| \leq n$. We denote the order of $x$ in $X$ by $\order(x,X)$.

Given a one-dimensional Peano continuum $X$, one can reconstruct the space $X$ from its core by 'adding' the arcs which connect points of $X\backslash\core(X)$ with the core as described in \ref{lemma arc to core}. This procedure yields a limit structure for one-dimensional Peano spaces which first appeared as Theorem 4.17 of \cite{ST12}. In the case of one-dimensional ANR-spaces, the core is a finite graph and hence the limit structure entirely consists of finite graphs. 

\begin{theorem}[{\cite[Theorem 4.17]{ST12}}]\label{thm ST}
Let $X$ be a one-dimensional Peano continuum. Then there exists a sequence $\{Y_k\}_{k=1}^\infty$ such that
\renewcommand{\labelenumi}{(\roman{enumi})}
\begin{enumerate}
  \item each $Y_k$ is a subcontinuum of $X$. \item $Y_k\subset Y_{k+1}$. 
  \item $\lim_kY_k=X$. \item $Y_1=\core(X)$ and for each $k$, $Y_{k+1}$ is obtained from $Y_k$ by attaching a line segment at a single point, i.e., $\overline{Y_{k+1}\backslash Y_k}$ is an arc with end point $p_k$ such that $\overline{Y_{k+1}\backslash Y_k}\cap Y_k=\{p_k\}$.
  \item letting $r_k\colon X\rightarrow Y_k$ be the first point map for $Y_k$ we have that $\{r_k\}_{k=1}^\infty$ converges uniformly to the identity map on $X$.
\end{enumerate}
If $X$ is also an ANR, then all $Y_k$ are finite graphs. If $X$ is even contractible (i.e. an AR), then $\core(X)$ is just some point and all $Y_k$ are finite trees.
\end{theorem}

We will need a local criterion for identifying one-dimensional ANR-spaces among general Peano spaces. It was observed by Ward how to get such a characterization in terms of embeddings of circles. 

\begin{definition}\label{def small circles}
Let $X$ be a compact, metric space, then $X$ does not contain small circles if there is an $\epsilon>0$ such that $\diam(\iota(S^1))\geq\epsilon$ for every embedding $\iota\colon S^1\rightarrow X$.
\end{definition}

Note that the property of containing arbitrarily small circles does not depend on the particular choice of metric. 

\begin{theorem}[\cite{War60}]\label{thm ward}
For a Peano space $X$ the following are equivalent: 
\renewcommand{\labelenumi}{(\roman{enumi})}
\begin{enumerate}
  \item $X$ does not contain small circles. 
  \item $X$ is an ANR-space of dimension at most one.
\end{enumerate}
\end{theorem}

This statement can also be interpreted as follows. As was shown independently by Bing (\cite{Bin49}) and Moise (\cite{Moi49}), every Peano continuum $X$ admits a geodesic metric $d$. Now non-embeddability of circles into $X$ is the same as uniqueness of geodesics in $X$. More precisely, a Peano continuum is a one-dimensional AR-space if and only if there is no embedding $S^1\hookrightarrow X$ if and only if $X$ admits unique geodesics. Similarly, Theorem \ref{thm ward} can be read as: A Peano continuum $X$ is a one-dimensional ANR-space if and only if it has locally unique geodesics, meaning that there exists $\epsilon>0$ such that any two points with distance smaller then $\epsilon$ can be joined by a unique geodesic.

\subsection{Subhomogeneous $C^*$-algebras}\label{section subhomogeneous}

In this section we collect some well known results on subhomogeneous $C^*$-algebras. In particular, we recall some facts on their primitive ideal spaces. More detailed information can be found in \cite[Chapter 3]{Dix77} and \cite[Section IV.1.4]{Bla06}.

\begin{definition}
Let $N\in\mathbb{N}$. A $C^*$-algebra $A$ is $N$-homogeneous if all its irreducible representations are of dimension $N$. $A$ is $N$-subhomogeneous if every irreducible representation of $A$ has dimension at most $N$.
\end{definition}

The standard example of a $N$-homogeneous $C^*$-algebra is $\C_0(X,\M_N)$ for some locally compact space $X$. As the next proposition shows, subhomogeneous $C^*$-algebras can be characterized as subalgebras of such. 
A proof of this fact can be found in \cite[IV.1.4.3-4]{Bla06}. 

\begin{proposition}\label{prop sub homogeneous}
A $C^*$-algebra $A$ is $N$-subhomogeneous if and only if it is isomorphic to a subalgebra of some $N$-homogeneous $C^*$-algebra $\C(X,\M_N)$. If $A$ is separable, we may choose $X$ to be the Cantor set $K$.
\end{proposition}

\begin{example}[1-NCCWs]\label{ex 1-nccw}
One of the most important examples of subhomogeneous $C^*$-algebras is the class of non-commutative CW-complexes (NCCWs) defined by Eilers, Loring and Pedersen in \cite{ELP98}. The one-dimensional NCCWs, which we will abbreviate by 1-NCCWs, are defined as pullbacks of the form
\[\xymatrix{
  \text{1-NCCW} \ar@{-->}[r] \ar@{-->}[d] & G \ar[d] \\
  \C([0,1],F) \ar[r]^(.58){\ev_0\oplus\ev_1} & F\oplus F
}\]
with $F$ and $G$ finite-dimensional $C^*$-algebras. These are particularly interesting since they are semiprojective by \cite[Theorem 6.2.2]{ELP98}.
\end{example}

For a subhomogeneous $C^*$-algebra $A$, the primitive ideal space $\prim(A)$, i.e. the set of kernels of irreducible representations endowed with the Jacobson topology, contains a lot of information. Another useful decription of the topology on $\prim(A)$ is given by the folllowing lemma which will make use of regularly. For an ideal $J$ in a $C^*$-algebra $A$ we write $\|x\|_J$ to denote the norm of the image of the element $x\in A$ in the quotient $A/J$.

\begin{lemma}[{\cite[II.6.5.6]{Bla06}}]\label{lemma top prim}
Let $A$ be a $C^*$-algebra.
\begin{enumerate}
  \item If $x\in A$, define $\check{x}\colon \prim(A)\rightarrow\mathbb{R}_{\geq 0}$ by $\check{x}(J)=\|x\|_J$. Then $\check{x}$ is lower semicontinuous.
  \item If $\{x_i\}$ is a dense set in the unit ball of $A$, and $U_i=\{J\in\prim(A)\colon\check{x_i}(J)>1/2\}$, then $\{U_i\}$ forms a base for the topology of $\prim(A)$.
  \item If $x\in A$ and $\lambda>0$, then $\{J\in\prim(A)\colon\check{x}(J)\geq\lambda\}$ is compact (but not necessarily closed) in $\prim(A)$.
\end{enumerate}
\end{lemma}

Since we will mostly be interested in finite-dimensional representations, we consider the subspaces 
\[\prim_n(A)=\{\ker(\pi)\in\prim(A)\colon\dim(\pi)=n\}\] 
for each finite $n$. Similarly, we write 
\[\prim_{\leq n}(A)=\{\ker(\pi)\in\prim(A)\colon\dim(\pi)\leq n\}=\bigcup_{k\leq n} \prim_k(A).\] 
The following theorem describes the structure of these subspaces of $\prim(A)$ and the relations between them.

\begin{theorem}[{\cite[3.6.3-4]{Dix77}}]\label{thm prim}
Let $A$ be a $C^*$-algebra. The following holds for each $n\in\mathbb{N}$:
\renewcommand{\labelenumi}{(\roman{enumi})}
\begin{enumerate}
\item $\prim_{\leq n}(A)$ is closed in $\prim(A)$. 
\item $\prim_n(A)$ is open in $\prim_{\leq n}(A)$. 
\item $\prim_n(A)$ is locally compact and Hausdorff.
\end{enumerate}
\end{theorem}

Now assume that $A$ is a $N$-subhomogeneous $C^*$-algebra. In this case Theorem \ref{thm prim} gives a set-theoretical (but in general not a topological) decomposition of its primitive spectrum
\[\prim(A)=\bigsqcup_{n=1}^N \prim_n(A).\]
While each subspace in this decomposition is nice, in the sense that it is Hausdorff, $\prim(A)$ itself is typically non-Hausdorff. In the subhomogeneous setting it is at least a $T_1$-space, i.e. points are closed. 
If we further assume $A$ to be separable and unital, the space $\prim(A)$ will also be separable and quasi-compact.

Given a general $C^*$-algebra $A$, there is a one-to-one correspondence between (closed) ideals $J$ of $A$ and closed subsets of $\prim(A)$. More precisely, one can identify $\prim(A/J)$ with the closed subset $\{K\in\prim(A)\colon J\subseteq K\}$. In particular, we can consider the quotient $A_{\leq n}$ corresponding to the closed subset $\prim_{\leq n}(A)\subseteq\prim(A)$. This quotient is the maximal $n$-subhomogeneous quotient of $A$ and has the following universal property: Any $^*$-homomorphism $\varphi\colon A\rightarrow B$ to some $n$-subhomogeneous $C^*$-algebra $B$ factors uniquelythrough $A_{\leq n}$: 
\[\xymatrix{
  A \ar[rr]^\varphi \ar@{->>}[dr]&& B \\
  & A_{\leq n} \ar@{-->}[ur]
}\]

\subsection{Semiprojective $C^*$-algebras}\label{section sp}
We recall the definition of semiprojectivity for $C^*$-algebras, the main property of study in this paper. More detailed information about lifting properties for $C^*$-algebras  can be found in Loring's book \cite{Lor97}.

\begin{definition}[{\cite[Definition 2.10]{Bla85}}]\label{def sp}
A separable $C^*$-algebra $A$ is semiprojective if for every $C^*$-algebra $B$ and every increasing chain of ideals $J_n$ in $B$ with $J_\infty=\overline{\bigcup_n J_n}$, and for every $^*$-homomorphism $\varphi\colon A\rightarrow B/J_\infty$ there exists $n\in\mathbb{N}$ and a $^*$-homomorphism $\overline{\varphi}\colon A\rightarrow B/J_n$ making the following diagram commute:
\[\xymatrix{
  & B \ar@{->>}[d]^{\pi_0^n} \\ & B/J_n \ar@{->>}[d]^{\pi_n^\infty} \\ 
  A \ar[r]^\varphi \ar@{-->}[ur]^{\overline{\varphi}} & B/J_\infty
}\] 
In this situation, the map $\overline{\varphi}$ is called a partial lift of $\varphi$. The $C^*$-algebra $A$ is projective if, in the situation above, we can always find a lift $\overline{\varphi}\colon A\rightarrow B$ for $\varphi$.

Let $\mathcal{C}$ be a class of $C^*$-algebras. A $C^*$-algebra $A$ is (semi)projective with respect to $\mathcal{C}$ if it satisfies the definitions above with the restriction that the $C^*$-algebras $B,B/J_n$ and $B/J_\infty$ all belong to the class $\mathcal{C}$. 
\end{definition}

\begin{remark}
One may also define semiprojectivity as a lifting property for maps to certain direct limits: an increasing sequence of ideals $J_n$ in $B$ gives an inductive system $(B/J_n)_n$ with surjective connecting maps $\pi_n^{n+1}\colon B/J_n\rightarrow B/J_{n+1}$ and limit (isomorphic to) $B/J_\infty$. On the other hand, it is easily seen that every such system gives an increasing chain of ideals $(\ker(\pi_1^n))_n$. Hence, semiprojectivity is equivalent to being able to lift maps to $\varinjlim D_n$ to a finite stage $D_n$ provided that all connecting maps of the system are surjective. It is sometimes more convenient to work in this picture.
\end{remark}

\subsubsection{An unsolvable lifting problem}\label{section unsolvable}

In order to show that a $C^*$-algebra does not have a certain lifting property, we need to construct unsolvable lifting problems. One such construction by Loring (\cite[Proposition 10.1.8]{Lor97}) uses the fact that normal elements in quotient $C^*$-algebras do not admit normal preimages in general, e.g. Fredholm operators of non-zero index. Here, we generalize Loring's construction and obtain a version which also works for almost normal elements. Combining this with Lin's theorem on almost normal matrices, we are able to construct unsolvable lifting problems not only for commutative $C^*$-algebras, as in Loring's case, but for the much larger class of strongly quasidiagonal $C^*$-algebras.

First we observe that almost normal elements in quotient $C^*$-algebras always admit (almost as) almost normal preimages. Given an element $x$ of some $C^*$-algebra and $\epsilon>0$, we say that $x$ is $\epsilon$-normal if $\|x^*x-xx^*\|\leq\epsilon\|x\|$ holds. 

\begin{lemma}\label{lemma preimage}
Let $A$, $B$ be $C^*$-algebras and $\pi\colon A\rightarrow B$ a surjective $^*$-homomorphism. Then for every $\epsilon$-normal element $y\in B$ there exists a $(2\epsilon)$-normal element $x\in A$ with $\pi(x)=y$ and $\|x\|=\|y\|$.
\end{lemma}

\begin{proof}
Let $(u_\lambda)_{\lambda\in\Lambda}$ denote an approximate unit for $\ker(\pi)$ which is quasicentral for $A$. Pick any preimage $x_0$ of $x$ with $\|x_0\|=\|x\|$ and set $x:=(1-u_{\lambda_0})x_0$ for a suitable $\lambda_0\in\Lambda$.
\end{proof}

The next lemma is due to Halmos. A short proof using the Fredholm alternative can be found in \cite[Lemma 2]{BH74}. 

\begin{lemma}[Halmos]\label{lemma Halmos}
Let $S\in\mathcal{B}(H)$ be a proper isometry, then 
\[\dist\left(S,\{N+K\,|\,N,K\in \mathcal{B}(H),\; N\,\text{normal},\; K\,\text{compact}\}\right)=1.\]
\end{lemma}

It is a famous result by H. Lin that in matrix algebras almost normal elements are uniformly close to normal ones (\cite{Lin97}). A short, alternative proof involving semiprojectivity arguments can be found in \cite{FR01}.

\begin{theorem}[Lin]\label{thm Lin}
For every $\epsilon>0$, there is a $\delta>0$ so that, for any $d$ and any $X$ in $\M_d$ satisfying 
\[\|XX^*-X^*X\|\leq\delta\;\;\;\text{and}\;\;\;\|X\|\leq 1\]
there is a normal $Y$ in $\M_d$ such that 
\[\|X-Y\|\leq\epsilon.\]
\end{theorem}

The following is the basis for most of our unsolvable lifting problems appearing in this paper. Recall that a $C^*$-algebra $A$ is strongly quasidiagonal if every representation of $A$ is quasidiagonal. See \cite[Section V.4.2]{Bla06} or \cite{Bro00} for more information on quasidiagonality.

In the following, let $\mathcal{T}$ denote the Toeplitz algebra $C^*(S|S^*S=1)$ and $\varrho\colon\mathcal{T}\rightarrow\C(S^1)$ the quotient map given by mapping $S$ to the canonical generator $z$ of $\C(S^1)$.

\begin{proposition}\label{prop nslp}
There exists $\delta>0$ such that the following holds for all $n\in\mathbb{N}$: If $A$ is strongly quasidiagonal and $\varphi\colon A\rightarrow\C(S^1)\otimes\M_n$ is any $^*$-homomorphism with $\dist(z\otimes 1_n,\image(\varphi))<\delta$, then $\varphi$ does not lift to a $^*$-homomorphism from $A$ to $\mathcal{T}\otimes\M_n$:
\[\xymatrix{
  && \mathcal{T}\otimes \M_n \ar@{->>}[d]^(0.4){\varrho\otimes\id}\\
  A \ar[rr]^\varphi \ar@{-->}@/^1pc/[urr]^(.4)\nexists && \C(S^1)\otimes \M_n
}\]
\end{proposition}

\begin{proof}
Choose $\delta'>0$ corresponding to $\epsilon=1/6$ as in Theorem \ref{thm Lin} and set $\delta=\delta'/14$. Let $a'\in A$ be such that $\|\varphi(a')-z\otimes 1_n\|<\delta$, then $\|[\varphi(a'),\varphi(a')^*]\|\leq 2\delta(\|\varphi(a')\|+1)<5\delta\|\varphi(a')\|$. Hence by Lemma \ref{lemma preimage} there exists a $(10\delta)$-normal element $a\in A$ with $\varphi(a)=\varphi(a')$ and $5/6<\|a\|=\|\varphi(a')\|<6/5$. Now if $\psi$ is a $^*$-homomorphism with $(\varrho\otimes\id)\circ\psi=\varphi$ as indicated, we regard $\psi$ as a representation on $\mathcal{H}^{\oplus n}$ with $\mathcal{T}$ generated by the unilateral shift $S$ on $\mathcal{H}$. By assumption, $\psi$ is then a quasidiagonal representation. In particular, $\psi(a)$ can be approximated arbitrarily well by block-diagonal operators (\cite[Theorem 5.2]{Bro00}). We may therefore choose a $(11\delta)$-normal block-diagonal operator $B$ with $5/6\leq\|B\|\leq 6/5$ within distance at most $1/3$ from $\psi(a)$. Applying Lin's Theorem to the normalized, $(14\delta)$-normal block-diagonal operator $\|B\|^{-1} B$ shows the existence of a normal element $N\in\mathcal{H}^{\oplus n}$ with $\|\psi(a)-N\|\leq 2/3$. But then we find
\[\begin{array}{rl}
  & \|(N-S\otimes 1_n)+\mathcal{K}(\mathcal{H}^{\oplus n})\| \\
  \leq & \|N-\psi(a)\|+\|(\varrho\otimes\id)(\psi(a)-S\otimes 1_n)\| \\
  \leq & \frac{2}{3}+\|\varphi(a')-z\otimes 1_n\| \\
  \leq & \frac{2}{3}+\delta<1
\end{array}\]
in contradiction to Lemma \ref{lemma Halmos}.
\end{proof}

\subsubsection{The homogeneous case}\label{section homogeneous}

In \cite{ST12}, A. S\o rensen and H. Thiel characterized semiprojectivity for commutative $C^*$-algebras. Moreover, they gave a description of semiprojectivity for homogeneous trivial fields, i.e. $C^*$-algebras of the form $\C_0(X,\M_N)$. Note that the projective case was settled earlier by A. Chigogidze and A. Dranishnikov in \cite{CD10}. Their result is as follows.

\begin{theorem}[\cite{ST12}]\label{thm comm case}
Let $X$ be a locally compact, metric space and $N\in\mathbb{N}$. Then the following are equivalent:
\begin{enumerate}
      \item $\C_0(X,\M_N)$ is (semi)projective.
      \item The one-point compactification $\alpha X$ is an A(N)R-space and $\dim(X)\leq 1$.
\end{enumerate}
\end{theorem}

The work of S{\o}rensen and Thiel will be the starting point for our analysis of semiprojectivity for subhomogeneous $C^*$-algebras. In this section, we reduce the general $N$-homogeneous case to their result by showing that semiprojectivity for homogeneous, locally trivial fields implies global triviality. We further obtain some information about parts of the primitive ideal space for general semiprojective $C^*$-algebras.

\begin{lemma}\label{lemma ideal Peano}
Let $I$ be a $N$-homogeneous ideal in a $C^*$-algebra $A$. If $A$ is semiprojective with respect to $N$-subhomogeneous $C^*$-algebras, then the one-point compactification $\alpha\prim(I)$ is a Peano space. If $A$ is semiprojective, we further have $\dim(\alpha\prim(A))\leq 1$.
\end{lemma}

\begin{proof}
Let $A_{\leq N}$ be the maximal $N$-subhomogeneous quotient of $A$, then $I$ is also an ideal in $A_{\leq N}$. Being $N$-homogeneous, the ideal $I$ is isomorphic to the section algebra $\Gamma_0(E)$ of a locally trivial $\M_N$-bundle $E$ over the locally compact, second countable, metrizable Hausdorff space $\prim(I)$ by \cite[Theorem 3.2]{Fel61}. Since $A_{\leq N}$ is separable and $N$-subhomogenous, we can embed it into $C(K,\M_N)$ with $K$ the Cantor set by Proposition \ref{prop sub homogeneous}. Using the well known middle-third construction of $K=\varprojlim_k (\bigsqcup^{2^k} [0,1])$, we can apply semiprojectivity of $A_{\leq N}$ with respect to $N$-subhomogenous $C^*$-algebras to obtain an embedding of $A_{\leq N}$ into $C([0,1]^{\oplus 2^k},\M_N)$ for some $k$. The restriction of this embedding to $I$ induces a continuous surjection $\pi$ of $\bigsqcup^{2^k} [0,1]$ onto $\alpha\prim(I)$. By the Hahn-Mazurkiewicz Theorem (\cite[Theorem 8.18]{Nad92}), this shows that $\alpha\prim(I)$ is a Peano space. Furthermore, we find a basis of compact neighborhoods consisting of Peano continua for any point $x$ of $\alpha\prim(I)$ by \cite[Theorem 8.10]{Nad92}.
 
Now let $A$ be semiprojective and assume that $\dim(\prim(I))=\dim(\alpha\prim(I))>1$. Arguing precisely as in \cite[Proposition 3.1]{ST12}, we use our basis of neighborhoods for points of $\prim(I)$ to find arbitrarily small circles around a point $x\in\prim(I)$. Using triviality of $E$ around $x$, we obtain a lifting problem for $A$:
\[\xymatrix{
  & A \ar[d] \ar@{-->}[r] & \left(\left(\bigoplus_\mathbb{N}\mathcal{T}\right)^+/\left(\bigoplus_1^n\mathbb{K}\right)\right)\otimes\M_N \ar@{->>}[d] \\ 
  I \ar[r] \ar[ur]^\subseteq&\left(\bigoplus_\mathbb{N}\C(S^1)\right)^+\otimes\M_N \ar@{=}[r] & \left(\left(\bigoplus_\mathbb{N}\mathcal{T}\right)^+/\left(\bigoplus_\mathbb{N}\mathbb{K}\right)\right)\otimes\M_N
}\]
Semiprojectivity of $A$ allows us to solve this lifting problem. Now restrict a partial lift to the ideal $I$ and consider its coordinates to obtain a commutative diagram
\[\begin{xy}\xymatrix{
  & \hspace{0.65cm}\mathcal{T}\otimes\M_N \ar@<0.15cm>@{->>}[d] \\
  I \ar@{->>}[r] \ar@{-->}[ur] & \C(S^1)\otimes\M_N.
}\end{xy}\]
The map on the bottom is surjective since it is induced by the inclusion of one of the circles around $x$. But a diagram like this does not exist by Proposition \ref{prop nslp} because $I$ is homogeneous and by that strongly quasidiagonal. 
\end{proof}

\begin{corollary}\label{cor sp peano}
Let $A$ be a semiprojective $C^*$-algebra, then $\alpha\prim_n(A)$ is a Peano space for every $n\in\mathbb{N}$.
\end{corollary}

\begin{proof}
If $A$ is semiprojective, each $A_{\leq n}$ is semiprojective with respect to $n$-subhomogeneous $C^*$-algebras. Hence we can apply Lemma \ref{lemma ideal Peano} to the $n$-homogeneous ideal $\ker(A_{\leq n}\rightarrow A_{\leq n-1})$ in $A_{\leq n}$ whose primitive ideal space is homeomorphic to $\prim_n(A)$.
\end{proof}

It is known to the experts that there are no non-trivial $\M_n$-valued fields over one-dimensional spaces and we are indebted to L. Robert for pointing this fact out to us. Since we couldn't find a proof in the literature, we include one here.

\begin{lemma}\label{lemma trivial bundles}
Let $E$ be a locally trivial field of $C^*$-algebras over a separable, metrizable, locally compact Hausdorff space $X$ with fiber $\M_N$ and $\Gamma_0(E)$ the corresponding section algebra. If $\dim(X)\leq 1$, then $\Gamma_0(E)$ is $\C_0(X)$-isomorphic to $\C_0(X,\M_N)$.
\end{lemma}

\begin{proof}
First assume that $X$ is compact. One-dimensionality of $X$ implies that that the Dixmier-Douady invariant $\delta\in \check{H}^3(X,\mathbb{Z})$ corresponding to $\Gamma_0(E)$ vanishes. Therefore $\Gamma_0(E)$ is stably $\C(X)$-isomorphic to $\C(X,\M_N)$ by Dixmier-Douady classification (see e.g. \cite[Corollary 5.56]{RW98}). Let $\psi\colon\Gamma(E)\otimes\mathcal{K}\rightarrow \C(X,\M_N)\otimes\mathcal{K}$ be such an isomorphism and note that $\Gamma(E)\cong \her(\psi(1_{\Gamma(E)}\otimes e))$ via $\psi$ with $e$ a minimal projection in $\mathcal{K}$. Equivalence of projections over one-dimensional spaces is completely determined by their rank by \cite[Proposition 4.2]{Phi07}. Since $\psi(1_{\Gamma(E)}\otimes e)$ and $1_{\C(X,\M_N)}\otimes e$ share the same rank $N$ everywhere we therefore find $v\in \C(X,\M_N)\otimes\mathcal{K}$ with $v^*v=\psi(1_{\Gamma(E)}\otimes e)$ and $vv^*=1_{\C(X,\M_N)}\otimes e$. But then $\Ad(v)$ gives a $C(X)$-isomorphism from $\her(\psi(1_{\Gamma(E)}\otimes e))$ onto $\her(1_{\C(X,\M_N)}\otimes e)=\C(X,\M_N)$.

Now consider the case of non-compact $X$. Since $X$ is $\sigma$-compact, it clearly suffices to prove the following: Given compact subsets $X_1\subseteq X_2$ of $X$ and a $\C(X_1)$-isomorphism $\varphi_1\colon \Gamma(E_{|X_1})\rightarrow \C(X_1,\M_N)$ there exists a $\C(X_2)$-isomorphism $\varphi_2\colon\Gamma(E_{|X_2})\rightarrow \C(X_2,\M_N)$ extending $\varphi_1$. By the first part of the proof there is a $\C(X_2)$-isomorphism $\psi_2\colon\Gamma(E_{|X_2})\rightarrow \C(X_2,\M_N)$. One-dimensionality of $X_1$ implies $\check{H}^2(X_1,\mathbb{Z})=0$, which means that every $\C(X_1)$-automorphism of $\C(X_1,\M_N)$ is inner by \cite[Theorem 5.42]{RW98}. In particular, $\varphi_1\circ(\psi^{-1}_2)_{|X_1}$ is of the form $\Ad(u)$ for some unitary $u\in \C(X_1,\M_N)$. It remains to extend $u$ to a unitary in $\C(X_2,\M_N)$. This, however, follows from one-dimensionality of $X$ and \cite[Theorem VI.4]{HW48}.
\end{proof}

We are now able to extend the results of \cite{ST12} to the case of general $N$-homogeneous $C^*$-algebras:

\begin{theorem}\label{thm homogeneous case}
Let $A$ be a $N$-homogeneous $C^*$-algebra. The following are equivalent:
\begin{enumerate} 
  \item $A$ is (semi)projective.
  \item $A\cong\C_0(\prim(A),\M_N)$ and $\alpha\prim(A)$ is an A(N)R-space of dimension at most 1.
\end{enumerate}
\end{theorem}

\begin{proof}
By Lemma \ref{lemma ideal Peano} and Lemma \ref{lemma trivial bundles}, we know that $(1)$ implies $A\cong\C_0(\prim(A),\M_N)$. The remaining implications are given by Theorem \ref{thm comm case}.
\end{proof}

\section{Constructions for semiprojective $C^*$-algebras}\label{section semiprojectivity}
Unfortunately, the class of semiprojective $C^*$-algebras lacks good permanence properties. In fact, semiprojectivity is not preserved by most $C^*$-algebraic standard constructions and the list of positive permanence results, most of which can be found in \cite{Lor97}, is surprisingly short. Here, we extend this list by a few new results.

\subsection{Extending lifting problems}

In this section, we introduce a technique to extend lifting problems from ideals to larger $C^*$-algebras. This technique can be used to show that in many situations lifting properties of a $C^*$-algebra pass to its ideals.

\begin{lemma}\label{lemma extended lifting}
Given a surjective inductive system of short exact sequences 
\[\xymatrix{
  0\ar[r] & C_n \ar[r]^{\iota_n} \ar@{->>}[d]_{\pi_n^{n+1}}& D_n \ar[r]^{\varrho_n} \ar@{->>}[d]_{\overline{\pi}_n^{n+1}}& E_n \ar[r] \ar@{->>}[d]_{\overline{\overline{\pi}}_n^{n+1}}& 0 \\
  0 \ar[r] & C_{n+1} \ar[r]^{\iota_{n+1}} & D_{n+1} \ar[r]^{\varrho_{n+1}} & E_{n+1} \ar[r] & 0
}\]
and a commutative diagram of extensions
\[\xymatrix{
  0 \ar[r] & \varinjlim C_n \ar[r]^{\iota_\infty} & \varinjlim D_n \ar[r]^{\varrho_\infty} & \varinjlim E_n \ar[r] & 0 \\ 
  0 \ar[r] & I \ar[r]^i \ar[u]^\varphi & A \ar[r]^p \ar[u]^{\overline{\varphi}} & B \ar[r] \ar[u]^{\overline{\overline{\varphi}}} & 0
}\]
the following holds: If both $A$ and $B$ are semiprojective, then $\varphi$ lifts to $C_n$ for some $n$. If both $A$ and $B$ are projective, then $\varphi$ lifts to $C_1$.
\end{lemma}

\begin{proof}
First observe that we may assume the $^*$-homomorphism $\overline{\overline{\varphi}}$ to be injective since otherwise we simply pass to the system of extensions 
\[\xymatrix{0\ar[r]&C_n\ar[r]^(.4){\iota_n}&D_n\oplus B\ar[r]^{\varrho_n\oplus\id}&E_n\oplus B\ar[r]&0}\]
and replace $\overline{\varphi}$ by $\overline{\varphi}\oplus p$ and $\overline{\overline{\varphi}}$ by $\overline{\overline{\varphi}}\oplus\id$. Using semiprojectivity of $B$, we can find a partial lift $\psi\colon B\rightarrow E_{n_0}$ of $\overline{\overline{\varphi}}$ for some $n_0$, i.e. $\overline{\overline{\pi}}_{n_0}^\infty\circ\psi=\overline{\overline{\varphi}}$. Now consider the $C^*$-subalgebras 
\[D_n':=\varrho_n^{-1}((\overline{\overline{\pi}}_{n_0}^n\circ\psi)(B))\subseteq D_n\]
and observe that the restriction of $\overline{\pi}_n^{n+1}$ to $D_n'$ surjects onto $D_{n+1}'$. We also find that the direct limit $\varinjlim D_n'=\overline{\pi}_{n_0}^\infty(D'_{n_0})$ of this new system contains $\overline{\varphi}(A)$. Hence semiprojectivity of $A$ allows us to lift $\overline{\varphi}$ (regarded as a map to $\varinjlim D_n'$) to $D_n'$ for some $n\geq n_0$. Let $\sigma\colon A\rightarrow D_n'$ be a suitable partial lift, i.e. $\overline{\pi}_n^{\infty}\circ\sigma=\overline{\varphi}$, then the restriction of $\sigma$ to the ideal $I$ will be a solution to the original lifing problem for $\varphi$: The only thing we need to check is that the image of $I$ under $\sigma$ is in fact contained in $C_n$. But we know that $\overline{\overline{\pi}}_n^\infty$ is injective on $(\varrho_n\circ\sigma)(A)\subseteq (\overline{\overline{\pi}}_{n_0}^n\circ\psi)(B)$ since $\overline{\overline{\varphi}}=\overline{\overline{\pi}}_n^\infty\circ(\overline{\overline{\pi}}_{n_0}^n\circ\psi)$ was assumed to be injective. Hence the identity 
\[(\overline{\overline{\pi}}_n^\infty\circ\varrho_n\circ\sigma)(i(I))=(\varrho_\infty\circ\overline{\pi}_n^\infty\circ\sigma)(i(I))=(\varrho_\infty\circ\overline{\varphi})(i(I))=(\varrho_\infty\circ\iota_\infty)(\varphi(I))=0\]
confirms that $\sigma(i(I))\subseteq i_n(C_n)$ holds.
\end{proof}

Now assume that we are given an inductive system 
\[\xymatrix{
  \cdots \ar@{->>}[r] & C_n \ar@{->>}[r]^{\pi_n^{n+1}} & C_{n+1} \ar@{->>}[r] & \cdots
}\] 
of separable $C^*$-algebras with surjective connecting homomorphisms. Then each connecting map $\pi_n^{n+1}$ canonically extends to a surjective $^*$-homomorphism $\overline{\pi}_n^{n+1}$ on the level of multiplier $C^*$-algebras (\cite[Theorem 2.3.9]{WO93}), i.e., we automatically obtain a surjective inductive system of extensions 
\[\xymatrix{
  0\ar[r] & C_n \ar[r] \ar@{->>}[d]_{\pi_n^{n+1}}& \mathcal{M}(C_n) \ar[r] \ar@{->>}[d]_{\overline{\pi}_n^{n+1}}& \mathcal{Q}(C_n) \ar[r] \ar@{->>}[d]_{\overline{\overline{\pi}}_n^{n+1}}& 0 \\
  0 \ar[r] & C_{n+1} \ar[r] & \mathcal{M}(C_{n+1}) \ar[r] & \mathcal{Q}(C_{n+1}) \ar[r] & 0
}.\]
We would like to apply Lemma \ref{lemma extended lifting} to such a system of extensions. However, the reader should be really careful when working with multipliers and direct limits at the same time since these constructions are not completely compatible: Each $\pi_n^\infty\colon C_n\rightarrow\varinjlim C_n$ extends to a $^*$-homomorphism $\mathcal{M}(C_n)\rightarrow\mathcal{M}(\varinjlim C_n)$. The collection of these maps induces a $^*$-homomor\-phism $p_\mathcal{M}\colon \varinjlim\mathcal{M}(C_n)\rightarrow\mathcal{M}(\varinjlim C_n)$ which is always surjective but only in trivial cases injective. The same occurs for the quotients, i.e. for the system of corona algebras $\mathcal{Q}(C_n)$. The situation can be summarized in the commutative diagram with exact rows 
\[\xymatrix{
  0 \ar[r] & C_n \ar@{->>}[d] \ar[r] & \mathcal{M}(C_n) \ar@{->>}[d] \ar[r] & \mathcal{Q}(C_n) \ar@{->>}[d] \ar[r] & 0\\
  0 \ar[r] &\varinjlim C_n \ar[r] \ar@{=}[d] & \varinjlim\mathcal{M}(C_n) \ar@{->>}[d]^{p_\mathcal{M}} \ar[r] & \varinjlim\mathcal{Q}(C_n) \ar@{->>}[d]^{p_\mathcal{Q}} \ar[r] & 0 \\
  0 \ar[r]& \varinjlim C_n \ar[r] & \mathcal{M}(\varinjlim C_n) \ar[r] & \mathcal{Q}(\varinjlim C_n) \ar[r] & 0
}\] 
where the quotient maps $p_\mathcal{M}$ and $p_\mathcal{Q}$ are the obstacles for an application of Lemma \ref{lemma extended lifting}. The following proposition makes these obstacles more precise.

\begin{proposition}\label{prop busby lifting}
Let $A$ and $B$ be semiprojective $C^*$-algebras and
\[\xymatrix{
  0\ar[r]&I\ar[r]&A\ar[r]&B\ar[r]&0&[\tau]
}\]
a short exact sequence with Busby map $\tau\colon B\rightarrow\mathcal{Q}(I)$. Let $I\xrightarrow{\sim}\varinjlim C_n$ be an isomorphism from $I$ to the limit of an inductive system of separable $C^*$-algebras $C_n$ with surjective connecting maps. If the Busby map $\tau$ can be lifted as indicated
\[\xymatrix{
  &\hspace{1.3cm}\varinjlim\mathcal{Q}(C_n) \ar@<0.7cm>@{->>}[d]^{p_\mathcal{Q}}\\
  B \ar[r]^(0.32)\tau \ar@{-->}[ur]&\mathcal{Q}(I)\cong\mathcal{Q}(\varinjlim C_n)
},\]
then $I\rightarrow\varinjlim C_n$ lifts to $C_n$ for some $n$. If both $A$ and $B$ are projective, we can obtain a lift to $C_1$.
\end{proposition}

\begin{proof}
Keeping in mind that $p_{\mathcal{Q}}$ is the Busby map associated to the extension $0\rightarrow\varinjlim C_n\rightarrow\varinjlim\mathcal{M}(C_n)\rightarrow\varinjlim\mathcal{Q}(C_n)\rightarrow0$, the claim follows by combining Theorem 2.2 of \cite{ELP99} with Lemma \ref{lemma extended lifting}.
\end{proof}

One special case, in which the existence of a lift for the Busby map $\tau$ as in Proposition \ref{prop busby lifting} is automatic, is when the quotient $B$ is a projective $C^*$-algebra. Hence we obtain a new proof for the permanence result below which has the advantage that it does not use so-called corona extendability (cf. \cite[Section 12.2]{Lor97}).

\begin{corollary}[\cite{LP98}, Theorem 5.3]\label{cor projective ideal}
Let $0\rightarrow I\rightarrow A\rightarrow B\rightarrow 0$ be short exact. If $A$ is (semi)projective and $B$ is projective, then $I$ is also (semi)projective.
\end{corollary}

Another very specific lifting problem for which Proposition \ref{prop busby lifting} applies, is the following mapping telescope contruction due to Brown.

\begin{lemma}\label{lemma telescope}
Let a sequence $(C_k)_k$ of separable $C^*$-algebras be given and consider the telescope system $(T_n,\varrho_n^{n+1})$ associated to $\bigoplus_{k=0}^\infty C_k=\varinjlim_n \bigoplus_{k=0}^n C_k$, i.e. 
\[T_n=\left\{f\in\C\left([n,\infty],\bigoplus_{k=0}^\infty C_k\right)\colon\;t\leq m\Rightarrow\; f(t)\in\bigoplus_{k=1}^mC_k\right\}\]
with $\varrho_n^{n+1}\colon T_n\rightarrow T_{n+1}$ the (surjective) restriction maps, so that $\varinjlim_n(T_n,\varrho_n^{n+1})\cong\bigoplus_{k=1}^\infty C_k$. Then both canonical quotient maps in the diagram 
\[\xymatrix{
  0 \ar[r] & \varinjlim T_n \ar[r] \ar@{=}[d] & \varinjlim\mathcal{M}(T_n) \ar[r] \ar@{->>}[d]^{p_\mathcal{M}} & \varinjlim\mathcal{Q}(T_n) \ar[r] \ar@{->>}[d]^{p_\mathcal{Q}} & 0 \\
  0 \ar[r] & \varinjlim T_n \ar[r] & \mathcal{M}(\varinjlim T_n) \ar[r] \ar@/^1pc/@{..>}[u]& \mathcal{Q}(\varinjlim T_n) \ar[r] \ar@/^1pc/@{..>}[u] & 0
}\] 
split.
\end{lemma}

\begin{proof}
It suffices to produce a split for $p_\mathcal{M}$ which is the identity on $\varinjlim T_n$. Under the identification $\varinjlim T_n \cong \bigoplus_{k=0}^\infty C_k$ we have $\mathcal{M}(\varinjlim T_n)\cong\prod_{k=0}^\infty\mathcal{M}(C_k)$. One checks that 
\[T_n=\bigoplus_{k=0}^n\C([n,\infty],C_k)\oplus\bigoplus_{k>n}\C_0((k,\infty],C_k)\]
and hence 
\[\prod_{k=0}^\infty\C([\max\{n,k\},\infty],\mathcal{M}(C_k))\subset\mathcal{M}(T_n).\]
It follows that the sum of embeddings as constant functions 
\[\prod_{k=0}^\infty\mathcal{M}(C_k)\rightarrow\prod_{k=0}^\infty\C([\max\{n,k\},\infty],\mathcal{M}(C_k))\subset\mathcal{M}(T_n)\]
defines a split for the quotient map $\varinjlim\mathcal{M}(T_n)\rightarrow\mathcal{M}(\varinjlim T_n)$. It is easily verified that this split is the identity on $\bigoplus_{k=1}^\infty C_k$.
\end{proof}

\begin{remark}[Lifting the Busby map]
Given an extension $0\rightarrow I\rightarrow A\rightarrow B\rightarrow 0$ with both $A$ and $B$ semiprojective, the associated Busby map does in general not lift as in \ref{prop busby lifting}. However, there are a number of interesting situations where it does lift and we therefore can use Propostion \ref{prop busby lifting} to obtain lifting properties for the ideal $I$. One such example is studied in \cite{End14}, where it is (implicitly) shown that the Busby map lifts if $B$ is a finite-dimensional $C^*$-algebra. This observation leads to the fact that semiprojectivity passes to ideals of finite codimension. Further examples will be given in section \ref{section extensions}, where we study Busby maps associated to extensions by homogeneous ideals and identify conditions which guarantee that \ref{prop busby lifting} applies. 
\end{remark}

\subsection{Direct limits which preserve semiprojectivity}\label{section limits}
\subsubsection{Weakly conditionally projective homomorphisms}\label{section wcp}

The following definition characterizes $^*$-homo\-morphisms along which lifting solutions can be extended in an approximate manner. This type of maps is implicitly used in \cite{CD10} and \cite{ST12} in the special case of finitely presented, commutative $C^*$-algebras. 

\begin{definition}\label{def wcp}
A $^*$-homomorphism $\varphi\colon A\rightarrow B$ is weakly conditionally projective if the following holds: Given $\epsilon>0$, a finite subset $F\subset A$ and a commuting square
\[\xymatrix{
  A \ar[d]_\varphi \ar[r]^\psi & D\ar@{->>}[d]^\pi \\ 
  B \ar[r]^\varrho & D/J,
}\]
there exists a $^*$-homomorphism $\psi'\colon B\rightarrow D$ as indicated 
\[\xymatrix{
  A \ar[d]_\varphi \ar[r]^\psi & D\ar@{->>}[d]^\pi \\ 
  B \ar[r]^\varrho \ar@{-->}[ur]^{\psi'}& D/J
}\] 
which satisfies $\pi\circ\psi'=\varrho$ and $\|(\psi'\circ\varphi)(a)-\psi(a)\|<\epsilon$ for all $a\in F$.
\end{definition}

The definition above is a weakening of the notion of conditionally projective morphisms, as introduced in section 5.3 of \cite{ELP98}, where one asks the homomorphism $\psi'$ in \ref{def wcp} to make both triangles of the lower diagram to commute exactly. While conditionally projective morphisms are extremely rare (even when working with projective $C^*$-algebras, cf. the example below), there is a sufficient supply of weakly conditionally projective ones, as we will show in the next section. 

\begin{example}
The inclusion map $\id\oplus\, 0\colon\C_0(0,1]\rightarrow\C_0(0,1]\oplus\C_0(0,1]$ is weakly conditionally projective but not conditionally projective. This can be illustrated by considering the commuting square
\[\xymatrix{
  \C_0(0,1] \ar[r]^\psi \ar[d]_{\id\oplus\, 0} & \C_0(0,3) \ar@{->>}[d]^\pi \\
  \C_0(0,1]\oplus\C_0[2,3) \ar@{=}[r] & \C_0(0,1]\oplus\C_0[2,3)
}\]
where $\pi$ is the restriction map and $\psi$ is given by sending the canonical generator $t$ of $\C_0(0,1]$ to the function
\[(\psi(t))(s)=\begin{cases}s & \text{if}\quad s\leq 1\\ 1-s & \text{if}\quad 1<s\leq 2\\ 0 & \text{if}\quad 2\leq s\end{cases}.\]
It is clear that there is no lift for the generator of $\C_0[2,3)$ which is orthogonal to $\psi(t)$. This shows that the map $\id\oplus\, 0$ is not conditionally projective. However, after replacing $\psi(t)$ with $(\psi(t)-\epsilon)_+$ for any $\epsilon>0$, finding an orthogonal lift for the generator of the second summand is no longer a problem. Using this idea, it will be shown in Proposition \ref{prop wcp examples} that $\id\oplus\, 0$ is in fact weakly conditionally projective, 
\end{example}

If $A$ is a (semi)projective $C^*$-algebra and $\varphi\colon A\rightarrow B$ is weakly conditionally projective, then $B$ is of course also (semi)projective. The next lemma shows that (semi)projectivity is even preserved along a sequence of such maps. Its proof is of an approximate nature and relies on a one-sided approximate intertwining argument (cf. section 2.3 of \cite{Ror02}), a technique borrowed from the Elliott classification program.  

\begin{lemma}\label{lemma limit criterium}
Suppose  $\xymatrix{A_1\ar[r]^{\varphi_1^2} & A_2 \ar[r]^{\varphi_2^3} & A_3 \ar[r]^{\varphi_3^4} & \cdots}$ is an inductive system of separable $C^*$-algebras. If $A_1$ is (semi)projective and all connecting maps $\varphi_n^{n+1}$ are weakly conditionally projective, then the limit $A_\infty=\varinjlim (A_n,\varphi_n^{n+1})$ is also (semi)projective.
\end{lemma}

\begin{proof}
We will only consider the projective case, the statement for the semiprojective case is proven analogously with obvious modifications. Choose finite subsets $F_n\subset A_n$ with $\varphi_n^{n+1}(F_n)\subseteq F_{n+1}$ such that the union $\bigcup_{m=n}^\infty (\varphi_n^m)^{-1}(F_m)$ is dense in $A_n$ for all $n$. Further let $(\epsilon_n)_n$ be a sequence in $\mathbb{R}_{>0}$ with $\sum_{n=1}^\infty\epsilon_n<\infty$.
Now let $\varrho\colon A_\infty\rightarrow D/J$ be a $^*$-homomorphism to some quotient $C^*$-algebra $D/J$. By projectivity of $A_1$ there is a $^*$-homomorphism $s_1\colon A_1\rightarrow D$ with $\pi\circ s_1=\varrho\circ\varphi_1^\infty$. Since the maps $\varphi_n^{n+1}$ are weakly conditionally projective, we can inductively choose $s_{n+1}\colon A_{n+1}\rightarrow D$ with $\pi\circ s_{n+1}=\varrho\circ\varphi_{n+1}^\infty$ such that 
\[\|s_n(a)-(s_{n+1}\circ\varphi_n^{n+1})(a)\|<\epsilon_n\]
holds for all $a\in F_n$. It is now a standard computation (and therefore ommited) to check that $((s_m\circ\varphi_n^m)(x))_m$ is a Cauchy sequence in $D$ for every $x\in F_n$. Furthermore, the induced map $\varphi_n^\infty(x) \mapsto \lim_m(s_m\circ\varphi_n^m)(x)$ extends from the dense subset $\bigcup_n\varphi_n^\infty(F_n)$ to a $^*$-homomorphism $s\colon A_\infty\rightarrow D$.
\[\xymatrix{
  A_n \ar[r]^{s_n} \ar[d]_{\varphi_n^{n+1}} & D \ar@{->>}[ddd]^\pi \\
  A_{n+1} \ar[dd]_{\varphi_{n+1}^\infty} \ar@{-->}[ur]^{s_{n+1}} \\ \\ 
  A_\infty \ar[r]^\varrho \ar@{..>}[uuur]^s & D/J
}\]
Since each $s_n$ lifts $\pi$, the same holds for their pointwise limit, i.e. the limit map $s$ satisfies $\pi\circ s=\varrho$. This shows that $A_\infty$ is projective.
\end{proof}

\subsubsection{Adding non-commutative edges}\label{section adding edges}

In order to make Lemma \ref{lemma limit criterium} a useful tool for constructing semiprojective $C^*$-algebras, we have to ensure the existence of weakly projective $^*$-homomorphisms as defined in \ref{def wcp}. The examples we work out in this section arise in special pullback situations where one 'adds a non-commutative edge' to a given $C^*$-algebra $A$. By this we mean that we form the pullback of $A$ and $\C([0,1])\otimes \M_n$ over a $n$-dimensional representation of $A$ and the evaluation map $\ev_0$. In the special case of $A=\C(X)$ being a commutative $C^*$-algebra and $n=1$ this pullback construction already appeared in \cite{CD10} and \cite{ST12} where it indeed corresponds to attaching an egde $[0,1]$ at one point to the space $X$. 
Here we show that the map obtained by extending elements of $A$ as constant functions onto the attached non-commutative edge gives an example of a weakly conditionally projective $^*$-homomorphism. As an application, we observe that the AF-telescopes studied in \cite{LP98} arise from weakly projective $^*$-homomorphisms and hence projectivity of these algebras is a direct consequence of Lemma \ref{lemma limit criterium}.\vspace{.5cm}

Adapting notation from \cite{ELP98}, we set
\[T(\mathbb{C},G)=\{f\in\C_0((0,2],G):\quad t\leq 1\Rightarrow f(t)\in\mathbb{C}\cdot 1_G\},\]
\[S(\mathbb{C},G)=\{f\in\C_0((0,2),G):\quad t\leq 1\Rightarrow f(t)\in\mathbb{C}\cdot 1_G\}\]
for $G$ a unital $C^*$-algebra. We further write
\[T(\mathbb{C},G,F)=\left\{f\in\C_0((0,3],F):\begin{array}{l} t\leq 2\Rightarrow f(t)\in G \\ t\leq 1\Rightarrow f(t)\in\mathbb{C}\cdot 1_G \end{array}\right\}\]
with respect to a fixed inclusion $G\subseteq F$. We have the diagram 
\[\xymatrix{
  T(\mathbb{C},G,F) \ar[r] \ar[d] & \C([2,3],F) \ar[d]^{\ev_2} \\ 
  T(\mathbb{C},G) \ar[r]^{\ev_2} & F
}\]
which is a special case of the pullback situation considered in the next proposition. However, this example is in some sense generic and implementing it into the general situation is an essential part of proving the following.

\begin{proposition}\label{prop wcp examples}
Given a (semi)projective $C^*$-algebra $Q$ and a $^*$-homomorphism $\tau\colon Q\rightarrow \M_n$, the following holds:
\begin{enumerate}
  \item The pullback $P$ over $\tau$ and $\ev_0\colon\C([0,1],\M_n)\rightarrow\M_n$, i.e.
  \[P=\{(q,f)\in Q\oplus\C([0,1],\M_n):\;\tau(q)=f(0)\},\]
  is (semi)projective. 
  \item The canonical split $s\colon Q\rightarrow P$, $q\mapsto (q,\tau(q)\otimes 1_{[0,1]})$ is weakly conditionally projective. 
\end{enumerate}
\end{proposition}

\begin{proof}
(1) Semiprojectivity of the pullback $P$ follows from \cite[Corollary 3.4]{End14}. Since $P$ is homotopy equivalent to $Q$, the projective statement follows from the semiprojective one using \cite[Corollary 5.2]{Bla12}.\\

(2) For technical reasons we identify the attached interval $[0,1]$ with $[2,3]$ and consider the pullback 
\[\xymatrix{
  P \ar@{-->}[r] \ar@{-->}[d] & \C([2,3],\M_n) \ar[d]^{\ev_2} \\
  Q \ar[r]^\tau & \M_n
}\]
with $s\colon Q\rightarrow P$, $q\mapsto(q,\tau(q)\otimes 1_{[2,3]})$ instead. Denote by $G\subseteq\M_n$ the image of $\tau$. According to \cite[Theorem 2.3.3]{ELP98}, we can find a $^*$-homomorphism $\overline{\varphi}\colon T(\mathbb{C},G)\rightarrow Q$ such that
\[\xymatrix{
  0 \ar[r] & \ker(\tau) \ar[r] & Q \ar[r]^(0.6)\tau & G \ar[r] & 0 \\
  0 \ar[r] & S(\mathbb{C},G) \ar[r] \ar[u] & T(\mathbb{C},G) \ar[u]^{\overline{\varphi}} \ar[r]^(0.6){ev_2} & G \ar@{=}[u] \ar[r] & 0
}\] 
commutes and $\overline{\varphi}_{|S(\mathbb{C},G)}$ is a proper $^*$-homomorphism to $\ker(\tau)$ (meaning that the hereditary subalgebra generated by its image is all of $\ker(\tau)$). Using the pullback property of $P$, $\overline{\varphi}$ can be extended to $\varphi\colon T(\mathbb{C},G,\M_n)\rightarrow P$ such that 
\[\xymatrix{
  0 \ar[r] & \C_0((2,3],\M_n) \ar[r] & P \ar[r] & Q \ar[r] \ar@/_/[l]_s& 0 \\ 
  0 \ar[r] & \C_0((2,3],\M_n) \ar@{=}[u] \ar[r] & T(\mathbb{C},G,\M_n) \ar[u]^\varphi \ar[r] & T(\mathbb{C},G) \ar[u]^{\overline{\varphi}} \ar[r] \ar@/_/[l]_{s'}& 0
}\] 
commutes. In particular we have $\varphi\circ s'=s\circ\overline{\varphi}$, where $s'$ is the canonical split which simply extends functions constantly onto $[2,3]$. 

Choose generators $f_1,...,f_l$ of norm 1 for $\C_0((2,3],\M_n)$ and generators $g_1,...,g_k$ of norm 1 for $T(\mathbb{C},G)$. We need the following 'softened' versions of $P$: 
For $\delta>0$ we consider the universal $C^*$-algebra 
\[P_\delta=C^*\left(\left\{f^\delta,q^\delta:f\in \C_0((2,3],\M_n),q\in Q\right\}|\mathcal{R}_{\C_0((2,3],\M_n)}\&\mathcal{R}_Q\&\mathcal{R}_\delta\right)\]
which is generated by copies of $\C_0((2,3],\M_n)$ and $Q$ (here $\mathcal{R}_{\C_0((2,3],\M_n)},\mathcal{R}_Q$ denote all the relations from $\C_0((2,3],\M_n)$ resp. from $Q$) and additional, finitely many relations 
\[\mathcal{R}_\delta=\left\{\|f_i^\delta(\overline{\varphi}(g_j))^\delta-(f_i(g_j(2)\otimes 1_{[2,3]}))^\delta\|\leq\delta\right\}_{\begin{subarray}{l}1\leq i\leq l \\ 1\leq j\leq k \end{subarray}}.\] 
Note that $P=\varinjlim P_\delta$ with respect to the canonical surjections $p_{\delta,\delta'}\colon P_\delta\rightarrow P_{\delta'}$ (for $\delta>\delta'$) and denote the induced maps $P_\delta\rightarrow P,f^\delta\mapsto f,q^\delta\mapsto s(q)$ by $p_{\delta,0}$. Since $P$ is semiprojective by part (1) of this proposition, we can find a partial lift $j_\delta\colon P\rightarrow P_\delta$ for some $\delta>0$, i.e. $p_{\delta,0}\circ j_\delta=\id_P$.

Now let a finite set $F=\{x_1,...,x_m\}\subseteq Q$ and $\epsilon>0$ and be given. Denoting the inclusions $Q\rightarrow P_\delta,q\mapsto q^\delta$ by $s_\delta$, we can (after decreasing $\delta$ if necessary) assume that $\|s_\delta(x_i)-(j_\delta\circ s)(x_i)\|\leq\epsilon$ holds for all $1\leq i\leq m$. Now given any commuting square 
\[\xymatrix{
  Q \ar[d]^s \ar[r]^\psi &D \ar@{->>}[d]^\pi \\ 
  P \ar[r]^\varrho & D/J
}\]
it only remains to construct a $^*$-homomorphism $\psi_\delta\colon P_\delta\rightarrow D$ such that in the diagram 
\[\xymatrix{
  Q \ar[rr]^\psi \ar[dd]_s \ar[dr]^{s_\delta} & & D \ar@{->>}[dd]^\pi \\
  & P_\delta \ar@{-->}[ur]^{\psi_\delta} \ar[dl]_(0.4){p_{\delta,0}}\\ 
  P \ar@/_1pc/[ur]_(0.6){j_\delta} \ar[rr]_\varrho && D/J
}\]
the upper central triangle and the lower right triangle commute.

We consider the following subalgebras of $T(\mathbb{C},G)$ and $S(\mathbb{C},G)$ for any $\eta>0$:
\[\begin{array}{rl}
  T_\eta(\mathbb{C},G) & =\{f\in T(\mathbb{C},G):f\;\text{is constant on}\;(0,\eta]\cup[2-\eta,2]\} \\
  S_\eta(\mathbb{C},G) &=\{f\in S(\mathbb{C},G):f\;\text{is constant (=0) on}\;(0,\eta]\cup[2-\eta,2]\}
\end{array}\]
Since 
\[T(\mathbb{C},G)=\overline{\bigcup\limits_{\eta>0}T_\eta(\mathbb{C},G)}\]
we find $0<\eta<\frac{1}{2}$ and elements $\tilde{g}_j\in T_\eta(\mathbb{C},G)$ with $\tilde{g}_j(2)=g_j(2)$ and $\|g_j-\tilde{g}_j\|<\delta$ for every $1\leq j\leq k$. Let $h\in T(\mathbb{C},G)$ be the scalar-valued function which equals $1_G$ on $[\eta,2-\eta]$, satisfies $h(0)=h(2)=0$ and is linear in between. Consider the hereditary $C^*$-subalgebra $D'=\overline{(1-(\psi\circ\overline{\varphi})(h))D(1-(\psi\circ\overline{\varphi})(h))}$ and define 
\[D'':=(\psi\circ\overline{\varphi})(T_\eta(\mathbb{C},G))+ D'\subseteq D.\]
Then $(\psi\circ\overline{\varphi})(S_\eta(\mathbb{C},G))$ and $D'$ are orthogonal ideals in $D''$ because $h$ is central in $T(\mathbb{C},G)$. We further have $(\varrho\circ\varphi)(\C_0((2,3],\M_n))\subseteq\pi(D')$ and hence obtain a commutative diagram 
\[\scalebox{0.88}{\xymatrix{
  0\ar[r]&(\psi\circ\overline{\varphi})(S_\eta(\mathbb{C},G))\ar[r]\ar@{->>}[d]^\pi&D''\ar[r]\ar@{->>}[d]^\pi&H_D+D'\ar[r]\ar@{-->>}[d]&0\\
  0\ar[r]&(\varrho\circ\varphi\circ s')(S_\eta(\mathbb{C},G))\ar[r]&\pi(D'')\ar[r]&H_{D/J}+\pi(D')\ar[r]&0\\
  0\ar[r]&s'(S_\eta(\mathbb{C},G))\ar[r]\ar[u]^{\varrho\circ\varphi}&s'(T_\eta(\mathbb{C},G))+\C_0((2,3],\M_n)\ar[u]^{\varrho\circ\varphi}\ar[r]&\hat{T}(G,\M_n)\ar[r]\ar@{-->}[u]&0
}}\]
where $H_D$ and $H_{D/J}$ are finite-dimensional $C^*$-algebras given by
\[H_D=(\psi\circ\overline{\varphi})(T_\eta(\mathbb{C},G))/(\psi\circ\overline{\varphi})(S_\eta(\mathbb{C},G)),\]
\[H_{D/J}=(\varrho\circ\varphi\circ s')(T_\eta(\mathbb{C},G))/(\varrho\circ\varphi\circ s')(S_\eta(\mathbb{C},G))\]
and $\hat{T}(G,\M_n)$ denotes what is called a crushed telescope in \cite{ELP98}:
\[\hat{T}(G,\M_n)=\{f\in\C([2,3],\M_n)\colon f(2)\in G\}\]  
By \cite[Proposition 6.1.1]{ELP98}, the embedding $G\rightarrow\hat{T}(G,\M_n)$ as constant functions is a conditionally projective map (in the sense of \cite[Section 5.3]{ELP98}). It is hence possible to extend the map $\xymatrix{G\ar[r]^(0.22)\sim & T_\eta(\mathbb{C},G)/S_\eta(\mathbb{C},G)\ar[r] & H_D\subset H_D+D'}$ to a $^*$-homomorphism $\psi'\colon \hat{T}(G,\M_n)\rightarrow H_D+D'$ such that the diagram with exact rows
\[\xymatrix{
  0\ar[r]&D'\ar[r]\ar@{->>}[d]_\pi&H_D+D'\ar[r]\ar@{->>}[d]&H_D\ar[r]\ar@{->>}[d]\ar@/_/[l]&0\\
  0\ar[r]&\pi(D')\ar[r]&H_{D/J}+\pi(D')\ar[r]|(0.69)\hole&H_{D/J}\ar[r]&0\\
  0\ar[r]&\C_0((2,3],\M_n)\ar[r]\ar[u]^{\varrho\circ\varphi}&\hat{T}(G,\M_n)\ar[r]_{\ev_{2}}\ar@{-->}@/^2pc/[uu]^(0.7){\psi'}|(0.45)\hole|(0.5)\hole|(0.55)\hole \ar[u]&G\ar[u]\ar@/_/[l]\ar[r]\ar@/^2pc/[uu]&0
}\]
commutes. In particular, $\psi'$ restricts to a $^*$-homomorphism $\C_0((2,3],\M_n)\rightarrow D'$ which we will also denote by $\psi'$. But then a diagram chase confirms that 
\[\psi'(f_i)\cdot(\psi\circ\overline{\varphi})(\tilde{g}_j)=\psi'(f_i\cdot (\tilde{g}_j(2)\otimes 1_{[2,3]}))\]
holds for every $i,j$. Finally, define $\psi_\delta\colon P_\delta\rightarrow D$ by
\[\begin{tabular}{ccc}$q^\delta\mapsto\psi(q)$&and&$f_i^\delta\mapsto\psi'(f_i)$.\end{tabular}\]
It needs to be checked that $\psi_\delta$ is well-defined, i.e. that the elements $\psi_\delta(f_i^\delta)$ and $\psi_\delta(\overline{\varphi}(g_j)^\delta)$ satisfy the relations $\mathcal{R}_\delta$:
\[\begin{array}{rl}
  & \|\psi_\delta(f_i^\delta)\psi_\delta(\overline{\varphi}(g_j)^\delta)-\psi_\delta((f_i(g_j(2)\otimes 1_{[2,3]}))^\delta)\|\\
  =& \|\psi'(f_i)((\psi\circ\overline{\varphi})(g_j))-\psi'(f_i(g_j(2)\otimes 1_{[2,3]}))\| \\
  \leq & \|\psi'(f_i)(\psi\circ\overline{\varphi})(\tilde{g}_j)-\psi'(f_i(\tilde{g}_j(2)\otimes 1_{[2,3]}))\|+\|f_i\|\cdot\|g_j-\tilde{g}_j\|<\delta
\end{array}\]
Since we also have $\psi_\delta\circ s_\delta=\psi$ and $\pi\circ\psi_\delta=\varrho\circ p_{\delta,0}$, the proof is hereby complete.
\end{proof}

One example, where pullbacks as in \ref{prop wcp examples} show up, is the class of so-called AF-telescopes defined by Loring and Pedersen:

\begin{definition}[\cite{LP98}]\label{def af-telescope}
Let $A=\overline{\bigcup A_n}$ be the inductive limit of an increasing union of finite-dimensional $C^*$-algebras $A_n$. We define the AF-telescope associated to this AF-system as 
\[T(A)=\{f\in\C_0((0,\infty],A):\;t\leq n\Rightarrow f(t)\in A_n\}.\]
\end{definition}

We have an obvious limit structure for $T(A)=\varinjlim T(A_k)$ over the finite telescopes 
\[T(A_k)=\{f\in\C_0((0,k],A_k)):\;t\leq n\Rightarrow f(t)\in A_n\}.\]
Now the embedding of $T(A_k)$ into $T(A_{k+1})$ is given by extending the elements of $T(A_k)$ constantly onto the attached interval $[k,k+1]$. This is nothing but a finite composition of maps as in part (2) of \ref{prop wcp examples}. Hence the connecting maps in the system of finite telescopes are weakly conditionally projective and using Lemma \ref{lemma limit criterium} we recover \cite[Theorem 7.2]{LP98}:

\begin{corollary}\label{cor af-telescopes}
All $AF$-telescopes are projective.
\end{corollary}

In contrast to the original proof we didn't have to work out any description of the telescopes by generators and relations. Such a description would have to encode the structure of each $A_n$ as well as the inlusions $A_n\subset A_{n+1}$ (i.e., the Bratteli-diagram of the system). Showing that such an infinite set of generators and relations gives rise to a projective $C^*$-algebra is possible but complicated. Instead we showed that these algebras are build up from the projective $C^*$-algebra $T(A_0)$=0 using operations which preserve projectivity.

\section{Extensions by homogeneous $C^*$-algebras}\label{section extensions}
In this section we study extensions by (trivially) homogeneous $C^*$-algebras, i.e. extensions of the form
\[\xymatrix{0 \ar[r] & \C_0(X,\M_N) \ar[r] & A \ar[r] & B \ar[r] & 0.}\]
Our final goal is to understand the behavior of semiprojectivity along such extensions, and we will eventually achieve this in Theorem \ref{thm 2 out of 3}.

\subsection{Associated retract maps}\label{section retract maps}
Identifying $X$ with an open subset of $\prim(A)$, we make the following definition of an associated retract map. This map will play a key role in our study of extensions.

\begin{definition}\label{def R}
Let $X$ be locally compact space with connected components $(X_i)_{i\in I}$ and 
\[\xymatrix{0 \ar[r] & \C_0(X,\M_N) \ar[r] & A \ar[r] & B \ar[r] & 0}\]
a short exact sequence of $C^*$-algebras. We define the (set-valued) retract map $R$ associated to the extension to be the map
\[R\colon\prim(A)\rightarrow 2^{\prim(B)}\]
given by
\[R(z)=\begin{cases}
  \;z & \text{if}\;z\in\prim(B), \\
  \partial X_i=\overline{X_i}\backslash X_i & \text{if}\;z\in X_i\subseteq X.
\end{cases}\]
\end{definition}

Note that $R$ defined as above takes indeed values in $2^{\prim(B)}$ because the connected components $X_i$ are always closed in $X$. However, in our cases of interest the components $X_i$ will actually be clopen in $X$ (e.g. if $X$ is locally connected) so that we have a topological decomposition $X=\bigsqcup_i X_i$.

\subsubsection{Regularity properties for set-valued maps}
Let $X,Y$ be sets and $S\colon X \rightarrow 2^{Y}$ a set-valued map. We say that $S$ has pointwise finite image if $S(x)\subseteq Y$ is a finite set for every $x\in X$. If furthermore $X$ and $Y$ are topological spaces, we will use the following notion of semicontinuity for $S$ (cf. \cite[Section 1.4]{AF90}).

\begin{deflem}\label{def lsc set map}
Let $X,Y$ be topological spaces. A set-valued map $S\colon X\rightarrow 2^Y$ is lower semicontinuous if one of the following equivalent conditions holds:
\renewcommand{\labelenumi}{(\roman{enumi})}
\begin{enumerate}
  \item $\{x\in X\colon S(x)\subseteq B\}$ is closed in $X$ for every closed $B\subseteq Y$.
  \item For every neighborhood $N(\overline{y})$ of $\overline{y}\in S(\overline{x})$ there exists a neighborhood $N(\overline{x})$ of $\overline{x}$ 
  with $S(x)\cap N(\overline{y})\neq\emptyset$ for every $x\in N(\overline{x})$.
  \item For every net $(x_\lambda)_{\lambda\in\Lambda}\subset X$ with $x_\lambda\rightarrow x$ and every $y\in S(x)$ there exists a net $(y_\mu)_{\mu\in M}\subset\{S(x_\lambda)\colon\lambda\in\Lambda\}$ 
  such that $y_\mu\rightarrow y$.
\end{enumerate}
\end{deflem}

\begin{proof}
$(i)\Rightarrow (ii)$: Let $N(\overline{y})$ be an open neighborhood of $\overline{y}\in S(\overline{x})$. Then $\{x\in X: S(x)\subset Y\backslash N(\overline{y})\}$ is closed and does not contain $\overline{x}$. 
Hence we find an open neighborhood $N(\overline{x})$ of $\overline{x}$ in $X\backslash\{x\in X: S(x)\subset Y\backslash N(\overline{y})\}=\{x\in X:S(x)\cap N(\overline{y})\neq\emptyset\}$.

$(ii)\Rightarrow (iii)$: Denote by $\mathcal{N}$ the family of neighborhoods of $y$ ordered by reversed inclusion. Set $M=\{(\lambda, N)\in\Lambda\times\mathcal{N}\colon S(x_{\lambda'})\cap N\neq\emptyset\;\forall\;\lambda'\geq\lambda\}$, then by assumption $M$ is nonempty and directed with respect to the partial order $(\lambda_1,N_1)\leq(\lambda_2,N_2)$ iff $\lambda_1\leq\lambda_2$ and $N_2\subseteq N_1$. Now pick a $y_{(\lambda,N)}\in S(x_\lambda)\cap N$ for each $(\lambda,N)\in M$, then $(y_\mu)_{\mu\in M}$ constitutes a suitable net converging to $y$.

$(iii)\Rightarrow (i)$: Let a closed set $B\subseteq Y$ and $(x_\lambda)_{\lambda\in\Lambda}\subset\{x\in X:S(x)\subseteq B\}$ with $x_\lambda\rightarrow\overline{x}$ be given. Then for any $\overline{y}\in S(\overline{x})$ we find a net $y_\mu\rightarrow\overline{y}$ with $(y_\mu)\subset\{S(x_\lambda):\lambda\in\Lambda\}\subset B$. Since $B$ is closed we have $\overline{y}\in B$ showing that $S(\overline{x})\subset B$.
\end{proof}

\begin{remark}
An ordinary (i.e. a single-valued) map is evidently lower semicontinuous in the sense above if and only if it is continuous. If both spaces $X$ and $Y$ are first countable, we may use sequences instead of nets in condition $(iii)$.
\end{remark}

Examples of set-valued maps that are lower semicontinuous in the sense above arise from split extensions by homogeneous $C^*$-algebras as follows.

\begin{example}\label{ex split}
Let a split-exact sequence of separable $C^*$-algebras 
\[\xymatrix{
  0 \ar[r]&\C_0(X,\M_n)\ar[r]&A \ar[r]_\pi&B\ar[r]\ar@/_1pc/[l]_s&0
}\] 
be given and consider the set-valued map $R_s\colon \prim(A)\rightarrow 2^{\prim(B)}$ given by 
\[R_s(z)=\begin{cases} \hspace{1.2cm} z&\text{if}\;z\in \prim(B) \\ \left\{[\pi_{z,1}],...,[\pi_{z,r(z)}]\right\}& \text{if}\;z\in X \end{cases}\]
where $\pi_{z,1}\oplus ...\oplus\pi_{z,r(z)}$ is the decomposition of $B\xrightarrow{s} A\rightarrow\C_b(X,\M_n)\xrightarrow{\ev_z}\M_n$ into irreducible summands. Then $R_s$ is lower semicontinuous in the sense of \ref{def lsc set map}.
\end{example}

\begin{proof}
We verify condition $(ii)$ of \ref{def lsc set map}: 
Let $z_n\rightarrow \overline{z}$ in $\prim(A)$ and a neighborhood $N(\overline{y})$ of $\overline{y}\in R_s(\overline{z})$ in $\prim(B)$ be given. By Lemma \ref{lemma top prim} we may assume that $N(\overline{y})$ is of the form $\{z\in\prim(B)\colon\check{b}(z)>1/2\}$ for some $b\in B$. By definition of $R_s$, we find $\check{s(b)}(z)=\max_{y\in R_s(z)} \check{b}(y)$ for all $z\in\prim(A)$. Hence $N(\overline{z})=\{z\in\prim(A)\colon\check{s(b)}>1/2\}$ constitutes a neighborhood of $\overline{z}$ in $\prim(A)$ which satisfies \ref{def lsc set map} $(ii)$.
\end{proof}

Note that the retract map $R_s$ in \ref{ex split} highly depends on the choice of splitting $s$ while the retract map $R$ from \ref{def R} is associated to the underlying extension in a natural way. It is the goal of section \ref{section lifting busby} to find a splitting $s$ such that $R=R_s$ holds. This is, however, not always possible. It can even happen that the underlying extension splits while $R$ is not of the form $R_s$ for any splitting $s$ (cf. remark \ref{remark comm retract}). Under suitable conditions, we will at least be able to arrange $R=R_s$ outside of a compact set $K\subset X$, i.e. we can find a (not necessarily multiplicative) splitting map $s$ such that $B\xrightarrow{s} A\rightarrow\C_b(X,\M_n)$ is multiplicative on $X\backslash K$ so that $R_s(x)$ is still well-defined and coincides with $R(x)$ for all $x\in\prim(A)\backslash K$.

\subsubsection{Lifting the Busby map}\label{section lifting busby}
In this section we identify conditions on an extension
\[\xymatrix{
    0 \ar[r] & \C_0(X,\M_N) \ar[r] & A \ar[r] & B \ar[r] \ar@{..>}@/_1pc/[l]^? & 0 & [\tau]
}\]
which allow us to contruct a splitting $s\colon B\rightarrow A$. This is evidently the same as asking for a lift of the corresponding Busby map $\tau$ as indicated on the left of the commutative diagram
\[\xymatrix{
  & \C(\beta X,\M_N) \ar@{->>}[d]^\varrho \ar@{=}[r] & \prod\limits_{i\in I}\C(\beta X_i,\M_N) \ar@{->>}[d] \\ 
  B \ar[r]^(0.3)\tau \ar@/^1pc/@{..>}[ur]^s \ar@/_2pc/[rr]^{\oplus\tau_i} & \C(\chi(X),\M_N) \ar@{->>}[r] & \prod\limits_{i\in I}\C(\chi(X_i),\M_N).
}\]
We will produce a suitable lift of $\tau$ in two steps:
\begin{enumerate}
 \item For every component $X_i$ of $X$, we trivialize the map $\tau_i\colon B\rightarrow\C(\chi(X_i),\M_N)$, i.e. we conjugate it to a constant map, so that it can be lifted to $\C(\beta X_i,\M_N)$. This step requires the associated retract map $R$ from \ref{def R} to have pointwise finite image and the spaces $\chi(X_i)$ to be connected and low-dimensional.
 \item We extend the collection of lifts for the $\tau_i$'s to a lift for $\tau$. Here we need the associated retract map $R$ to be lower semicontinuous.
\end{enumerate}

In many cases of interest, the spaces $\chi(X_i)$ will not be connected, so that we have to modify the first step of the lifting process. This results in the fact that we cannot find a (multiplicative) split $s$ in general. Instead we will settle for a lift $s$ of $\tau$ with slightly weaker multiplicative properties.

First we give the connection between the retract map $R$ and the Busby map $\tau$ of the extension.

\begin{lemma}\label{lemma Busby vs boundary}
Let a short exact sequence 
\[\xymatrix{
  0 \ar[r] & I \ar[r] & A \ar[r]^\pi & B \ar[r] & 0
}\]
with Busby map $\tau\colon B\rightarrow \mathcal{Q}(I)$ be given. Identifying $\prim(I)$ with the open subset $\{J|I\nsubseteq J\}$ of $\prim(A)$, the following statements hold:
\renewcommand{\labelenumi}{(\roman{enumi})}
\begin{enumerate}
  \item $J\in\partial\prim(I)\Leftrightarrow I+I^\bot\subseteq J$ for every $J\in\prim(A)$,
  \item $\partial\prim(I)=\prim(\tau(B))$.
\end{enumerate}
If in addition $I$ is subhomogeneous, we further have
\begin{enumerate}\setcounter{enumi}{2}
  \item $\left|\partial\prim(I)\right|<\infty\Leftrightarrow\dim(\tau(B))<\infty$.
\end{enumerate}
\end{lemma}

\begin{proof}
For (i) it suffices to check that $\prim(I^\bot)=\prim(A)\backslash\overline{\prim(I)}$ where $I^\bot$ denotes the annihilator of $I$ in $A$. But this follows directly from the definition of the Jacobson topology on $\prim(A)$:
\[\begin{array}{rcl} 
  J\notin\overline{\prim(I)} & \Leftrightarrow & \bigcap\limits_{K\in \prim(I)}K\nsubseteq J \\
  & \Leftrightarrow & \exists x\in A\colon x\notin J\;\text{while}\;\left\|x\right\|_K=0\;\forall\; K\in\prim(I) \\ 
  & \Leftrightarrow & \exists x\in I^\bot\colon x\notin J \\ 
  & \Leftrightarrow & I^\bot\nsubseteq J \\ 
  & \Leftrightarrow & J\in\prim(I^\bot).
\end{array}\]
Now (i) and $\ker(\tau\circ\pi)=I+I^\bot$ imply
\[\begin{array}{rl} 
  \prim(\tau(B))= &\prim((\tau\circ\pi)(A)) \\
  = & \{J\in\prim(A)\colon\ker(\tau\circ\pi)\subseteq J\} \\ 
  = &\{J\in\prim(A)\colon I+I^\bot\subseteq J\} \\ 
  = &\partial\prim(I).
\end{array}\]
For the last statement note that if all irreducible representations of $I$ have dimension at most $n$, the same holds for all irreducible representations $\pi$ of $A$ with $\ker(\pi)$ contained in $\overline{\prim(I)}$. So by the correspondence described in (ii), irreducible representations of $\tau(B)$ are also at most $n$-dimensional. Hence, in this case, finitenesss of $\partial\prim(I)$ is equivalent to finite-dimensionality of $\tau(B)$.
\end{proof}

For technical reasons we would prefer to work with unital extensions. However, it is not clear whether unitization preserves the regularity of $R$, i.e. whether the retract map associated to a unitized extension $0\rightarrow \C_0(X,\M_N)\rightarrow A^+ \rightarrow B^+ \rightarrow 0$ is lower semicontinuous provided that the retract map associated to the original extension is. As the next lemma shows, this is true and holds in fact for more general extensions.

\begin{lemma}\label{lemma finite extension}
Let a locally compact space $X$ with clopen connected components and a commutative diagram
\[\xymatrix{
  && 0 \ar[d] & 0 \ar[d] \\
  0 \ar[r] & \C_0(X,\M_N) \ar[r] \ar@{=}[d] & A \ar[r] \ar[d] & B \ar[r] \ar[d] & 0 \\
  0 \ar[r] & \C_0(X,\M_N) \ar[r] & C \ar[r] \ar[d]^\pi & D \ar[r] \ar[d] & 0 \\
  && F \ar@{=}[r] \ar[d] & F \ar[d] \\
  && 0 & 0
}\]
of short exact sequences of separable $C^*$-algebras be given. Let $R\colon \prim(A)\rightarrow 2^{\prim(B)}$ (resp. $S\colon \prim(C)\rightarrow 2^{\prim(D)}$) be the set-valued retract map associated to the upper (resp. the lower) horizontal sequence as in \ref{def R}. If the quotient $F$ is a finite-dimensional $C^*$-algebra, then the following holds:
\begin{enumerate}
 \item If $R$ has pointwise finite image, then so does $S$.
 \item If $R$ is lower semicontinuous, then so is $S$.
\end{enumerate}
\end{lemma}

\begin{proof}
(1) This is immediate since $\prim(F)$ is a finite set and one easily verfies $S(x)\subseteq R(x)\cup\prim(F)$ for all $x\in X$.

(2) We may assume that $F$ is simple and hence $\pi$ is irreducible. 
Note that $S(J)=R(J)$ for all $J\in\prim(B)\subset\prim(D)$, while for $x\in X$ we have either $S(x)=R(x)$ or $S(x)=R(x)\cup\{[\pi]\}$. Given a closed subset $K\subseteq\prim(D)$, we need to verify that $\{J\in\prim(C)\colon S(J)\subseteq K\}$ is closed in $\prim(C)$. If $[\pi]\in K$, then $\{J\in\prim(C)\colon S(J)\subseteq K\}=\{J\in\prim(A)\colon R(J)\subseteq K\}\cup\{[\pi]\}$ is closed in $\prim(C)$ because $\{J\in\prim(A)\colon R(J)\subseteq K\}$ is closed in $\prim(A)$ by semicontinuity of $R$. Now if $[\pi]\notin K$, the only relevant case to check is a sequence $x_n\subset X$ converging to $\overline{x}\in\prim(D)$ with $S(x_n)\subseteq K$ for all $n$. We then need to show that $S(\overline{x})=\overline{x}\in K$ as well. Decompose $X=\bigcup_{i\in I}X_i$ into its clopen connected components and write $x_n\in X_{i_n}$ for suitable $i_n\in I$. We may assume that $i_n\neq i_m$ for $n\neq m$ since otherwise $\overline{x}\in\partial X_{i_n}=S(x_n)$ for some $n$. Since $R$ is lower semicontinuous, we know that the boundary of $\bigcup_n X_{i_n}$ in $\prim(A)$ is contained in $K\cap\prim(A)$ and hence $\partial\left(\bigcup_n X_{i_n}\right)\subset K\cup\{[\pi]\}$ in $\prim(C)$.

Let $p$ denote the projection of $\C_0(X,\M_N)$ onto $\C_0(\bigcup_n X_{i_n},\M_N)$. This map canonically extends to $\overline{p}$ and $\overline{\overline{p}}$ making the diagram 
\[\xymatrix{
    0 \ar[r] & \C_0(X,\M_k) \ar[r] \ar@{->>}[d]^p & C \ar[r] \ar[d]^{\overline{p}} & D \ar[r]\ar[d]^{\overline{\overline{p}}} & 0\\
    0 \ar[r] & \bigoplus\limits_n\C_0(X_{i_n},\M_N) \ar[r] \ar[d]^\subseteq & \prod\limits_n\C(\beta X_{i_n},\M_N) \ar[r] \ar@{=}[d]& \frac{\prod_n\C(\beta X_{i_n},\M_N)}{\bigoplus_n\C_0(X_{i_n},\M_N)} \ar[r] \ar@{->>}[d]^q&0\\
    0\ar[r] & \prod\limits_n\C_0(X_{i_n},\M_N) \ar[r] & \prod\limits_n\C(\beta X_{i_n},\M_N) \ar[r] & \prod\limits_n\C(\chi(X_{i_n}),\M_N) \ar[r] & 0
}\] 
commute. Using Lemma \ref{lemma Busby vs boundary}, we can indentify the boundary of $\bigcup_n X_{i_n}$ in $\prim(C)$ with $\prim\left(\overline{\overline{p}}(D)\right)$. We already know that $\overline{\overline{p}}$ factors through $D_K\oplus F$, where $D_K$ denotes the quotient corresponding to the closed subset $K$ of $\prim(D)$, and denote the induced map by $\varphi$:
\[\xymatrix{
    D \ar[rr]^{\overline{\overline{p}}} \ar[dr]_{\pi_K\oplus \pi}&& \frac{\prod_n\C(\beta X_{i_n},\M_N)}{\bigoplus_n\C_0(X_{i_n},\M_N)} \\ 
    & D_K\oplus F \ar@{-->}[ur]_\varphi
}\]
We further know that the composition $q\circ\varphi_{|F}\colon F\rightarrow\prod_n\C(\chi(X_{i_n}),\M_N)$ vanishes because $[\pi]\notin\partial X_{i_n}=R(x_n)\subseteq K$ for all $n$. Hence the image of $F$ under $\varphi$ is contained in $\ker(q)=\frac{\prod_n\C_0(X_{i_n},\M_N)}{\bigoplus_n\C_0(X_{i_n},\M_N)}$. But since this $C^*$-algebra is projectionless and $F$ is finite-dimensional, we find $\varphi_{|F}=0$. 
Consequently, $\overline{\overline{p}}$ factors through $D_K$ which means nothing but $\overline{x}\in\partial\left(\bigcup_n X_{i_n}\right)=\prim\left(\overline{\overline{p}}(D)\right)\subseteq K$.
\end{proof}

\begin{lemma}\label{lemma diagonal}
Let $X$ be a connected, compact space of dimension at most 1. For every finite-dimensional $C^*$-algebra $F\subseteq\C(X,\M_n)$ there exists a unitary $u\in\C(X,\M_n)$ such that $uFu^*$ is contained in the constant $\M_n$-valued functions on $X$.
\end{lemma}

\begin{proof}
Since $\dim(X)\leq 1$, equivalence of projections in $\C(X,\M_n)$ is completely determined by their rank (\cite[Proposition 4.2]{Phi07}). In particular, the $C^*$-algebra $\C(X,\M_n)$ has cancellation. Hence \cite[Lemma 7.3.2]{RLL00} shows that the inclusion $F\subset\C(X,\M_n)$ is unitarily equivalent to any constant embedding $\iota\colon F\rightarrow\M_n\subseteq\C(X,\M_n)$ with $\rank(\iota(p))=\rank(p)$ for all minimal projections $p\in F$.
\end{proof}

\begin{lemma}\label{lemma lift unitary}
Let $X$ be a connected, locally compact, metrizable space of dimension at most 1. Then every unitary in $\C(\chi(X),\M_n)$ lifts to a unitary in $\C(\beta X,\M_n)$.
\end{lemma}

\begin{proof}
By \cite[Proposition 4.2]{Phi07}, we have $K_0(\C(\alpha X,\M_n))\cong\Z$ via $[p]\mapsto\rank(p)$. Using the 6-term exact sequence in $K$-theory, this shows that the induced map $K_1(\C(\beta X,\M_n))\rightarrow K_1(\C(\chi(X),\M_n))$ is surjective. Combining this with $K_1$-bijectivity of $\C(\beta X,\M_n)$, which is guaranteed by $\dim(\beta X)=\dim(X)\leq 1$ (\cite[Thm. 9.5]{Nag70}) and \cite[Theorem 4.7]{Phi07}, the claim follows.
\end{proof}

\begin{proposition}\label{prop diagonal Busby}
Let a short exact sequence of separable $C^*$-algebras 
\[\xymatrix{
  0 \ar[r] & \C_0(X,\M_N) \ar[r] & A \ar[r] & B \ar[r] & 0 & [\tau]
}\]
with Busby invariant $\tau$ be given. Assume that $X$ is at most one-dimensional, has clopen connected components $(X_i)_{i\in I}$ and that every corona space $\chi(X_i)$ has only finitely many connected components.
If the associated set-valued retract map $R$ as in \ref{def R} has pointwise finite image, then there is a unitary $U\in\C(\beta X,\M_N)$ such that for each $i\in I$ the composition
\[ B \xrightarrow{\tau} \C(\chi(X),\M_n) \xrightarrow{\Ad(\varrho(U))}\C(\chi(X),\M_N)\rightarrow \C(\chi(X_i),\M_n)\]
has image contained in the locally constant $\M_N$-valued functions on $\chi(X_i)$.
\end{proposition}

\begin{proof}
By Lemma \ref{lemma Busby vs boundary}, the image of each $\tau_i\colon B \xrightarrow{\tau} \C(\chi(X),\M_N) \rightarrow \C(\chi(X_i),\M_N)$ is finite-dimensional. Since by \cite[Thm. 9.5]{Nag70} furthermore $\dim\chi(X_i)\leq \dim\beta X_i=\dim X_i\leq\dim X\leq 1$, we can apply Lemma \ref{lemma diagonal} to obtain unitaries $u_i\in\C(\chi(X_i),\M_N)$ such that $u_i\tau_i(B)u_i^*$ is contained in the locally constant functions on $\chi(X_i)$. These unitaries can be lifted to unitaries $U_i\in\C(\beta X_i,\M_N)$ by Lemma \ref{lemma lift unitary}. Now $U=\oplus_i U_i\in\prod_i\C(\beta X_i,\M_N)=\C(\beta X,\M_N)$ has the desired property.
\end{proof}

\begin{lemma}\label{lemma approx unitary}
Let a short exact sequence of separable $C^*$-algebras 
\[\xymatrix{
  0 \ar[r] & \C_0(X,\M_N) \ar[r] & A \ar[r] & B \ar[r] & 0 & [\tau]
}\] 
with Busby map $\tau$ be given. Assume that $X$ is at most one-dimensional and that the connected components $(X_i)_{i\in I}$ of $X$ are clopen. Further assume that the image of $\tau$ is constant on each $\chi(X_i)\subseteq\chi(X)$. Denote by $\iota\colon A \rightarrow \mathcal{M}(\C_0(X,\M_N))=\C(\beta X,\M_N)$ the canonical map. If the set-valued retract map $R\colon \prim(A)\rightarrow 2^{\prim(B)}$ as defined in \ref{def R} is lower semicontinuous, the following statement holds: 

For every finite set $\mathcal{G}\subset A$, every $\epsilon>0$ and almost every $i\in I$ there exists a unitary $U_i \in\C(\alpha X_i,\M_N)\subset\C(\beta X_i,\M_N)$ such that
\[\left\|(U_i\iota(a)_{|\beta X_i}U_i^*)(x)-\iota(a)(y)\right\|<\epsilon\]
holds for all $a\in\mathcal{G}$, $x\in\beta X_i$ and $y\in\chi(X_i)$.
\end{lemma}

\begin{proof}
We may assume that $A$ is unital by Lemma \ref{lemma finite extension}. Let a finite set $\mathcal{G}\subset A$ and $\epsilon>0$ be given. For each $x\in\beta X$, we write $F_x=\image(\ev_x\circ\iota)\subseteq\M_N$ and
\[T_1(F_x)=\{f\in\C([0,1],F_x)\colon f(0)\in\mathbb{C}\cdot 1_{F_x}\},\]
\[S_1(F_x)=\{f\in\C_0([0,1),F_x)\colon f(0)\in\mathbb{C}\cdot 1_{F_x}\}.\]
Further let $h_\eta\in\C_0[0,1)$ denote the function $t\mapsto \max\{1-t-\eta,0\}$. Using the Urysohn-type result \cite[Theorem 2.3.3]{ELP98}, we find for each $x\in\beta X$ a commuting diagram
\[\xymatrix{
    0 \ar[r] & J_x \ar[r] & A \ar[r]^{\ev_x\circ\iota} & F_x \ar[r] & 0 \\
    0 \ar[r] & S(\mathbb{C},F_x) \ar[u]^{\varphi_x} \ar[r] & T(\mathbb{C},F_x) \ar[u]^{\overline{\varphi}_x} \ar[r]_{\ev_1} & F_x \ar@{=}[u] \ar[r] \ar@/_1pc/@{..>}[l]_{s_x}& 0
}\]
such that $\overline{\varphi}_x$ is unital and $\varphi_x$ is proper. Let $s_x$ be any split for the lower sequence satisfiying $s_x(b)(t)=b$ for $t\geq 1/2$.

Now consider
\[V_{x,\delta}=\{y\in\beta X\colon\quad (\ev_y\circ\iota)(\overline{\varphi}_x(h_\delta))=0\}\]
which is, for $\delta>0$, a closed neighborhood of $x$ in $\beta X$. Note that by assumption $\chi(X_i)\cap V_{x,\delta}\neq\emptyset$ implies $\chi(X_i)\subseteq V_{x,\delta}$. We further claim the following: For almost every $i\in I$ the inclusion $\chi(X_i)\subseteq V_{x,\delta}$ implies $X_i\subset V_{x,2\delta}$. Assume otherwise, then we find pairwise different $i_n\in I$, points $x_n\in X_{i_n}$ and some $1\leq j\leq m$ such that $\chi(X_{i_n})\subseteq V_{x,\delta}$ while $x_n\notin V_{x,2\delta}$ for all $n$. We may assume that $\ev_{x_n}\circ\iota$ converges pointwise to a representation $\pi$. Then 
\[\|\pi(\overline{\varphi}_x(h_\delta))\|=\lim_n\|(\ev_{x_n}\circ\iota\circ\overline{\varphi}_x)(h_\delta)\|\geq\delta\]
since $x_n\neq V_{x,2\delta}$ implies that $\ev_{x_n}\circ\iota\circ\overline{\varphi}$ contains irreducible summands corresponding to evaluations at points $t$ with $t<1-2\delta$. On the other hand, since the retract map $R$ is lower semicontinuous, we find each irreducible summand of $\pi$ to be the limit of irreducible subrepresentations $\varrho_n$ of $\ev_{y_n}\circ\iota$ where $y_n\in\chi(X_{i_n})\subseteq V_{x,\delta}$. Hence
\[\|\pi(\overline{\varphi}_x(h_\delta))\|\leq\liminf_n\|\varrho_n(\overline{\varphi}_x(h_\delta))\|=0\]
by \ref{lemma top prim}, giving a contradiction and thereby proving our claim.

Since $\varphi_x$ is proper, we have $J_x=\overline{\bigcup_{\eta>0}\her(\varphi_x(h_\eta))}$. Hence there exists $1/2>\delta(x)>0$ such that 
\[\inf\left\{\|(a-(\overline{\varphi}_x\circ s_x\circ\ev_x\circ\iota)(a))-b\|\colon b\in\her(\varphi_x(h_{2\delta(x)}))\right\}<\frac{\epsilon}{2}\]
for all $a\in\mathcal{G}$. By compactness of $\chi(X)$, we find $x_1,...,x_m$ such that 
\[\chi(X)\subseteq\bigcup_{j=1}^mV_{x_j,\delta(x_j)}.\]
Then by the claim proved earlier, for almost every $i$ with $\chi(X_i)\subseteq V_{x_j,\delta(x_j)}$ we have a factorization as indicated
\[\xymatrix{
    A \ar[d]_{\ev_{x_j}\circ\iota} \ar@{->>}[drr]^{\pi_j} \ar[rr]^\iota && \prod_i\C(\alpha X_i,\M_N) \ar[r] & \C(\alpha X_i,\M_N), \\
    F_{x_j} \ar[rr]_{\pi_j\circ\overline{\varphi}_{x_j}\circ s_{x_j}} && A/\langle\overline{\varphi}_{x_j}(h_{2\delta(x_j)})\rangle \ar@{-->}[ur]_{\overline{\iota}_i}
}\]
where $\langle\overline{\varphi}_{x_j}(h_{2\delta(x_j)})\rangle$ denotes the ideal generated by $\overline{\varphi}_{x_j}(h_{2\delta(x_j)})$ and $\pi_j$ the corresponding quotient map. By the choice of $\delta(x_j)$, the lower left triangle commutes up to $\epsilon/2$ on the finite set $\mathcal{G}$. Also note that the map $\pi_j\circ\overline{\varphi}_{x_j}\circ s_{x_j}$ is multiplicative.

Finally, by Lemma \ref{lemma diagonal} there exists a unitary $U_i\in\C(\alpha X_i,\M_N)$ such that $\Ad(U_i)\circ(\overline{\iota}_i\circ\pi_j\circ\overline{\varphi}_{x_j}\circ s_{x_j})$ is a constant embedding. Of course, we may arrange $U(\infty)=1$. We then verify
\[\begin{array}{rl}
  & \|(U_i\iota(a)_{|\beta X_i}U_i^*)(x)-\iota(a)(y)\| \\
  \leq & \|(U_i(\overline{\iota}_i\circ\pi_j\circ\overline{\varphi}_{x_j}\circ s_{x_j})((\ev_{x_j}\circ\iota)(a))U_i^*)(x)-(\overline{\iota}_i\circ\pi_j)(a)(y)\| + \frac{\epsilon}{2}\\
  \leq & \|(U_i(\overline{\iota}_i\circ\pi_j\circ\overline{\varphi}_{x_j}\circ s_{x_j})((\ev_{x_j}\circ\iota)(a))U_i^*)(y)-(\overline{\iota}_i\circ\pi_j)(a)(y)\| + \frac{\epsilon}{2}\\
  = & \|(\overline{\iota}_i\circ(\pi_j\circ\overline{\varphi}_{x_j}\circ s_{x_j})\circ(\ev_{x_j}\circ\iota))(a)(y)-(\overline{\iota}_i\circ\pi_j)(a)(y)\| + \frac{\epsilon}{2}\\
  \leq & \epsilon. \\
\end{array}\]
Applying this procedure to each of the finitely many points $x_1,...,x_m$, the statement of the lemma follows.
\end{proof}

Using Lemma \ref{lemma approx unitary} we can now construct a split for our sequence of interest - at least in the case of $\tau(B)$ being constant on each $\chi(X_i)$.

\begin{corollary}\label{cor projective split}
If $0\rightarrow\C_0(X,\M_N)\rightarrow A\rightarrow B\rightarrow 0$ is a short exact sequence of separable $C^*$-algebras such that the assumptions of Lemma \ref{lemma approx unitary} hold, then this sequence splits.
\end{corollary}

\begin{proof}
Let $\tau\colon B\rightarrow\mathcal{Q}(\C_0(X,\M_N))=\C(\chi(X),\M_N)$ denote the Busby map of the sequence. We have the canonical commutative diagram 
\[\xymatrix{
  0 \ar[r] & \bigoplus_i \C_0(X_i,\M_N) \ar[r] & \C(\beta X,\M_N) \ar[r]^\varrho & \C(\chi(X),\M_N) \ar[r] & 0 \\ 
  0 \ar[r] & \C_0(X,\M_N) \ar[r] \ar@{=}[u] & A \ar[r]^\pi \ar[u]^\iota & B \ar[r] \ar[u]^\tau & 0.
}\]
Choose points $y_i\in\chi(X_i)$ for every $i\in I$. Using separability of $A$ and Lemma \ref{lemma approx unitary}, we find a unitary $U\in\prod_i\C(\alpha X_i,\M_N)\subset\prod_i\C(\beta X_i,\M_N)=\C(\beta X,\M_N)$ with
\[U\iota(a)U^*-\prod_i\iota(a)(y_i)\cdot 1_{\alpha X_i}\in\bigoplus_i\C_0(X_i,\M_N)\]
for all $a\in A$ (where $\iota(a)(y_i)\cdot 1_{\alpha X_i}$ denotes the function on $\alpha X_i$ with constant value $\iota(a)(y_i)$). By setting $s(\pi(a))=U^*\left(\prod_i(\iota(a)(y_i)\cdot 1_{\alpha X_i}\right)U$ we find $s\colon B\rightarrow\C(\beta X,\M_N)$ with $(\varrho\circ s)(\pi(a))=(\varrho\circ\iota)(a)=\tau(\pi(a))$ by the formula above. Identifying $A$ with the pullback over $\varrho$ and $\tau$, we can regard $s$ as a map from $B$ to $A$ with $\pi\circ s=\id_B$, i.e. we have constructed a split for the sequence. 
\end{proof}

As the example $0 \rightarrow \C_0(0,1) \rightarrow \C[0,1] \rightarrow \mathbb{C}^2 \rightarrow 0$ shows, we cannot expect extensions by $\C_0(X,M_N)$ to split if the corona space of $X$ (or of one of its components) is not connected. We will now deal with these components and show that one can still obtain a split $s\colon B\rightarrow A$ which, though not multiplicative in general, has still good multiplicative properties.

 \begin{lemma}\label{lemma almost split}
Let $0 \rightarrow \C_0(X,\M_N) \rightarrow A \rightarrow B \rightarrow 0$ be a short exact sequence with Busby map $\tau$. Assume that the corona space $\chi(X)$ of $X$ has only finitely many connected components and that the image of $\tau$ is contained in the locally constant functions on $\chi(X)$. Then there exists a compact set $K\subset X$ and a completely positive split $s\colon B \rightarrow \C(\beta X,\M_N)$ which is multiplicative outside of an open set $U\subset K$.
\end{lemma}

\begin{proof}
Let $\chi(X)=\bigcup_{k=1}^K Y_k$ be the decomposition of the corona space into its connected components. By assumption $\tau$ decomposes as $\oplus_{k=1}^K \tau_k$ with $\image(\tau_k)\subset\M_N\cdot 1_{Y_k}\subseteq\C(\chi(X),\M_N)$. Lift the indicator functions $1_{Y_1},\cdots,1_{Y_K}$ to pairwise orthogonal contractions $h_1,\cdots,h_K$ in $\C(\beta X,\mathbb{C}\cdot 1_{M_N})$ and let $f\colon[0,1]\rightarrow[0,1]$ be the continuous function which equals $1$ on $\left[\frac{1}{2},1\right]$, satisfies $f(0)=0$ and is linear in between. We define a completely positive map $s\colon B\rightarrow\C(\beta X,\M_N)$ by $s(b)(x)=\sum_{k=1}^K \tau_k(b)\cdot f(h_k)(x)$ and check that in the diagram 
\[\xymatrix{
  A \ar[r]^(0.3)\iota \ar[d] & \C(\beta X,\M_N) \ar[d] \\ 
  B \ar[r]^(0.3)\tau \ar@{-->}[ur]^s & \C(\chi(X),\M_N)
}\]
the right triangle commutes. Set $K=\bigcap_{k=1}^K h_k^{-1}([0,\frac{1}{2}])\subset X$, then $s$ is multiplicative outside of the open set $U=\bigcap_{k=1}^K h_k^{-1}([0,\frac{1}{2}))\subset K\subset X$.
\end{proof}

\begin{proposition}\label{prop almost split 2}
Let $0 \rightarrow \C_0(X,\M_N) \rightarrow A \rightarrow B \rightarrow 0$ be a short exact sequence of separable $C^*$-algebras with Busby map $\tau$. Assume that $X$ is at most one-dimensional and has clopen connected components $(X_i)_{i\in I}$. Further assume that each corona space $\chi(X_i)$ has only finitely many connected components and that $\chi(X_i)$ is connected for almost all $i\in I$. If for each $i\in I$ the image of $\tau_i\colon B \xrightarrow{\tau} \C(\chi(X),\M_N)\rightarrow\C(\chi(X_i),\M_N)$ is locally constant on $\chi(X_i)$ and the set-valued retract map $R\colon\prim(A)\rightarrow 2^{\prim(B)}$ as in \ref{def R} is lower semicontinuous, the following holds: There exists a compact set $K\subset X$ and a completely positive split $s\colon B\rightarrow\C(\beta X,\M_N)$ which is multiplicative outside of an open set $U\subset K$.
\end{proposition}

\begin{proof}
Let $I_0\subseteq I$ be a finite set such that $\chi(X_i)$ is connected for every $i\in I_1:=I\backslash I_0$. We may then study the extensions ($*=0$ or 1)
\[\xymatrix{
  0 \ar[r] & \bigoplus\limits_{i\in I}\C_0(X_i,\M_N) \ar[r] \ar@{->>}[d]^{pr_{I_*}} & A \ar[r] \ar@{=}[d]& B \ar[r] \ar[d]^{\varphi_*} & 0 & [\tau] \\ 
  0 \ar[r] & \bigoplus\limits_{i\in I_*}\C_0(X_i,\M_N) \ar[r] & A \ar[r] \ar[d]^{\iota_*}& A/\bigoplus\limits_{i\in I_*}\C_0(X_i,\M_N) \ar[r] \ar[d]^{\tau_*} \ar@{-->}[dl]_{s_*} & 0 & [\tau_*] \\
  0 \ar[r] & \bigoplus\limits_{i\in I_*}\C_0(X_i,\M_N) \ar[r] \ar@{=}[u]& \prod\limits_{i\in I_*}\C(\beta X_i,\M_N) \ar[r] & \frac{\prod_{i\in I_*}\C(\beta X_i,\M_N)}{\bigoplus_{i\in I_*} \C_0(X_i,\M_N)} \ar[r] & 0
}\]
with Busby maps $\tau_*$. Denote the map $B\rightarrow A/\bigoplus_{i\in I_*}\C_0(X_i,\M_N)$ induced by the projection $pr_{I_*}$ by $\varphi_*$. It is now easy to check that for $*=1$ the short exact sequence in the middle row satisfies the assumptions of Lemma \ref{lemma approx unitary} and hence admits a splitting $s_1$ by Corollary \ref{cor projective split}. For $*=0$, we apply Lemma \ref{lemma almost split} to obtain a compact set $K\subset\bigsqcup_{i\in I_0}X_i$ and a completely positive split $s_0$ which is multiplicative outside of an open set $U\subset K\subset\bigsqcup_{i\in I_0}X_i$. Setting $s=s_0\circ\varphi_0\oplus s_1\circ\varphi_1$, we now get a split for the original sequence. In particular, $\varrho\circ s=\tau$ holds due to the commutative diagram
\[\xymatrix{
    0 \ar[r] & \C_0(X,\M_N) \ar[r] \ar@{=}[d] & A \ar[r] \ar[d]^{\iota_0\oplus\iota_1} & B \ar[r] \ar[d]_{\tau_0\oplus\tau_1} \ar@{-->}[dl]_s \ar@/^2pc/[dd]^(0.7)\tau|(0.37)\hole|(0.4)\hole|(0.45)\hole|(0.5)\hole|(0.55)\hole & 0 \\
    0 \ar[r] & \bigoplus\limits_{*=0,1}\bigoplus\limits_{i\in I_*}\C_0(X_i,\M_N) \ar[r] \ar@{=}[d] & \bigoplus\limits_{*=0,1}\prod\limits_{i\in I_*}\C(\beta X_i,\M_N) \ar[r] \ar@{=}[d] & \bigoplus\limits_{*=0,1}\frac{\prod_{i\in I_*}\C(\beta X_i,\M_N)}{\bigoplus_{i\in I_*}\C_0(X_i,\M_N)} \ar[r] \ar@{=}[d] & 0 \\ 
    0 \ar[r] & \bigoplus\limits_{i\in I}\C_0(X_i,\M_N) \ar[r] & \prod\limits_{i\in I}\C(\beta X_i,\M_N) \ar[r]^\varrho & \frac{\prod_{i\in I}\C(\beta X_i,\M_N)}{\bigoplus_{i\in I}\C_0(X_i,\M_N)} \ar[r] & 0.
}\] 
\end{proof}

Summarizing the results of this section, we obtain following.

\begin{theorem}\label{thm semiprojective split}
Let a short exact sequence of separable $C^*$-algebras 
\[\xymatrix{
    0 \ar[r] & \C_0(X,\M_N) \ar[r] & A \ar[r] & B \ar[r] & 0 & [\tau]
}\] 
with Busby map $\tau$ be given. Assume that $X$ satisfies the conditions
\begin{enumerate}
 \item $\dim X\leq 1$,
 \item the connected components $(X_i)_{i\in I}$ of $X$ are clopen,
 \item each $\chi(X_i)$ has finitely many connected components,
 \item almost all $\chi(X_i)$ are connected,
\end{enumerate}
then the following holds: If the associated set-valued retract map $R\colon \prim(A)\rightarrow 2^{\prim(B)}$ given as in \ref{def R} by
\[R(z)=
\begin{cases}
  z\;\text{if}\;z\in \prim(B) \\ 
  \partial X_i=\overline{X_i}\backslash X_i\;\text{if}\;z\in X_i\subseteq X
\end{cases}\]
is lower semicontinuous and has pointwise finite image, then there exists a compact set $K\subset X$ and a completely positive split $s\colon B\rightarrow A$ for the sequence such that the composition 
\[\xymatrix{
  B \ar[r]^s & A \ar[r] & \mathcal{M}(\C_0(X,\M_N))=\C_b(X,\M_N)
}\] 
is multiplicative outside of an open set $U\subset K$.
\end{theorem}

\begin{proof}
Note that we can replace the given extension by any strongly unitarily equivalent one (in sense of \cite[II.8.4.12]{Bla06}) without changing the retract map $R$. Hence, by Proposition \ref{prop diagonal Busby}, we may assume that the image of $\tau$ is locally constant on each $\chi(X_i)$. Now Proposition \ref{prop almost split 2} provides a split $s$ with the desired properties.  
\end{proof}

\subsubsection{Retract maps for semiprojective extensions}
We now verify the regularity properties for the set-valued retract map $R\colon \prim(A)\rightarrow 2^{\prim(B)}$ associated to an extension $0\rightarrow\C_0(X,\M_N)\rightarrow A\rightarrow B\rightarrow 0$ in the case that both the ideal $\C_0(X,\M_N)$ and the extension $A$ are semiprojective $C^*$-algebras.\\

First we need the following definition which is an adaption of \ref{def core} and \ref{def fpm} to the setting of pointed spaces.

\begin{definition}\label{def extended core}
Let $(X,x_0)$ be a pointed one-dimensional Peano continuum and $r\colon X\rightarrow \core(X)$ the first point map onto the core of $X$ as in \ref{def fpm} (where we choose $\core(X)$ to be any point $x\neq x_0$ if $X$ is contractible). Denote the unique arc from $x_0$ to $r(x_0)$ by $[x_0,r(x_0)]$, then we say that 
\[\core(X,x_0):=\core(X)\cup[x_0,r(x_0)]\]
is the core of $(X,x_0)$. It is the smallest subcontinuum of $X$ which contains both $\core(X)$ and the point $x_0$.
\end{definition}

Now let $X$ be a non-compact space with the property that its one-point compactification $\alpha X=X\cup\{\infty\}$ is a one-dimensional ANR-space. We are interested in the structure of the space $X$ at around infinitity (i.e. outside of large compact sets) which is reflected in its corona space $\chi(X)=\beta X\backslash X$. At least some information about $\chi(X)$ can be obtained by studying neighborhoods of the point $\infty$ in $\alpha X$. The following lemma describes some special neighborhoods which relate nicely to the finite graph $\core(\alpha X,\infty)$.

\begin{lemma}\label{lemma 1-ANR at infinity}
Let $X$ be a connected, non-compact space such that its one-point compactification $\alpha X=X\cup\{\infty\}$ is a one-dimensional ANR-space.Fix a geodesic metric $d$ on $\alpha X$, then for any compact set $C\subset X\backslash\{x_0\}$ there exists a closed neighborhood $V$ of $\infty$ with the following properties:
\renewcommand{\labelenumi}{(\roman{enumi})}
\begin{enumerate}
  \item $\{x\in X\colon d(x,\infty)\leq\epsilon\}\subseteq V\subseteq X\backslash C$ for some $\epsilon>0$.
  \item $V\cap \core(\alpha X,\infty)$ is homeomorphic to the space of $K$ many intervals $[0,1]$ glued together at the $0$-endpoints with $K=\order(\infty,\core(\alpha X,\infty))$. 
  The gluing point corresponds to $\infty$ under this identification.
\end{enumerate}
Let $D^{(k)}\subseteq V$ denote the $k$-th copy of $[0,1]$ under the identification described above and let $r$ be the first point map onto $\core(\alpha X,\infty)$. We can further arrange:
\renewcommand{\labelenumi}{(\roman{enumi})}
\begin{enumerate}\setcounter{enumi}{2}
  \item $V=\bigcup_{k=1}^K r^{-1}\left(D^{(k)}\right)$ and $r^{-1}\left(D^{(k\phantom{'})}\right)\cap r^{-1}\left(D^{(k')}\right)=\{\infty\}$ for $k\neq k'$.
  \item The connected components of $V\backslash\{\infty\}$ are given by $V^{(k)}:=r^{-1}\left(D^{(k)}\backslash\{\infty\}\right)$.
  \item Every path in $V$ from $x\in V^{(k)}$ to $x'\in V^{(k')}$ with $k\neq k'$ contains $\infty$.
\end{enumerate}
\end{lemma}

\begin{proof}
We first note that $r^{-1}(\{\infty\})\cap X$ is open. Assume there is $x\in X$ with $r(x)=\infty$ and $d(x,\infty)=r>0$. Then given any $y\in X$ with $d(x,y)< r$ we choose an isometric arc $\alpha\colon [0,d(x,y)]\rightarrow \alpha X$ from $x$ to $y$. Now the arc from $y$ to $\infty$ given by first following $\alpha$ in reverse direction and then going along the unique arc from $x$ to $\infty$ must run through $r(y)$ by \ref{def fpm}. Since every point on the second arc gets mapped to $\infty$ by $r$, we find either $r(y)=\infty$ or there is $0<t<d(x,y)$ such that $\alpha(t)=r(y)\in \core(\alpha X,\infty)$. In the second case, the arc $\alpha_{|[0,t]}$ must run through $r(x)=\infty$ which, using the fact that $\alpha$ was isometric, gives the contradiction $d(x,\infty)<t<d(x,y)<r$. Since $r^{-1}(\{\infty\})$ is also closed, connectedness of $X$ implies in fact $r^{-1}(\{\infty\})=\{\infty\}.$

By definition of $K$ (see section \ref{section 1-ANR}), the closed set  $\{x\in\core(\alpha X,\infty)\colon d(x,\infty)\leq\epsilon\}$ satisfies the description in (ii) for all sufficiently small $\epsilon>0$. We set 
\[V=\{x\in \alpha X\colon d(r(x),\infty)\leq\epsilon\},\]
then $V\cap \core(\alpha X,\infty)=r(V)=\{x\in \core(\alpha X,\infty)\colon d(x,\infty)\leq\epsilon\}$ so that condition (ii) is satisfied. For (i), we observe that $d(x,\infty)\leq\epsilon$ implies $d(r(x),\infty)\leq\epsilon$ since $d$ is geodesic and every arc from $x$ to $\infty$ runs through $r(x)$. Since $\infty\notin r(C)$, we have $\min\{d(r(x),\infty)\colon x\in C\}>0$ and therefore $V\cap C=\emptyset$ for $\epsilon$ sufficiently small. Condition (iii) follows immediately from the definition of $V$. The sets $V^{(k)}$ are connected and open by construction, so that (iv) holds. (v) follows from (iv). 
\end{proof}

We now collect some information about the corona space $\chi(X)$ in the case of connected $X$. These observations are mostly based on the work of Grove and Pedersen in \cite{GP84} and the graph-like structure of one-dimensional ANR-spaces.

\begin{lemma}\label{lemma corona}
Let $X$ be a connected, non-compact space such that its one-point compactification $\alpha X$ is a one-dimensional ANR-space. Then the corona space $\chi(X)$ has covering dimension at most 1 and its number of connected components is given by $K=\order(\infty,\core(\alpha X,\infty))<\infty$. In particular, if $\alpha X$ is a one-dimensional AR-space, then $\chi(X)$ is connected.
\end{lemma}

\begin{proof}
Apply Lemma \ref{lemma 1-ANR at infinity} to $(\alpha X,\infty)$. It is straightforward to check that the map 
\[\C(\chi(X))=\C_b(X)/\C_0(X)\rightarrow\bigoplus_{k=1}^K\C_b(V^{(k)})/\C_0(V^{(k)})=\bigoplus_{k=1}^K\C(\chi(V^{(k)}))\]
is an isomorphism. Therefore we find $\chi(X)=\bigsqcup_{k=1}^K\chi(V^{(k)})$ and it suffices to check that each $\chi(V^{(k)})$ is connected. By Proposition 3.5 of \cite{GP84}, it is now enough to show that each $V^{(k)}$ is connected at infinity. So let a compact set $C_1\subset V^{(k)}$ be given and denote by $r\colon V^{(k)}\cup\{\infty\}\rightarrow D^{(k)}$ the first point map. Using the identification $[0,1]\cong D^{(k)}$ where the point $0$ corresponds to the point $\infty$, we find $t>0$ such that $r(C_1)\subset [t,1]$. But $C_2:=r^{-1}([t,1])$ is easily seen to be compact while $V^{(k)}\backslash C_2=r^{-1}((0,t))$ is pathconnected by definition of $r$. For the dimension statement we note that $\dim(\chi(X))\leq\dim(\beta X)=\dim(X)\leq 1$ by \cite[Theorem 9.5]{Nag70}.
\end{proof}

\begin{remark}\label{remark component}
The assumption that $X$ is connected in \ref{lemma corona} is necessary. If we drop it, the corona space $\chi(X)$ may no longer have finitely many connected components, but the following weaker statement holds: If $\alpha X$ is a one-dimensional ANR-space, so will be $\alpha X_i$ for any connected component $X_i$ of $X$. However, it follows from \ref{thm ward} that all but finitely many components lead to contractible spaces $\alpha X_i$, i.e. to one-dimensional AR-spaces. Since in this case $\core(\alpha X_i,\infty)$ is just an arc $[x,\infty]$ for some $x\in X_i$, we see from Lemma \ref{lemma corona} that $\chi(X_i)$ is connected for almost every component $X_i$ of $X$.
\end{remark}

We will now see that, in the situation described in the beginning of this section, the set-valued retract map $R$ has pointwise finite image, i.e. $|R(z)|<\infty$ for all $z\in \prim(A)$. The cardinality of these sets is in fact uniformly bounded and we give an upper bound which only depends on $N$ and the structure of the finite graph $\core(\alpha X,\infty)$.

\begin{proposition}\label{prop finite boundary}
Let $A$ be a semiprojective $C^*$-algebra containing an ideal of the form $\C_0(X,\M_N)$. If $\alpha X=X\cup\{\infty\}$ is a one-dimensional ANR-space, then every connected component $C$ of $X$ has finite boundary $\partial C=\overline{C}\backslash C$ in $\prim(A)$. More precisely, we find 
\[|\partial C|\leq N\cdot \order(\infty,(\alpha C,\infty))<\infty.\]
\end{proposition}

\begin{proof}
Since $X$ is locally connected, the connected components of $X$ are clopen and $\alpha C$ is again a one-dimensional ANR-space for every component $C$ of $X$. Hence we may assume that $C=X$. Fix a geodesic metric $d$ on $\alpha X=X\cup\{\infty\}$ and let $V$ be a neighborhood of $\infty$ as constructed in Lemma \ref{lemma 1-ANR at infinity}, satisfying $\{x\in\alpha X\colon d(x,\infty)\leq\epsilon\}\subseteq V$ for some $\epsilon>0$. We further choose sequences $(x_n^{(k)})_n\subseteq D^{(k)}\backslash\{\infty\}$ converging to $\infty$ and write $x_\infty^{(1)}=\dots=x_\infty^{(K)}=\infty$. By compactness of the unit ball in $\M_N$ and separability of $A$, we may assume that the representation 
\[\pi^{(k)}\colon A\rightarrow\M_N,\quad a\mapsto\lim_{n\rightarrow\infty}a (x_{n}^{(k)})\]
exists for all $1\leq k\leq K$. Here, $a(x)$ denotes the image of $a\in A$ under the extension of the point evaluation $\ev_x\colon\C_0(X,\M_N)\rightarrow \M_N$ to $A$. For a sequence $(x_n)_n$ in $X\subseteq\prim(A)$ we write $\Lim (x_n)=\{z\in \prim(A)\colon x_n\rightarrow z\}$. Our goal is then to show that there exists a finite set $S\subset\prim(A)$ such that $\Lim (x_n)\subset S$ for every sequence $(x_n)_n\subset X$ with $x_n\rightarrow\infty$ in $\alpha X$. We will show that each $S^{(k)}:=\Lim (x_n^{(k)})$ consists of at most $N$ elements and that $S:=\bigcup_{k=1}^K S^{(k)}$ has the desired property described above. First observe that 
\[S^{(k)}=\left\{\left[\pi_1^{(k)}\right],\dots,\left[\pi_{r(k)}^{(k)}\right]\right\}\]
holds, where $\pi^{(k)}\simeq\pi_1^{(k)}\oplus\cdots\oplus\pi_{r(k)}^{(k)}$ is the decomposition of $\pi^{(k)}$ into irreducible summands. The $\supseteq$-inclusion is immediate, for the other direction assume that $x_n^{(k)}\rightarrow\ker(\varrho)$ for some irreducible representation $\varrho$ with $\varrho\not\simeq\pi_i^{(k)}$ for all $i$. Since all $x_n^{(k)}$ correspond to $N$-dimensional representations, we also have $\dim(\varrho)\leq N$. Therefore all $\pi_i^{(k)}$ and $\varrho$ drop to irreducible representations of the maximal $N$-subhomogeneous quotient $A_{\leq N}$ of $A$ (cf. section \ref{section subhomogeneous}). Because $\prim(A_{\leq N})$ is a $T_1$-space, the finite set $\{[\pi_1^{(k)}],\dots,[\pi_{r(k)}^{(k)}]\}$ is closed and $[\varrho]$ can be separated from it. In terms of \ref{lemma top prim}, this means that there exists $a\in A$ such that $\|\varrho(a)\|>1$ while $\|\pi_i^{(k)}(a)\|\leq 1$ for all $i$. On the other hand, we find
\[\|\varrho(a)\|\leq\liminf_{n\rightarrow\infty}\left\|a(x_n^{(k)})\right\|=\left\|\pi^{(k)}(a)\right\|=\max\limits_{i=1...r(k)}\left\|\pi_i^{(k)}(a)\right\|\leq1,\]
using \ref{lemma top prim} again. Hence $[\varrho]=[\pi_i^{(k)}]$ for some $i$ and in particular $\left|S^{(k)}\right|=r(k)\leq N$ for every $k$.

It now suffices to show that $\Lim (x_n)\subseteq S^{(k)}$ for sequences $(x_n)\subset X$ with $x_n\rightarrow\infty$ such that $r(x_n)\in D^{(k)}$ for some fixed $k$ and all $n$. Let such a sequence $(x_n)_n$ for some fixed $k$ be given and pick $z\in\Lim (x_n)$. In order to show that $z\in S^{(k)}$, we consider the compact spaces
\[Y_n:=\left\{(t,t)|0\leq t\leq\frac{1}{n}\right\}\cup\bigcup_{m\geq n} \left(\left\{\frac{1}{m}\right\}\times\left[0,\frac{1}{m}\right]\right)\subset\mathbb{R}^2.\]
Note that $Y_{n+1}\subset Y_n$ and $\bigcap_n Y_n=(0,0)$. We will now 'glue' $\C(Y_n,\M_N)$ to $A$ in the following way: As before, we may assume that $x_n\rightarrow z$ in $\prim(A)$ and that $\pi_\infty(a)=\lim_n a(x_n)$ exists for every $a\in A$. In particular, we find $z=[\pi_{i,\infty}]$ for some $i$ where $\pi_\infty\simeq\pi_{1,\infty}\oplus\cdots\oplus\pi_{r_\infty,\infty}$ is the decomposition of $\pi_\infty$ into irreducible summands. Let $c$ denote the $C^*$-algebra of convergent $\M_N$-valued sequences, we can then form the pullback $A_n:=A\oplus_c\C(Y_n,\M_N)$ over the two $^*$-homomorphisms
\[\begin{array}{ccc}
  A\longrightarrow c &\text{and}&\C(Y_n,\M_N)\longrightarrow c . \\
  a\mapsto (a(x_n),a(x_{n+1}),a(x_{n+2}),\dots)&& f\mapsto f((\frac{1}{n},0),f(\frac{1}{n+1},0),f(\frac{1}{n+2},0),\dots)
\end{array}\]
These pullbacks form an inductive system in the obvious way. Further note that the connecting maps $A_n\rightarrow A_{n+1}$ are all surjective. The limit $\varinjlim A_n$ can be identified with $A$ via the isomorphism induced by the projections $A_n=A\oplus_c\C(Y_n,\M_N)\rightarrow A$ onto the left summand. Using semiprojectivity of $A$, we can find a partial lift to some finite stage $A_n$ of this inductive system:
\[\xymatrix{
  &&A_n=A\oplus_c\C(Y_n,\M_N) \ar[r] \ar@{->>}[d] & \C(Y_n,\M_N) \\
  \C_0(X,\M_N) \ar[r]^(.6)\subseteq & A \ar[r]^{\cong} \ar@{-->}[ur] & \varinjlim A_n 
}\]
Let $\varphi\colon A\rightarrow\C(Y_n,\M_N)$ be the composition of this lift with the projection $A_n\rightarrow\C(Y_n,\M_N)$ to the right summand. The restriction of $\varphi$ to the ideal $\C_0(X,\M_N)$ then induces a continuous map $\varphi^*\colon Y_n\rightarrow\alpha X$ with $\varphi^*\left(\frac{1}{m},0\right)=x_m$ for all $m\geq n$ and $\varphi^*(0,0)=\infty$. Denote by $h$ the strictly positive element of $\C_0(X,\M_N)$ given by $h(x)=d(x,\infty)\cdot 1_{\M_N}$. After increasing $n$, we may assume that $\|\varphi(h)\|<\epsilon$ holds. For $m\geq n$, we consider the paths $\alpha_m\colon\left[0,\frac{2}{m}\right]\rightarrow Y_n$ given by
\[\alpha_m(t)=\begin{cases} 
  \left(\frac{1}{m},t\right) & \text{if}\;\;0\leq t\leq\frac{1}{m} \\
  \left(\frac{2}{m}-t,\frac{2}{m}-t\right) & \text{if}\;\;\frac{1}{m} \leq t\leq\frac{2}{m}.
\end{cases}\]
Set $t_{\infty,m}=\min\{t\colon\varphi(h)(\alpha_m(t))=0\}$, then $0<t_{\infty,m}\leq\frac{2}{m}$ because of $\|\varphi(h)(\alpha_m(0))\|=\|\varphi(h)(\frac{1}{m},0)\|=\|h(x_m)\|=d(x_m,\infty)>0$ and $\varphi(h)(\alpha_m(\frac{2}{m}))=\varphi(h)(0,0)=h(\infty)=0$. By setting $\beta_m(t)=\varphi^*(\alpha_m(t))$ we obtain paths $\beta_m\colon[0,t_{\infty,m}]\rightarrow\alpha X$ which have the properties
\[\begin{array}{l}
  (1)\;\beta_m(0)=x_m,\\
  (2)\;\beta_m(t)=\infty\;\text{if and only if}\;t=t_{\infty,m},\\
  (3)\;\image(\beta_m)\subseteq V^{(k)}\;\text{for all}\;m,\\
  (4)\;x_l^{(k)}\in\image(\beta_m)\;\text{for fixed}\; m \;\text{and all sufficiently large}\;l. 
\end{array}\]
The first property is clear while the second one follows directly from the definition of $t_{\infty,m}$. In order to verify properties $(3)$ and $(4)$ we have to involve the structure of the neighborhood $V$ and by that the special structure of $\alpha X$ as a one-dimensional ANR-space. From $\|\varphi(h)\|<\epsilon$ we obtain $\image(\beta_m)\subseteq\image(\alpha_m) \subseteq \{x\in\alpha X\colon d(x,\infty)\leq\epsilon\}\subseteq V$, it then follows from (1), (2) and property (v) in Lemma \ref{lemma 1-ANR at infinity} that $\image(\beta_m)\subseteq V^{(k)}$. For (4), observe that $\image(\beta_m)$ contains $r(\image(\beta_m))$ by part (ii) of Lemma \ref{lemma arc to core}, where $r$ is the first-point map $\alpha X\rightarrow\core(\alpha X,\infty)$. Under the identification $D^{(k)}\cong [0,1]$, the connected set $r(\image(\beta_m))$ corresponds to a proper interval containing the $0$-endpoint and hence it contains $x_l^{(k)}$ for almost every $l$.

Now set $\pi_m=\ev_{\beta(t_{\infty,m})}\circ\varphi\colon A\rightarrow\M_N$ and let $\pi_m\simeq\pi_{1,m}\oplus\dots\oplus\pi_{r_m,m}$ be the decomposition into irreducible summands. We claim that the identity 
\[S^{(k)}=\left\{\left[\pi_{1,m}\right],\cdots,\left[\pi_{r_m,m}\right]\right\}\]
holds for all $m$. Involving property (4) for the path $\beta_m$, we find
\[\|\pi_m(a)\|=\lim\limits_{t\nearrow t_{\infty,m}}\left\|(\ev_{\beta(t)}\circ\varphi)(a)\right\|=\lim\limits_{l\rightarrow\infty}\left\|a\left(x_l^{(k)}\right)\right\|=\left\|\pi^{(k)}(a)\right\|\]
for every fixed $m$ and all $a\in A$. Now the same separation argument as in the beginning of the proof shows that the finite-dimensional representations $\pi^{(k)}$ and $\pi_m$ share the same irreducible summands for every $m$. Since $\beta_m(t_{\infty,m})\rightarrow (0,0)$ in $Y_n$, we find $\pi_m=\ev_{\beta(t_{\infty,m})}\circ\varphi\rightarrow \ev_{(0,0)}\circ\varphi=\pi_\infty$ pointwise. Hence by the above identity, $\pi_\infty$ and $\pi^{(k)}$ also share the same irreducible summands. In particular, we find $z\in S^{(k)}$ which finishes the proof.
\end{proof}

Next, we show that in our situation the set-valued retract map $R$ is also lower semicontinuous in the sense of \ref{def lsc set map}.

\begin{proposition}\label{prop R is lsc}
Let $0\rightarrow\C_0(X,\M_N)\rightarrow A\rightarrow B\rightarrow 0$ be a short exact sequence of separable $C^*$-algebras. If $\alpha X$ is a one-dimensional ANR-space and $A$ is semiprojective, then the associated retract map $R$ as in \ref{def R} is lower semicontinuous.
\end{proposition}

\begin{proof}
Let $X=\bigsqcup_{i\in I}X_i$ denote the decomposition of $X$ into its connected components. By separability of $A$ it suffices to verify condition (iii) of Lemma \ref{def lsc set map} for a given sequence $x_n\rightarrow z$ in $\prim(A)$. The case $z\in X$ is trivial since $X$ is locally connected and therefore has open connected components. The critical case is when $x_n\in X$ for all $n$ but $z\in\prim(B)$. 
In this case, we write $x_n\in X_{i_n}$ and we may assume that $\pi_\infty(a)=\lim_n a(x_n)$ is well defined for all $a\in A$. In particular, $z$ corresponds to the kernel of an irreducible summand $\pi_{j,\infty}$ of $\pi_\infty\simeq\pi_{1,\infty}\oplus\cdots\oplus\pi_{r,\infty}$, as we have already seen in the beginning of the proof of Proposition \ref{prop finite boundary}. Using exactly the same construction of 'gluing the space $Y$ to $A$ along the sequence $(x_n)$' as in the proof of \ref{prop finite boundary}, one now shows that 
\[\{[\pi_{1,\infty}],\cdots,[\pi_{r,\infty}]\}\subseteq\overline{\bigcup\limits_n \partial X_{i_n}}.\]
Hence we find $y_n\in\partial X_{i_n}=R(x_n)$ with $y_n\rightarrow [\pi_{j,\infty}]=z$ showing that the retract map $R$ is in fact lower semicontinuous.
\end{proof}

\subsection{Existence of limit structures}\label{section existence limit structures}

Consider an extension of separable $C^*$-algebras 
\[0\rightarrow\C_0(X,\M_N)\rightarrow A\rightarrow B\rightarrow 0\]
where the one-point compactification of $X$ is assumed to be a one-dimensional ANR-space. We know from Theorem \ref{thm ST} that in this case $\alpha X$ comes as a inverse limit of finite graphs over a surprisingly simple system of connecting maps. Here we show that under the right assumptions on the set-valued retract map $R\colon \prim(A)\rightarrow 2^{\prim(B)}$ associated to the sequence above, this limit structure for $\alpha X$ is compatible with the extension of $B$ by $\C_0(X,\M_N)$ in the following sense: We prove the existence of a direct limit structure for $A$ which describes it as the $C^*$-algebra $B$ with a sequence of non-commutative finite graphs (1-NCCW's) attached. The connecting maps of this direct system are obtained from the limit structure for $\alpha X$ and hence can be described in full detail.    

\begin{lemma}\label{lemma limit structure}
Let a short exact sequence of separable $C^*$-algebras $0\rightarrow\C_0(X,\M_N)\rightarrow A\rightarrow B\rightarrow 0$ with Busby map $\tau$ be given. Assume that $\alpha X$ is a one-dimensional ANR-space and that the associated set-valued retract map $R\colon \prim(A)\rightarrow 2^{\prim(B)}$ as in \ref{def R} is lower semicontinuous and has pointwise finite image. Then $A$ is isomorphic to the direct limit $B_\infty$ of an inductive system 
\[\xymatrix{
    B_0 \ar[r]_{s_0^1}& B_1 \ar[r]_{s_1^2} \ar@/_1pc/@{->>}[l]_{r_1^0}& B_2 \ar[r] \ar@/_1pc/@{->>}[l]_{r_2^1} & 
    \cdots \ar[r]_{s_{i-1}^i} \ar@/_1pc/@{->>}[l]& B_i \ar[r] \ar@/_1pc/[rr]_{s_i^\infty} \ar@/_1pc/@{->>}[l]_{r_i^{i-1}}& \cdots \ar[r] & 
    B_\infty=\varinjlim \left(B_i,s_i^{i+1}\right) \ar@/_1pc/@{->>}[ll]_{r_\infty^i}
}\] 
where 
\begin{itemize}
  \item $B_0$ is given as a pullback $B\oplus_F D$ with $D$ a 1-NCCW and $\dim(F)<\infty$. Furthermore, if $\alpha X$ is contractible, we may even arrange $B_0\cong B$.
\end{itemize} 
and 
\begin{itemize}
  \item for every $i\in\mathbb{N}$ there is a representation $\pi_i\colon B_i\rightarrow\M_N$ such that $B_{i+1}$ is defined as the pullback  
  $\xymatrix{
      B_{i+1} \ar@{..>>}[d]^{r_{i+1}^i} \ar@{..>}[r]& \C([0,1],\M_N) \ar@{->>}[d]^{ev_0} \\ 
      B_i \ar@/^1pc/[u]^{s_i^{i+1}} \ar[r]^{\pi_i}& \M_N.
   }$\\ 
The map $s_i^{i+1}\colon B_i\rightarrow B_{i+1}$ is given by $a\mapsto (a,\pi_i(a)\otimes 1_{[0,1]})$ and hence satisfies $r_{i+1}^i\circ s_i^{i+1}=\id_{B_i}$.
\end{itemize}
\end{lemma}

\begin{proof}
Let $X=\bigsqcup_{j\in J}C_j$ be the decomposition of $X$ into its clopen connected components. Denote by $J_1\subseteq J$ the subset of those indices for which the corona space $\chi(C_j)$ is connected and note that $J_0:=J\backslash J_1$ is finite by Remark \ref{remark component}. We have the canonical commutative diagram 
\[\xymatrix{
    0 \ar[r] & \C_0(X,\M_N) \ar[r] \ar@{=}[d] & A \ar[r] \ar[d]^{\iota_0\oplus\iota_1} & B \ar[r] \ar[d]_{\tau_0\oplus\tau_1} & 0 \\ 
    0 \ar[r] & \bigoplus\limits_{*=0,1}\bigoplus\limits_{j\in J_*}\C_0(C_j,\M_N) \ar[r] & 
    \bigoplus\limits_{*=0,1}\prod\limits_{j\in J_*}\C(\beta C_j,\M_N) \ar[r]^{q_0\oplus q_1} & \bigoplus\limits_{*=0,1}\frac{\prod_{j\in J_*}\C(\beta C_j,\M_N)}{\bigoplus_{j\in J_*}\C_0(C_j,\M_N)} \ar[r] &  0
}\]
where $\tau_0\oplus\tau_1$ is the Busby map $\tau$ and the right square is a pullback diagram. Since we can pass to any strongly unitarily equivalent extension (in the sense of \cite[II.8.4.12]{Bla06}) without changing the retract map $R$, we can, by Proposition \ref{prop diagonal Busby} and the finiteness condition on $R$, assume that for every $j$ the image of 
\[\tau_j\colon B\xrightarrow{\tau}\frac{\prod_{j'}\C(\beta C_{j'},\M_N)}{\bigoplus_{j'}\C_0(C_{j'},\M_N)}\rightarrow\frac{\C(\beta C_j,\M_N)}{\C_0(C_j,\M_N)}=\C(\chi(C_j),\M_N)\]
is locally constant on $\chi(C_j)$, and even constant if $j\in J_1$. Furthermore, using lower semicontinuity of $R$ and arguing as in the proof of Corollary \ref{cor projective split}, we may assume that 
\[\iota_1(A)\subseteq \prod_{j\in J_1}\M_N\cdot 1_{\beta C_j}+\bigoplus_{j\in J_1}\C_0(C_j,\M_N).\] 

Next, we write $\alpha X=X\cup\{\infty\}$ as a limit of finite graphs. By Theorem \ref{thm ST} we can find a sequence of finite graphs $X_i\subset X_{i+1} \subset \alpha X$ such that $X_0=\core(\alpha X,\infty)$ (in the sense of \ref{def extended core}) and each $X_{i+1}$ is obtained from $X_i$ by attaching a line segment $[0,1]$ at the $0$-endpoint to a single point $y_{i}$ of $X_i$. Furthermore we have $\varprojlim X_i = \alpha X$ along the sequence of first point maps $\varrho_\infty^i\colon\alpha X\rightarrow X_i$. We need to fix some notation: Denote the inclusion of $X_i$ into $X_{i+1}$ by $\iota_i^{i+1}$ and the retract from $X_{i+1}$ to $X_i$ by collapsing the attached interval to the attaching point $y_i$ by $\varrho_{i+1}^i$. An analogous notation is used for the inclusion $X_i\subseteq\alpha X$: 
\[\xymatrix{
    X_i \ar[r]_{\iota_i^{i+1}}& X_{i+1}\ar@{->>}@/_1pc/[l]_{\varrho_{i+1}^i} \ar[r]_{\iota_{i+1}^{\infty}}& \alpha X\ar@{->>}@/_1pc/[l]_{\varrho_{\infty}^{i+1}}
}\]

Now for every pair of indices $i,j$ we have $X_i\cap C_j$ sitting inside $C_j$. Note that $X_{i+1}\backslash X_i\cap C_{j(i)}\neq\emptyset$ for a unique $j(i)\in J$ since $\infty\in X_0$. 
We define suitable compactifications $\alpha_j(X_i\cap C_j)$ of $X_i\cap C_j$ as follows: if $X_0\cap C_j=\emptyset$, we let $\alpha_j(X_i\cap C_j)=\alpha(X_i\cap C_j)$ be the usual one-point compactification for any $i\in\mathbb{N}$. In the case $X_0\cap C_j\neq\emptyset$, which will occur only finitely many times, we have an inclusion $\C_b(X_i\cap C_j)\subseteq \C_b(C_j)$ induced by the surjective retract $\varrho^i_{\infty|C_j}\colon C_j\rightarrow X_i\cap C_j$ and we define $\alpha_j(X_i\cap C_j)$ via 
\[\C(\alpha_j(X_i\cap C_j))=\left\{f\in\C_b(X_i\cap C_j)\subseteq\C_b(C_j)=\C(\beta C_j)\colon \begin{array}{c}f\;\text{is locally} \\ \text{constant on }\;\chi(C_j)\end{array}\right\}.\] 
Since the corona space $\chi(C_j)$ has only finitely many connected components by Lemma \ref{lemma corona}, $\alpha_j(X_i\cap C_j)$ will be a finite-point compactification of $X_i\cap C_j$ (meaning that $\alpha_j(X_i\cap C_j)\backslash(X_i\cap C_j)$ is a finite set). In particular, $\alpha_j(X_i\cap C_j)$ is a finite graph for any pair of indices $i$ and $j$. We are now ready to iteratively define the $C^*$-algebras $B_i$ as the pullbacks over 
\[\xymatrix{
    B_i \ar@{..>}[r] \ar@{..>}[d] & B \ar[d]^\tau\\
    \prod_j\C(\alpha_j(X_i\cap C_j),\M_N) \ar[r]^(0.6)q & \frac{\prod_j\C(\beta C_j,\M_N)}{\bigoplus_j\C_0(C_j,\M_N)}
}\]
with respect to the inclusions $(\varrho^i_{\infty|C_j})^*\otimes\id_{\M_N}\colon\C(\alpha_j(X_i\cap C_j),\M_N)\subseteq\C(\beta C_j,\M_N)$. Let us first simplify the description of $B_i$. For every fixed $i$, the set $X_i\cap C_j$ is empty for almost every $j\in J$ so that $\C(\alpha_j(X_i\cap C_j),\M_N)=\M_N\cdot 1_{\beta C_j}$ for almost every $j$. Given $((f_j)_j,b)\in B_i$, this implies $f_j=\tau_j(b)\cdot 1_{\beta C_j}$ for almost every $j$. Hence $B_i$ is isomorphic to the pullback
\[\xymatrix{
    B_i \ar@{..>}[r] \ar@{..>}[d] & B \ar[d]^{\mathop{\oplus}\limits_{j\in J(i)} \tau_j} \\ 
    \bigoplus\limits_{j\in J(i)} \C(\alpha_j (X_i\cap C_j),\M_N) \ar[r]^(0.6)q & \bigoplus\limits_{j\in J(i)}\frac{\C(\beta C_j,\M_N)}{\C_0(C_j,\M_N)}
}\]
for the finite set $J(i)=\{j\in J\colon X_i\cap C_j\neq\emptyset\}\subseteq J$. Since every $\alpha(X_i\cap C_j)$ is a finite graph, the $C^*$-algebra on the lower left side is a 1-NCCW. One also checks that the pullbacks are taken over finite-dimensional $C^*$-algebras because $(\oplus_{j\in J(i)}\tau_j)(B)$ consists of locally constant functions on the space $\bigsqcup_{j\in J(i)}\chi(C_j)$ which has only finitely many connected components by Lemma \ref{lemma corona}.

Next, we specify the inductive structure, i.e. the connecting maps $s_i^{i+1}\colon B_i\rightarrow B_{i+1}$ and retracts $r_{i+1}^i\colon B_{i+1}\rightarrow B_i$. By definition, we find $B_i\subseteq B_{i+1}\subseteq A$ with the inclusions coming from $(\varrho^i_{i+1})^*\otimes\id_{\M_N}$ resp. by $(\varrho^{i+1}_\infty)^*\otimes\id_{\M_N}$. We denote them by $s^{i+1}_i$ resp. by $s_i^\infty$. Since $\overline{\bigcup_i\C(\alpha_j(X_i\cap C_j),\M_N)}\supseteq\overline{\bigcup_i\C_0(X_i\cap C_j,\M_N)}=\C_0(X\cap C_j,\M_N)$ for every $j\in J$, we find $\C_0(X,\M_N)\subseteq\overline{\bigcup_i B_i}$. One further checks that $\bigoplus_{j\in J_0}\C(\alpha_j(X_0\cap C_j),\M_N)$ surjects via $q$ onto the locally constant functions on $\bigsqcup_{j\in J_0}\chi(C_j)$. Together with $\tau_1(B)\subseteq q_1(\prod_{j\in J_1}\M_N\cdot 1_{\beta C_j})\subseteq q_1(\prod_{j\in J_1}\C(\alpha_j(X_0\cap C_j),\M_N))$ it follows that $\overline{\bigcup_iB_i}$ is the pullback over $\tau$ and $q$, and hence $\overline{\bigcup_i B_i}=A$.

It remains to verify the description of the connecting maps $s_i^{i+1}$. We have $X_i\cap C_j=X_{i+1}\cap C_j$ if $j\neq j(i)$ and $\alpha_j(X_i\cap C_{j(i)})\subseteq \alpha_j(X_{i+1}\cap C_{j(i)})\cong \alpha_j(X_i\cap C_{j(i)})\cup_{\{y_i\}=\{0\}}[0,1]$. This means there is a pullback diagram 
\[\xymatrix{
    \C(\alpha_j(X_{i+1}\cap C_{j(i)}),\M_N) \ar@{..>}[r] \ar@{..>>}[d] & \C([0,1],\M_N) \ar@{->>}[d]^{\ev_0} \\ 
    \C(\alpha_j(X_i\cap C_{j(i)}),\M_N) \ar[r]^(0.6){\ev_{y_i}} \ar@/^1pc/[u]^{(\varrho_{i+1}^i)^*\otimes \id_{\M_N}}& \M_N
}\]
where $(\varrho_{i+1}^i)^*\otimes \id_{\M_N}$ corresponds to $f\mapsto(f,f(y_i)\otimes 1_{[0,1]})$ in the pullback picture and the downward arrow on the left side comes from the inclusion $\alpha_j(X_i\cap C_{j(i)})\subseteq\alpha_j(X_{i+1}\cap C_{j(i)})$. This map induces a surjection $B_{i+1}\rightarrow B_i$ which will be denoted by $r_{i+1}^i$ and gives the claimed pullback diagram.

Finally, if $\alpha X$ is an AR-space, the core $X_0=\core(\alpha X,\infty)=[x_0,\infty]$ is nothing but an arc from some point $x_0\in X$ to $\infty$. In this case the finite set $J(0)$ consists of a single element $j(0)$, namely the index corresponding to the component containing $x_0$. By definition, $B_0$ comes as a pullback  
\[\xymatrix{
    B_0 \ar@{..>}[r] \ar@{..>}[d] & B \ar[d]^{\tau_{j(0)}} \\ 
    \C([x_0,\infty],\M_N) \ar[r]^{\ev_\infty} & \M_N\cdot 1_{\chi(C_{j(0)})}
}\] 
and hence an index shift allows us to start with $B_0\cong B$.
\end{proof}

The procedure of forming extensions by $C^*$-algebras of the form $\C_0(X,\M_N)$ can of course be iterated. The next proposition shows that, if all the attached spaces $X$ are one-dimensional ANRs up to compactification, the limit structures which we get from Lemma \ref{lemma limit structure} for each step can be combined into a single one.

\begin{proposition}\label{prop iteration}
Let a short exact sequence of separable $C^*$-algebras $0\rightarrow\C_0(X,\M_N)\rightarrow A\rightarrow B\rightarrow 0$ be given. Assume that $\alpha X$ is a one-dimensional ANR-space and that the associated set-valued retract map $R\colon \prim(A)\rightarrow 2^{\prim(B)}$ as in \ref{def R} is lower semicontinuous and has pointwise finite image. Further assume that there exists a direct limit structure for $B$
\[\xymatrix{
  B_0 \ar[r]_{s_0^1}& B_1 \ar[r]_{s_1^2} \ar@/_1pc/@{->>}[l]_{r_1^0}& B_2 \ar[r] \ar@/_1pc/@{->>}[l]_{r_2^1} & 
  \cdots \ar[r]_{s_{i-1}^i} \ar@/_1pc/@{->>}[l]& B_i \ar[r] \ar@/_1pc/[rr]_{s_i^\infty} \ar@/_1pc/@{->>}[l]_{r_i^{i-1}}& \cdots \ar[r] & B \ar@/_1pc/@{->>}[ll]_{r_\infty^i}
}\]
such that all $B_i$ are 1-NCCWs and at each stage there is a representation $p_i\colon B_i\rightarrow\M_{n_i}$ such that $B_{i+1}$ is defined as the pullback 
\[\xymatrix{
  B_{i+1} \ar@{..>>}[d]^{r_{i+1}^i} \ar@{..>}[r]^(0.3){t_{i+1}}& \C([0,1],\M_{n_i}) \ar@{->>}[d]^{\ev_0} \\ 
  B_i \ar@/^1pc/[u]^{s_i^{i+1}} \ar[r]^{p_i}& \M_{n_i}
}\]
and $s_i^{i+1}\colon B_i\rightarrow B_{i+1}$ is given by $a\mapsto (a,p_i(a)\otimes 1_{[0,1]})$.

Then $A$ is isomorphic to the limit $A_\infty$ of an inductive system 
\[\xymatrix{
  A_0 \ar[r]_{\sigma_0^1}& A_1 \ar[r]_{\sigma_1^2} \ar@/_1pc/@{->>}[l]_{\varrho_1^0}& A_2 \ar[r] \ar@/_1pc/@{->>}[l]_{\varrho_2^1} & 
  \cdots \ar[r]_{\sigma_{i-1}^i} \ar@/_1pc/@{->>}[l]& A_i \ar[r] \ar@/_1pc/[rr]_{\sigma_i^\infty} \ar@/_1pc/@{->>}[l]_{\varrho_i^{i-1}}& \cdots \ar[r] & A_\infty \ar@/_1pc/@{->>}[ll]_{\varrho_\infty^i}
}\] 
where all $A_i$ are 1-NCCWs and at each stage there is a representation $\pi_i\colon A_i\rightarrow\M_{m_i}$ such that $A_{i+1}$ is defined as the pullback  
\[\xymatrix{
  A_{i+1} \ar@{..>>}[d]^{\varrho_{i+1}^i} \ar@{..>}[r]& \C([0,1],\M_{m_i}) \ar@{->>}[d]^{\ev_0} \\
  A_i \ar@/^1pc/[u]^{\sigma_i^{i+1}} \ar[r]^{\pi_i}& \M_{m_i}
}\] 
and $\sigma_i^{i+1}\colon A_i\rightarrow A_{i+1}$ is given by $a\mapsto (a,\pi_i(a)\otimes 1_{[0,1]})$. Furthermore, if $\alpha X$ is an AR-space we may even arrange $A_0\cong B_0$.
\end{proposition}

\begin{proof}
By Lemma \ref{lemma limit structure}, we know that $A$ can be written as an inductive limit
\[\xymatrix{
  \overline{A}_0 \ar[r]_{\overline{s}_0^1}& \overline{A}_1 \ar[r]_{\overline{s}_1^2} \ar@/_1pc/@{->>}[l]_{\overline{r}_1^0}& \overline{A}_2 \ar[r] \ar@/_1pc/@{->>}[l]_{\overline{r}_2^1} & 
  \cdots \ar[r]_{\overline{s}_{i-1}^i} \ar@/_1pc/@{->>}[l]& \overline{A}_i \ar[r] \ar@/_1pc/[rr]_{\overline{s}_i^\infty} \ar@/_1pc/@{->>}[l]_{\overline{r}_i^{i-1}}& \cdots \ar[r] & 
  A \ar@/_1pc/@{->>}[ll]_{\overline{r}_\infty^i}
}\]
with a pullback structure 
\[\xymatrix{
  \overline{A}_{i+1} \ar@{..>>}[d]^{\overline{r}_{i+1}^i} \ar@{..>}[r]& \C([0,1],\M_N) \ar@{->>}[d]^{\ev_0} \\ 
  \overline{A}_i \ar@/^1pc/[u]^{\overline{s}_i^{i+1}} \ar[r]^{\overline{p}_i}& \M_N
}\]
at every stage and $\overline{s}_i^{i+1}\colon\overline{A}_i\rightarrow \overline{A}_{i+1}$ given by $a\mapsto (a,\overline{p}_i(a)\otimes 1_{[0,1]})$. The starting algebra $\overline{A}_0$ comes as a pullback 
\[\xymatrix{
  \overline{A}_0 \ar[r] \ar[d] & D \ar[d]^\varphi \\ 
  B \ar[r]^\psi & F
}\] 
with $D$ a 1-NCCW and $\dim(F)<\infty$. 
In the case of $\alpha X$ being an AR-space, we may choose $\overline{A}_0=B$, i.e. $D=0$. For $j\in\mathbb{N}$ we now define $A_{0,j}$ to be the pullback
\[\xymatrix{
  A_{0,j} \ar@{..>}[r] \ar@{..>}[d]_{\varrho_{0,j}} & D \ar[d]^\varphi \\ 
  B_j \ar[r]^{\psi\circ s_j^{\infty}} & F.
}\]
The maps $s_j^{j+1}$,$s_j^\infty$ induce compatible homomorphisms $\sigma_{0,j}^{0,j+1}\colon A_{0,j}\rightarrow A_{0,j+1}$ and $\sigma_{0,j}^{0,\infty}\colon A_{0,j}\rightarrow\overline{A}_0$, leading to an inductive limit structure with $\varinjlim_j(A_{0,j},\sigma_{0,j}^{0,j+1})=\overline{A}_0$. We proceed iteratively, defining $A_{i+1,j}$ to be the pullback
\[\xymatrix{
  A_{i+1,j} \ar@{..>}[r] \ar@{..>}[d]^{\varrho_{i+1,j}^{i,j}} & \C([0,1],\M_N) \ar[d]^{\ev_0} \\ 
  A_{i,j} \ar[r]_{\overline{p}_i\circ\sigma_{i,j}^{i,\infty}} \ar@/^1pc/@{..>}[u]^{\sigma_{i,j}^{i+1,j}}& \M_N
}\]
with $\sigma_{i,j}^{i+1,j}\colon A_{i,j}\rightarrow A_{i+1,j}$ given by $a\mapsto (a,(\overline{p}_i\circ\sigma_{i,j}^{i,\infty})(a)\otimes 1_{[0,1]})$. It is then checked that $\sigma_{i,j}^{i,j+1}$ and $\sigma_{i,j}^{i,\infty}$ induce compatible homomorphisms $\sigma_{i+1,j}^{i+1,j+1}\colon A_{i+1,j}\rightarrow A_{i+1,j+1}$ and $\sigma_{i+1,j}^{i+1,\infty}\colon A_{i+1,j}\rightarrow \overline{A}_{i+1}$ with $\varinjlim_j (A_{i+1,j},\sigma_{i+1,j}^{i+1,j+1})=\overline{A}_{i+1}$. Observing that for every $i$ and $j$
\[\xymatrix{
  A_{i,j+1} \ar[rrr]^{t_{j+1}\circ\varrho_{0,j+1}\circ\varrho_{i,j+1}^{0,j+1}} \ar[d]^{\varrho_{i,j+1}^{i,j}} &&& \C([0,1],\M_{n_j}) \ar[d]^{\ev_0} \\ 
  A_{i,j} \ar[rrr]_{p_j\circ\varrho_{0,j}\circ\varrho_{i,j}^{0,j}} \ar@/^1pc/[u]^{\sigma_{i,j}^{i,j+1}} &&& \M_{n_j}
}\]
is indeed a pullback diagram, we get the desired limit structure for $A$ by following the diagonal in the commutative diagram
\[\xymatrix{
  A_{00} \ar[r] \ar@{..>}[d] & A_{01} \ar[r] \ar[d] & A_{02} \ar[r] \ar[d] & A_{03} \ar[r] \ar[d] & \cdots \ar[r] & \overline{A}_0  \ar[d]\\
  A_{10} \ar@{..>}[r] \ar[d] & A_{11} \ar[r] \ar@{..>}[d] & A_{12} \ar[r] \ar[d] & A_{13} \ar[r] \ar[d] & \cdots \ar[r] & \overline{A}_1 \ar[d] \\
  A_{20} \ar[r] \ar[d] & A_{21} \ar@{..>}[r] \ar[d] & A_{22} \ar[r] \ar@{..>}[d] & A_{23} \ar[r] \ar[d] & \cdots \ar[r] & \overline{A}_2 \ar[d] \\
  A_{30} \ar[r] \ar[d] & A_{31} \ar[r] \ar[d] & A_{32} \ar@{..>}[r] \ar[d] & A_{33} \ar[r] \ar@{..>}[d] & \cdots \ar[r] & \overline{A}_3 \ar[d]\\
  \vdots & \vdots & \vdots & \vdots & \ddots &\vdots
}\]
as indicated. Note that, since all connecting maps are injective, the limit over the diagonal equals $\varinjlim\overline{A}_n=A$.
\end{proof}

\subsection{Keeping track of semiprojectivity}\label{section keeping track}

We now reap the fruit of our labor in the previous sections and work out a '2 out of 3'-type statement describing the behavior of semiprojectivity with respect to extensions by homogeneous $C^*$-algebras. While for general extensions the behavior of semiprojectivity is either not at all understood or known to be rather bad, Theorem \ref{thm 2 out of 3} gives a complete and satisfying answer in the case of homogeneous ideals. It is the very first result of this type and allows to understand semiprojectivity for $C^*$-algebras which are built up from homogeneous pieces, see chapter \ref{section main}. 

\begin{proposition}\label{prop ideal ANR}
Let $0\rightarrow\C_0(X,\M_N)\rightarrow A\rightarrow B\rightarrow 0$ be a short exact sequence of $C^*$-algebras. 
If both $A$ and $B$ are (semi)projective, then the one-point compactification $\alpha X$ is a one-dimensional A(N)R-space.
\end{proposition}

\begin{proof}
The projective case follows directly from Corollary \ref{cor projective ideal} and Theorem \ref{thm comm case} while the semiprojective case requires some more work. By Lemma \ref{lemma ideal Peano} we know that $\alpha X$ is a Peano space of dimension at most 1. The proof of \ref{lemma ideal Peano} further shows that there are no small circles accumulating in $X$. However, in order to show that $\alpha X$ is an ANR-space we have to exclude the possibility of smaller and smaller circles accumulating at $\infty\in\alpha X$, see Theorem \ref{thm ward}. Assume that we find a sequence of circles with diameters converging to $0$ (with respect to some fixed geodesic metric) at around $\infty\in\alpha X$. After passing to a subsequence, there are two possible situations: either each circle contains the point $\infty$ or none of them does. Both cases are treated exactly the same, for the sake of brevity we only consider the situation where $\infty$ is contained in all circles. In this case have a $^*$-homomorphism $\varphi\colon\C_0(X,\M_N)\rightarrow\bigoplus_{n=1}^\infty\C_0((0,1)_n,\M_N)$ where $(0,1)_n\cong(0,1)$ is the part of the $n$-th circle contained in $X$. Note that each coordinate projection gives a surjection $\varphi_n\colon\C_0(X,\M_N)\rightarrow\C_0((0,1),\M_N)$ while $\varphi$ itself is not necessarily surjective (because the circles might intersect in $X$). We make use of Brown's mapping telescope associated to $\bigoplus_{n=1}^\infty\C_0((0,1)_n,\M_N)$, i.e. 
\[T_k=\left\{f\in\C([k,\infty],\bigoplus_{n=1}^\infty\C_0((0,1)_n,\M_N))\colon t\leq l\Rightarrow f(t)\in\bigoplus_{n=1}^l\C_0((0,1)_n,\M_N)\right\}\]
with the obvious (surjective) restrictions as connecting maps giving $\varinjlim T_k=\bigoplus_{n=1}^\infty\C_0((0,1)_n,\M_N)$. Using Lemma \ref{lemma telescope}, we find a commutative diagram 
\[\xymatrix{
  0 \ar[r] & \varinjlim T_k \ar[r] & \varinjlim\mathcal{M}(T_k) \ar[r] & \varinjlim\mathcal{Q}(T_k) \ar[r] & 0 \\ 
  0 \ar[r] & \C_0(X,\M_N) \ar[r] \ar[u]^\varphi & A \ar[r] \ar[u]^{\overline{\varphi}} & B \ar[r] \ar[u]^{\overline{\overline{\varphi}}} & 0
}.\]
It now follows from the semiprojectivity assumptions and Lemma \ref{lemma extended lifting} that $\varphi$ lifts to $T_k$ for some $k$, which is equivalent to a solution of the original lifting problem 
\[\xymatrix{
  & \bigoplus\limits_{n=1}^k \C_0((0,1)_n,\M_N) \ar[d]^\subseteq \\
  \C_0(X,\M_N) \ar[r]^(0.4)\varphi \ar@/^1pc/@{..>}[ur] & \bigoplus\limits_{n=1}^\infty\C_0((0,1)_n,\M_N)
}\]
up to homotopy. This, however, implies
\[\image(K_1(\varphi))\subseteq K_1\left(\bigoplus_{n=1}^k\C_0((0,1)_n,\M_N)\right)=\sum_{n=1}^k\mathbb{Z}\subset\sum_{n=1}^\infty\mathbb{Z} =K_1\left(\bigoplus_{n=1}^\infty\C_0((0,1)_n,\M_N)\right)\]
which gives a contradiction as follows. Because $\varphi_{k+1}$ is surjective and $\dim(\alpha X)\leq 1$, we can extend the canonical unitary function from $\alpha((0,1)_n)$ to a unitary $u$ on all of $\alpha X$ by \cite[Theorem VI.4]{HW48}. This unitary then satisfies $u-1\in\C_0(X)$ and $K_1(\varphi)([u\otimes 1_{\M_N}])=N\in\mathbb{Z}=K_1(\C_0((0,1)_{k+1},\M_N))$. This shows that there are no small circles at around $\infty$ in $\alpha X$ and hence that $\alpha X$ is a one-dimensional ANR-space by Theorem \ref{thm ward}. 
\end{proof}

\begin{theorem}\label{thm 2 out of 3}
Let a short exact sequence of $C^*$-algebras $0\rightarrow I\rightarrow A\rightarrow B\rightarrow 0$ be given and assume that the ideal $I$ is a $N$-homogeneous $C^*$-algebra with $\prim(I)=X$. Denote by $(X_i)_{i\in I}$ the connected components of $X$ and consider the following statements:
\renewcommand{\labelenumi}{(\Roman{enumi})}\begin{enumerate}
  \item $I$ is (semi)projective.
  \item $A$ is (semi)projective.
  \item $B$ is (semi)projective and the set-valued retract map $R\colon\prim(A)\rightarrow 2^{\prim(B)}$ given as in \ref{def R} by
\[R(z)=\begin{cases}
  \;z & \text{if}\;z\in\prim(B), \\
  \partial X_i=\overline{X_i}\backslash X_i & \text{if}\;z\in X_i\subseteq X=\prim(I)
\end{cases}\]
is lower semicontinuous and has pointwise finite image.
\end{enumerate}
If any two of these statements are true, then the third one also holds.
\end{theorem}

\begin{proof}
(I)+(II)$\Rightarrow$(III): By Theorem \ref{thm homogeneous case}, we know that the sequence is isomorphic to an extension 
\[\xymatrix{0\ar[r]&\C_0(X,\M_N)\ar[r]&A\ar[r]^\pi&B\ar[r]&0}\]
with the one-point compactification of $X$ a one-dimensional A(N)R-space. The set-valued retract map $R$ is then lower semicontinuous by Proposition \ref{prop R is lsc} and has pointwise finite image by Proposition \ref{prop finite boundary}. But now Theorem \ref{thm semiprojective split} applies and shows that there is a completely positive split $s$ for the quotient map $\pi$ such that the composition $B\xrightarrow{s} A\xrightarrow{\iota}\C_b(X,\M_N)$ is multiplicative outside of an open set $U\subset K\subset X$ where $K$ is compact.

Let a lifting problem $\varphi\colon B\xrightarrow{\sim} D/J=\varinjlim D/J_n$ be given. Since $A$ is semiprojective, we can solve the resulting lifting problem for $A$, meaning we find $\psi\colon A\rightarrow D/J_n$ for some $n$ with $\pi_n\circ\psi=\varphi\circ\pi$. Restricting to $\her_{D/J_n}(\psi(\C_0(X,\M_N)))+\psi(A)\subseteq D/J_n$, we may assume that $\psi_{|\C_0(X,\M_N)}$ is proper as a $^*$-homomorphism to $J/J_n$ 
and hence induces a map $\mathcal{M}(\psi)$ between multiplier algebras. Since the restriction of $\pi_n\circ\psi$ to the ideal $\C_0(X,\M_N)$ vanishes, we may use compactness of $K$ to assume that $\psi$ maps $\C_0(U,\M_N)$ to $0$ (after increasing $n$ if necessary). This further implies that $\mathcal{M}(\psi)$ factors through $r\colon\C_b(X,\M_N)\rightarrow\C_b(X\backslash U,\M_N)$. We then find $s':=r\circ\iota\circ s$ to be multiplicative and hence a $^*$-homomorphism: 
\[\xymatrix{
  A\ar[dr]^\iota \ar[ddd]^\pi \ar[rrr]^\psi & & & D/J_n \ar[dr]^{\iota_n}\ar@{->>}[ddd]^(0.6){\pi_n}|!{[dll];[dr]}\hole|!{[ddl];[dr]}\hole \\ 
  & \C_b(X,\M_N) \ar@{->>}[dr]+(-12.5,3)^(.65)r \ar[rrr]^{\mathcal{M}(\psi)} & & & \mathcal{M}(J/J_n) \ar@{->>}[dd]^{\varrho_n}\\ 
  & & \C_b(X\backslash U,\M_N)\ar@{-->}[urr]^(0.35){\mathcal{M}(\psi)'} \\ B \ar[rrr]^\varphi \ar@/^2pc/[uuu]^s \ar@{-->}[urr]^{s'}& & & D/J \ar[r]^(0.4){\tau_n}& \mathcal{Q}(J/J_n)
}\]
The inclusion of $J/J_n$ as an ideal in $D/J_n$ gives canonical $^*$-homomorphisms $\iota_n$ and $\tau_n$ as in the diagram above. One now checks that $\varrho_n\circ(\mathcal{M}(\psi)'\circ s')=\tau_n\circ\varphi$ holds. Combining this with the fact that the trapezoid on the right is a pullback diagram, we see that the pair $(\varphi,(\mathcal{M}(\psi)'\circ s'))$ defines a lift $B\rightarrow D/J_n$ for $\varphi$. This shows that the quotient $B$ is semiprojective.

For the projective version of the statement, one uses Corollary \ref{cor projective split} to see that the sequence admits a multiplicative split $s\colon B\rightarrow A$ rather than just a completely positive one.\\

(I)+(III)$\Rightarrow$(II): We know that $I\cong\C_0(X,\M_N)$ with $\alpha X$ a one-dimensional A(N)R-space by Theorem \ref{thm homogeneous case}. Now Lemma \ref{lemma limit structure} applies and we obtain a limit structure for $A$
\[\xymatrix{
    B_0 \ar[r]_{s_0^1}& B_1 \ar[r]_{s_1^2} \ar@/_1pc/@{->>}[l]_{r_1^0}& B_2 \ar[r] \ar@/_1pc/@{->>}[l]_{r_2^1} & \cdots \ar[r]_{s_{i-1}^i} \ar@/_1pc/@{->>}[l]& B_i \ar[r] \ar@/_1pc/[rr]_{s_i^\infty} \ar@/_1pc/@{->>}[l]_{r_i^{i-1}}& \cdots \ar[r] & \varinjlim \left(B_i,s_i^{i+1}\right)\cong A \ar@/_1pc/@{->>}[ll]_{r_\infty^i}
}\]
with $B_0$ given as a pullback of $B$ and a 1-NCCW $D$ over a finite-dimensional $C^*$-algebra. In particular, $B_0$ is semiprojective by \cite[Corollary 3.4]{End14}. In the projective case, we can take $B_0=B$ to be projective. In both cases, the connecting maps in the system above arise from pullback diagrams
\[\xymatrix{
  B_{i+1} \ar@{..>>}[d]^{r_{i+1}^i} \ar@{..>}[r]& \C([0,1],\M_N) \ar@{->>}[d]^{\ev_0} \\ 
  B_i \ar@/^1pc/[u]^{s_i^{i+1}} \ar[r]^{\pi_i}& \M_N
}\]
with $s_i^{i+1}(a)=(a,\pi_i(a)\otimes 1_{[0,1]})$. 
Since these maps are weakly conditionally projective by Proposition \ref{prop wcp examples}, we obtain (semi)projectivity of $A$ from Lemma \ref{lemma limit criterium}.\\

(II)+(III)$\Rightarrow$(I): This implication holds under even weaker hypothesis. More precisely, we show that (semi)projectivity of both $A$ and $B$ implies $I$ to be (semi)projective. The assumption on the retract map $R$ is not needed here. 

First we apply Lemma \ref{lemma ideal Peano} to find the one-point compactification of $\prim(I)$ to be a Peano space of dimension at most 1, and hence $I$ is trivially homogeneous by Lemma \ref{lemma trivial bundles}. Now Proposition \ref{prop ideal ANR} shows that $\alpha X$ is in fact an ANR-space which, together with Theorem \ref{thm homogeneous case}, means that $I$ is semiprojective. The projective version is Corollary \ref{cor projective ideal}.
\end{proof}

\begin{remark}\label{remark comm retract}
Theorem \ref{thm 2 out of 3} shows that regularity properties of the retract map $R\colon\prim(A)\rightarrow 2^{\prim(B)}$ are crucial for semiprojectivity to behave nicely with respect to extensions by homogeneous $C^*$-algebras. This can already be observed and illustrated in the commutative case. Given an extension of commutative $C^*$-algebras 
\[0\rightarrow\C_0(X)\rightarrow\C_0(Y)\rightarrow\C_0(Y\backslash X)\rightarrow 0,\]
the following holds: If both the ideal $\C_0(X)$ and the quotient $\C_0(Y\backslash X)$ are (semi)projective, then the extension $\C_0(Y)$ is (semi)projective if and only if the retract map $R\colon Y\rightarrow 2^{Y\backslash X}$ is lower semicontinuous and has pointwise finite image. The following examples show that both properties for $R$ are not automatic: 

(a) An examples with $R$ not having pointwise finite image is contained as example 5.5 in \cite{LP98}, we include it here for completeness. Let $X=\{(x,\sin x^{-1}):0<x\leq 1\}\subset\mathbb{R}^2$ and $Y=X\cup\{(0,y)\colon -1\leq y< 2\}$, then we get an extension isomorphic to 
\[0\rightarrow\C_0(0,1]\rightarrow\C_0(Y)\rightarrow\C_0(0,1]\rightarrow 0.\] 
Here both the ideal and the quotient are projective, but the extension $\C_0(Y)$ is not (because $\alpha Y$ is not locally connected and hence not an AR-space). In this example, we find $R(x)=\{(0,y)\colon -1\leq y\leq 1\}$ to be infinite for all $x\in X$.

(b) The following is an example where $R$ fails to be lower semicontinuous. Consider $Y=\{(x,0)\colon 0\leq x<1\}\cup\bigcup_n C_n\subset\mathbb{R}^2$ with $C_n=\{(t,(1-t)/n)\colon 0\leq t<1\}$ the straight line from $(0,1/n)$ to $(1,0)$ with the endpoint $(1,0)$ removed. With $X=\bigcup_n C_n\subset Y$ we obtain an extension isomorphic to 
\[0\rightarrow\bigoplus_n\C_0(0,1]\rightarrow \C_0(Y)\rightarrow\C_0(0,1]\rightarrow 0.\]
Here both the ideal and the quotient are projective while the extension $\C_0(Y)$ is not (again because $\alpha Y$ is not locally connected). We also find $(0,1/n)\rightarrow (0,0)$ in $Y$ but $R((0,1/n))=\emptyset$ for all $n$, which shows that $R$ is not lower semicontinuous. The descriptive reason for $\C_0(Y)$ not being projective in this case is that the length of the attached intervals $C_n$ does not tend to $0$ as $n$ goes to infinity.  
\end{remark}

\section{The structure of semiprojective subhomogeneous $C^*$-algebras}\label{section structure}

\subsection{The main result}\label{section main}

With Theorem \ref{thm 2 out of 3} at hand, we are now able to keep track of semiprojectivity when decomposing a subhomogeneous $C^*$-algebra into its homogeneous subquotients. On the other hand, Theorem \ref{thm 2 out of 3} also tells us in which manner homogeneous, semiprojective $C^*$-algebras may be combined in order to give subhomogeneous $C^*$-algebras which are again semiprojective. This leads to the main result of this chapter, Theorem \ref{thm structure}, which gives two characterizations of projectivity and semiprojectivity for subhomogeneous $C^*$-algebras.

\begin{lemma}\label{lemma Prim_max}
Let $A$ be a $N$-subhomogeneous $C^*$-algebra. If $A$ is semiprojective, then the maximal $N$-homogeneous ideal of $A$ is also semiprojective.
\end{lemma}

\begin{proof}
By Lemma \ref{lemma ideal Peano} we know that the one-point compactification of $X=\prim_N(A)$ is a one-dimensional Peano space. Since any locally trivial $\M_N$-bundle over $X$ is globally trivial by Lemma \ref{lemma trivial bundles}, we are concerned with an extension of the form 
\[\xymatrix
  {0\ar[r]&\C_0(X,\M_N)\ar[r]&A\ar[r]^(0.4)\pi&A_{\leq N-1}\ar[r]&0
}\]
where $A_{\leq N-1}$ denotes the maximal ($N$-1)-subhomogeneous quotient of $A$. Since $A$ is semiprojective, $A_{\leq N-1}$ will be semiprojective with respect to ($N$-1)-subhomogeneous $C^*$-algebras. In order to show that $\C_0(X,\M_N)$ is semiprojective, it remains to show that $\alpha X=X\cup\{\infty\}$ does not contain small circles at around $\infty$, cf. Theorem \ref{thm ward}. The proof for this is similar to the one of \ref{prop ideal ANR}. We use notations from \ref{lemma ideal Peano} and follow the proof there to arrive a commutative diagram
\[\xymatrix{
  0\ar[r]&\varinjlim T_k \ar[r] & \varinjlim\mathcal{M}(T_k) \ar[r] & \varinjlim\mathcal{Q}(T_k) \ar[r] & 0 \\ 
  0 \ar[r] & \C_0(X,\M_N) \ar[r] \ar[u]^\varphi & A \ar[r]^(0.4)\pi \ar[u]^{\overline{\varphi}} & A_{\leq N-1} \ar[r] \ar[u]^{\overline{\overline{\varphi}}} & 0
}.\]
We may not solve the lifting problem for $A_{\leq N-1}$ directly since the algebras $\mathcal{Q}(T_k)$ are not ($N$-1)-subhomogeneous. Instead we will replace the $\mathcal{Q}(T_k)$ by suitable ($N$-1)-subhomogeneous subalgebras which will then lead to a solvable lifting problem for $A_{\leq N-1}$. Let $\iota_n$ denote the $n$-th coordinate of the map $A\rightarrow\C_b(X,\M_N)\rightarrow\prod_n\C_b((0,1)_n,\M_N)$. We then have a lift of $\overline{\varphi}$ given by
\[\begin{array}{rcccl}
    A & \rightarrow & \C([k,\infty],\prod_n\C_b((0,1),\M_N)) & \rightarrow & \mathcal{M}(T_k) \\
    a & \mapsto & 1_{[k,\infty]}\otimes(\iota_n(a))_{n=1}^\infty
\end{array}\]
where the map on the right is induced by the inclusion of $T_k$ as an ideal in $\C([k,\infty],\prod_n\C_b((0,1),\M_N))$. Consider in there the central element $f=(f_n)_{n=1}^\infty$ with $f_n$ the scalar function that equals 0 on $[k,n]$, 1 on $[n+1,\infty]$ and which is linear in between. Then 
\[\begin{array}{rcccl}
   \psi\colon A & \rightarrow & \C([k,\infty],\prod_n\C_b((0,1),\M_N)) & \rightarrow & \mathcal{M}(T_k) \\
   a & \mapsto & (f_n\otimes\iota_n(a))_{n=1}^\infty
  \end{array}\]
is a completely positive lift of $\overline{\varphi}$ which sends $\C_0(X,\M_N)$ to $T_k$. Hence $\psi$ induces a completely positive lift $\psi'\colon A_{\leq N-1}\rightarrow \mathcal{Q}(T_k)$ of $\overline{\overline{\varphi}}$. We claim that $C^*(\psi'(A_{\leq N-1}))$ is in fact ($N$-1)-subhomogeneous. To see this, we use the algebraic characterization of subhomogenity as described in \cite[IV.1.4.6]{Bla06}. It suffices to check that $\gamma(C^*(\psi'(A_{\leq N-1})))$ satisfies the polynomial relations $p_{r(N-1)}$ for every irreducible representation $\gamma$ of $\mathcal{Q}(T_k)$. By definition of $\psi$, we find $\gamma\circ\psi'(\pi(a))=t\cdot\gamma'(\iota(a))$ for some representation $\gamma'$ of $\iota(A)$, some $t\in[0,1]$ and every $a\in A$. Moreover, since $\psi'$ maps $\C_0(X,\M_N)$ to 0, we obtain $\gamma\circ\psi'(\pi(a))=t\cdot\gamma''(\pi(a))$ for some representation $\gamma''$ of $A_{\leq N-1}$. Using ($N$-1)-subhomogeneity of $A_{\leq N-1}$, it now follows easily that the elements of $\gamma(C^*(\psi'(A_{\leq N-1})))$ satisfy the polynomial relations $p_{r(N-1)}$ from \cite[IV.1.4.6]{Bla06}. Knowing that the image of $\overline{\overline{\varphi}}$ has a ($N$-1)-subhomogenous preimage in $\mathcal{Q}(T_k)$, we may now solve the lifting problem for $A_{\leq N-1}$. It then follows from Lemma \ref{lemma extended lifting} (and its proof) that $\varphi$ lifts to $T_k$ for some $k$. The remainder of the proof is exactly the same as the one of Proposition \ref{prop ideal ANR}.
\end{proof}

We now present two characterizations of projectivity and semiprojectivity for subhomogeneous $C^*$-algebras. The first one describes semiprojectivity of these algebras in terms of their primitive ideal spaces. The second description characterizes them as those $C^*$-algebras which arise from 1-NCCWs by adding a sequence of non-commutative edges (of bounded dimension), cf. section \ref{section adding edges}.

\begin{theorem}\label{thm structure}
Let $A$ be a $N$-subhomogeneous $C^*$-algebra, then the following are equivalent:
\renewcommand{\labelenumi}{(\arabic{enumi})}
\begin{enumerate}
\item $A$ is semiprojective (resp. projective).
\item For every $n=1,...,N$ the following holds:
	\begin{itemize}
		\item The one-point compactification of $\prim_n(A)$ is an ANR-space (resp. an AR-space) of dimension at most 1.
		\item If $(X_i)_{i\in I}$ denotes the family of connected components of $\prim_n(A)$, then the set-valued retract map \[R_n\colon\prim_{\leq n}(A)\rightarrow 2^{\prim_{\leq n-1}(A)}\] given by
		      \[\begin{array}{rl}
			  z&\mapsto
			      \begin{cases}
				  z & \;\text{if}\;z\in \prim_{\leq n-1}(A) \\ 
				  \partial X_i & \;\text{if}\;z\in X_i\subset \prim_n(A)
			      \end{cases}
		      \end{array}\]
		      is lower semicontinuous and has pointwise finite image.
	\end{itemize}
	\item $A$ is isomorphic to the direct limit $\varinjlim_k(A_k,s_k^{k+1})$ of a sequence of 1-NCCW's
	\[\xymatrix{\cdots \ar[r] & A_k \ar[r]_{s_k^{k+1}} \ar@/_1pc/@{->>}[l] & A_{k+1} \ar[r] \ar@/_1pc/@{->>}[l]_{r_{k+1}^k} & \cdots \ar@/_1pc/@{->>}[l]}\]
	(with $A_0=0$) such that for each stage there is a pullback diagram 
	    \[\xymatrix{
		A_{k+1} \ar@{..>}[r] \ar@{..>>}[d]^{r_{k+1}^k}& \C([0,1],\M_n) \ar[d]^{\ev_0} \\
		A_k \ar[r]^{\pi_k} \ar@/^1pc/[u]^{s_k^{k+1}} & \M_n
	    }\]
	with $n\leq N$ and $s_k^{k+1}$ given by $a\mapsto(a,\pi_k(a)\otimes 1_{[0,1]})$.
\end{enumerate}
\end{theorem}

\begin{proof}
$(1)\Rightarrow (2)$: We prove the implication by induction over $N$. The base case $N=1$ is given by Theorem \ref{thm comm case}. Now given a $N$-subhomogeneous, (semi)projective $C^*$-algebra $A$, we know by Lemma \ref{lemma Prim_max} that the maximal $N$-homogeneous ideal $A_N$ of $A$ is (semi)projective as well. This forces $\alpha\prim_N(A)$ to be a one-dimensional A(N)R-space by Theorem \ref{thm homogeneous case}. Applying Theorem \ref{thm 2 out of 3} to the sequence 
\[0\rightarrow A_N\rightarrow A\rightarrow A_{\leq N-1}\rightarrow 0\]
now shows that the retract map $R_N\colon\prim_N(A)\rightarrow 2^{\prim_{\leq N-1}(A)}$ is lower semicontinuous, has pointwise finite image and that the maximal ($N$-1)-subhomogeneous quotient $A_{\leq N-1}$ is (semi)projective. The remaining statements follow from the induction hypothesis applied to $A_{\leq N-1}$.

$(2)\Rightarrow (3)$: By Lemma \ref{lemma trivial bundles}, we know that the maximal $N$-homogeneous ideal $A_N$ of $A$ is of the form $\C_0(\prim_N(A),\M_N)$. Using induction over $N$, the statement then follows from Proposition \ref{prop iteration} applied to the sequence 
\[0\rightarrow\C_0(\prim_N(A),\M_N)\rightarrow A\rightarrow A_{\leq N-1}\rightarrow 0.\] 
The base case $N=1$ is given by Theorem \ref{thm ST}.

$(3)\Rightarrow (1)$: Note that the connecting maps are weakly conditionally projective by Proposition \ref{prop wcp examples}, then apply Lemma \ref{lemma limit criterium}.
\end{proof}

\begin{remark}
The most prominent examples of subhomogeneous, semiprojective $C^*$-algebras are the one-dimensional non-commutative CW-complexes (1-NCCWs, see Example \ref{ex 1-nccw}). The structure theorem \ref{thm structure} shows that these indeed play a special role in the class of all subhomogeneous, semiprojective $C^*$-algebras. By part (2) of \ref{thm structure}, they  are precisely those subhomogeneous, semiprojective $C^*$-algebras for which the spaces $\alpha\prim_n$ are all finite graphs rather than general one-dimensional ANR-spaces. Hence 1-NCCWs should be thought of as the elements of 'finite type' in the class of subhomogeneous, semiprojective $C^*$-algebras. Moreover, part (3) of \ref{thm structure} shows that every subhomogeneous, semiprojective $C^*$-algebra can be constructed from 1-NCCWs in a very controlled manner. Therefore these algebras share many properties with 1-NCCWs, as we will see in section \ref{section properties} in more detail.
\end{remark}

\subsection{Applications}\label{section applications}

Now we discuss some consequences of Theorem \ref{thm structure}. First we collect some properties of semiprojective, subhomogeneous $C^*$-algebras which follow from the descriptions in \ref{thm structure}. This includes information about their dimension and $K$-theory as well as details about their relation to 1-NCCWs and some further closure properties.

At least in principle one can use the structure theorem \ref{thm structure} to test any given subhomogeneous $C^*$-algebra $A$ for (semi)projectivity. Since this would require a complete computation of the primitive ideal space of $A$, it is not recommended though. Instead one might use \ref{thm structure} as a tool to disprove semiprojectivity for a candidate $A$. In fact, showing directly that a $C^*$-algebra $A$ is not semiprojective can be surprisingly difficult. One might therefore take one of the conditions from \ref{thm structure} which are easier to verify and test $A$ for those instead. We illustrate this strategy in section \ref{section A4} by proving the quantum permutation algebras to be not semiprojective. This corrects a claim in \cite{Bla04} on semiprojectivity of universal $C^*$-algebras generated by finitely many projections with order and orthogonality relations.

\subsubsection{Further structural properties}\label{section properties}
By part (3) of Theorem \ref{thm structure}, we know that any semiprojective, subhomogeneous $C^*$-algebra comes as a direct limit of 1-NCCWs. Since the connecting maps are explicitly given and of a very special nature, it is possible to show that these limits are approximated by 1-NCCWs in a very strong sense. The following corollary makes this approximation precise. 

\begin{corollary}[Approximation by 1-NCCWs]\label{cor approx nccw}
Let $A$ be a subhomogeneous $C^*$-algebra. If $A$ is semiprojective, then for every finite set $\mathcal{G}\subset A$ and every $\epsilon>0$ there exist a 1-NCCW $B\subseteq A$ and a $^*$-homomorphism $r\colon A\rightarrow B$ such that $\mathcal{G}\subset_\epsilon B$ and $r$ is a strong deformation retract for $B$, meaning that there exists a homotopy $H_t$ from $H_0=\id_A$ to $H_1=r$ with $H_{t|B}=\id_B$ for all $t$. In particular, $A$ is homotopy equivalent to a one-dimensional non-commutative CW-complex.
\end{corollary}

\begin{proof}
Use part (3) of Theorem \ref{thm structure} to write $A=\varinjlim A_n$ and find a suitable 1-NCCW $B=A_{n_0}$ which almost contains the given finite set $\mathcal{G}$. It is straightforward to check that the strong deformation retracts $r_n^{n_0}\colon A_n\rightarrow A_{n_0}$ give rise to a strong deformation retract $r\colon \varinjlim A_n\rightarrow A_{n_0}$.
\end{proof}

In particular, 1-NCCWs and semiprojective, subhomogeneous $C^*$-algebras share the same homotopy invariant properties. For example, we obtain the following restrictions on the $K$-theory of these algebras:

\begin{corollary}\label{cor k-theory}
Let $A$ be a subhomogeneous $C^*$-algebra. If $A$ is semiprojective, then its $K$-theory is finitely generated and $K_1(A)$ is torsion free.
\end{corollary}

Another typical phenomenon of (nuclear) semiprojective $C^*$-algebras is that they appear to be one-dimensional in some sense. In the context of subhomogeneous $C^*$-algebras, we can now make this precise, using the notion of topological dimension given by $\topdim(A)=\max_n \dim(\prim_n(A))$. 

\begin{corollary}\label{cor dimension}
Let $A$ be a subhomogeneous $C^*$-algebra. If $A$ is semiprojective, then $A$ has stable rank 1 and $\topdim(A)\leq 1$.
\end{corollary}

\begin{proof}
The statement on the stable rank of $A$ follows from Corollary \ref{cor approx nccw}, while the topological dimension can be estimated using part (2) of Theorem \ref{thm structure}.
\end{proof}

Our structure theorem can also be used to study permanence properties of semiprojectivity when restricted to the class of subhomogeneous $C^*$-algebras. In fact, these turn out to be way better then in the general situation. This can be illustrated by the following longstanding question by Blackadar and Loring: Given a short exact sequence of $C^*$-algebras
\[\xymatrix{0 \ar[r] & I \ar[r] & A \ar[r] & F \ar[r] & 0}\]
with finite-dimensional $F$, does the following hold? 
\[I\;\text{semiprojective}\Leftrightarrow A\;\text{semiprojective}\]
While we showed the '$\Leftarrow$'-implication to hold in general in \cite{End14}, S. Eilers and T. Katsura proved the '$\Rightarrow$'-implication to be wrong (\cite{EK}), even in the case of split extensions by $\mathbb{C}$. We refer the reader to \cite{Sor12} for counterexamples which involve infinite $C^*$-algebras. However, when one restricts to the class of subhomogeneous $C^*$-algebras, this implication holds:

\begin{corollary}\label{cor blalor}
Let a short exact sequence of $C^*$-algebras 
\[\xymatrix{0 \ar[r]&I\ar[r]&A\ar[r]^\pi &F\ar[r]&0}\]
with finite-dimensional $F$ be given. If $I$ is subhomogeneous and semiprojective, then $A$ is also semiprojective.
\end{corollary}

\begin{proof}
We verify condition (2) in Theorem \ref{thm structure} for $A$. By assumption, each $\prim_k(I)$ is a one-dimensional ANR-space after compactification and the same holds for any space obtained from $\prim_k(I)$ by adding finitely many points (\cite[Theorem 6.1]{ST12}). Hence the one-point compactifications of $\prim_k(A)$ are 1-dimensional ANRs for all $k$. If we assume $F=\M_n$, then the set-valued retract maps $R_k$ are unchanged for $k<n$. For $k=n$, regularity of $R_k$ follows from regularity of the retract map for $I$ and the fact that $\{[\pi]\}$ is closed in $\prim_{\leq k}(A)=\prim_{\leq k}(I)\cup\{[\pi]\}$. For $k>n$, we apply Lemma \ref{lemma finite extension} to
\[\xymatrix{
  && 0 \ar[d] & 0 \ar[d] \\
  0 \ar[r] & \C_0(\prim_k(I),\M_k) \ar[r] \ar@{=}[d] & I_{\leq k} \ar[r] \ar[d] & I_{\leq k-1} \ar[r] \ar[d] & 0 \\
  0 \ar[r] & \C_0(\prim_k(A),\M_k) \ar[r] & A_{\leq k} \ar[r] \ar[d] ^\pi & A_{\leq k-1} \ar[r] \ar[d] & 0 \\
  && F \ar[d] \ar@{=}[r] & F \ar[d] \\
  && 0 & 0
}\]
and see that $R_k\colon\prim_{\leq k}(A)\rightarrow 2^{\prim_{\leq k-1}(A)}$ is again lower semicontinuous and has pointwise finite image.
\end{proof}

\subsubsection{Quantum permutation algebras}\label{section A4}
We are now going to demonstrate how the structure theorem \ref{thm structure} can be used to show that certain $C^*$-algebras fail to be semiprojective. We would like to thank T. Katsura for pointing out to us the quantum permutation algebras (\cite{Wan98}, \cite{BC08}) as a testcase: 

\begin{definition}[\cite{BC08}]\label{definition A4}
For $n\in\mathbb{N}$, the quantum permutation algebra $A_s(n)$ is the universal $C^*$-algebra generated by $n^2$ elements $u_{ij}$, $1\leq i,j\leq n$, with relations
\[\begin{array}{rcl}u_{ij}=u_{ij}^*=u_{ij}^2 & \& & \sum_j u_{ij}=\sum_i u_{ij} =1.\end{array}\] 
\end{definition}

It is not clear from the definition whether the $C^*$-algebras $A_s(n)$ are semiprojective or not. For $n\in\{1,2,3\}$ one easily finds $A_s(n)\cong\mathbb{C}^{n!}$ so that we have semiprojectivity in that cases. 
For higher $n$ one might expect semiprojectivity of $A_s(n)$ because of the formal similarity to graph $C^*$-algebras. In fact, their definition only involves finitely many projections and orthogonality resp. order relations between them. Since graph $C^*$-algebras associated to finite graphs are easily seen to be semiprojective, one might think that we also have semiprojectivity for the quantum permutation algebras. This was even erroneously claimed to be true in \cite[example 2.8(vi)]{Bla04}. In this section we will show that the $C^*$-algebras $A_s(n)$ are in fact not semiprojective for all $n\geq 4$.

One can reduce the question for semiprojectivity of these algebras to the case $n=4$. The following result of Banica and Collins shows that the algebra $A_s(4)$ is 4-subhomogeneous, so that our machinery applies. The idea is to get enough information about the primitive spectrum of $A_s(4)$ to show that it contains closed subsets of dimension strictly greater than 1. This will then contradict part (2) of \ref{thm structure}, so that $A_s(4)$ cannot be semiprojective.\\

We follow notations from \cite{BC08} and denote the Pauli matrices by
\[\begin{array}{cccc}
  c_1=\begin{pmatrix}1 & 0 \\ 0 & 1\end{pmatrix},&c_2=\begin{pmatrix}i & 0 \\0 & -i\end{pmatrix},&c_3=\begin{pmatrix}0&1\\-1&0\end{pmatrix},&c_4=\begin{pmatrix}0&i\\i&0\end{pmatrix}.
\end{array}\]
Set $\xi_{ij}^x=c_ixc_j$ and regard $\M_2$ as a Hilbert space with respect to the scalar product $<a|b>=\tr(b^*a)$. Then for any $x\in SU(2)$ we find $\{\xi^x_{ij}\}_{j=1..4}$ and $\{\xi^x_{ij}\}_{i=1..4}$ to be a basis for $\M_2$. Under the identification $\M_4\cong\mathcal{B}(\M_2)$, Banica and Collins studied the following representation of $A_s(4)$: 

\begin{proposition}[Theorem 4.1 of \cite{BC08}]\label{prop pauli rep}
The $^*$-homomorphism given by
\[\begin{array}{rl}
  \pi\colon A_s(4)\longrightarrow & \C(SU(2),\M_4) \\ u_{ij}\mapsto & \left(x\mapsto\text{rank one projection onto}\; \mathbb{C}\cdot \xi^x_{ij}\right)
\end{array}\]
is faithful. It is called the Pauli representation of $A_s(4)$.
\end{proposition}

For the remainder of this section let $S$ denote the following subset of $SU(2)$:
\[S:=\left\{\begin{pmatrix}\lambda & -\overline{\mu} \\ \mu &\overline{\lambda}\end{pmatrix}\in SU(2)\colon\min\left\{|\Re(\lambda\mu)|,|\Im(\lambda\mu)|,|\Re(\overline{\lambda}\mu)|,|\Im(\overline{\lambda}\mu)|,\left||\lambda|-|\mu|\right|\right\}=0\right\}\]
We will now study the representations of $A_s(4)$ obtained by composing the Pauli representation with a point evaluation. As we will see, most points of $SU(2)$ lead to irreducible representations which are furthermore locally pairwise inequivalent.

\begin{lemma}\label{lemma A4 irreducible}
The representation $\pi_x=\ev_x\circ\pi\colon A_s(4)\rightarrow\M_4$ is irreducible for every $x\in SU(2)\backslash S$.
\end{lemma}

\begin{proof}
Let $x=\begin{pmatrix}\lambda &-\overline{\mu}\\ \mu & \overline{\lambda}\end{pmatrix}\in SU(2)\backslash S$ be given, we show that the commutant of $\pi_x(A_s(4))$ equals the scalars. Therefore we will check the matrix entries of the elements $\pi_x(u_{ij})$ with respect to the orthonormal basis 
$\left\{\frac{1}{\sqrt{2}}\xi_{11}^x,\frac{1}{\sqrt{2}}\xi_{12}^x,\frac{1}{\sqrt{2}}\xi_{13}^x,\frac{1}{\sqrt{2}}\xi_{14}^x\right\}$ of $M_2\cong\mathbb{C}^4$. 
Since in this picture $\pi_x(U_{1i})$ equals the elementary matrix $e_{ii}$, every element in $\left(\pi_x(A_s(4))\right)'$ is diagonal. But we also find
\[\begin{array}{rl}
  \left(\pi_x(U_{23})\right)_{12} &=\frac{1}{2}<\pi_x(U_{23})\xi^x_{12}|\xi^x_{11}> \\ &=\frac{1}{4}<\xi^x_{12}|\xi^x_{23}><\xi^x_{23}|\xi^x_{11}>=4\cdot \Re(\lambda\mu)Im(\lambda\mu)\neq 0, \\ 
  \left(\pi_x(U_{22})\right)_{13} &=\frac{1}{4}<\xi^x_{13}|\xi^x_{22}><\xi^x_{22}|\xi^x_{11}>=2\cdot \Re(\lambda\mu)(|\lambda|^2-|\mu|^2)\neq 0, \\  
  \left(\pi_x(U_{22})\right)_{14} &=\frac{1}{4}<\xi^x_{14}|\xi^x_{22}><\xi^x_{22}|\xi^x_{11}>=-2\cdot \Im(\lambda\mu)(|\lambda|^2-|\mu|^2)\neq 0.
\end{array}\]
So the only elements of $\M_4$ commuting with all of $\pi_x(A_s(4))$ are the scalars.
\end{proof}

\begin{proposition}\label{prop A4 inequivalent}
Every $x\in SU(2)\backslash S$ admits a small neighborhood $V\subseteq SU(2)\backslash S$ such that for all distinct $y,y'\in V$ the representations $\pi_y$ and $\pi_{y'}$ are not unitarily equivalent.
\end{proposition}

\begin{proof}
Let $x=\begin{pmatrix}\lambda_0&-\overline{\mu_0}\\ \mu_0&\overline{\lambda_0}\end{pmatrix}\in SU(2)\backslash S$ be given, then 
\[\epsilon:=\min\left\{|\Re(\lambda_0\mu_0)|,|\Im(\lambda_0\mu_0)|,|\Re(\overline{\lambda_0}\mu_0)|,|\Im(\overline{\lambda_0}\mu_0)|,\left||\lambda_0|-|\mu_0|\right|,|\lambda_0|\right\}>0.\] 
Define a neighborhood $V\subseteq SU(2)\backslash S$ of $x$ by
\[V=\left\{\begin{pmatrix}\lambda &-\overline{\mu}\\ \mu & \overline{\lambda}\end{pmatrix}\in SU(2)\backslash S:|\lambda-\lambda_0|<\frac{\epsilon}{3},|\mu-\mu_0|<\frac{\epsilon}{3}\right\}.\]
Now let $y,y'\in V$ with unitarily equivalent representations $\pi_y$ and $\pi_{y'}$ be given. We compute the value
\[\begin{array}{rl}
  \|\pi_{y}(U_{11}U_{22})\| &=\frac{1}{4}\|(<-\;|\xi^{y}_{11}>\xi^{y}_{11})\circ(<-\;|\xi^{y}_{22}>\xi^{y}_{22})\| \\ 
  &=\frac{1}{4}|<\xi^{y}_{22}|\xi^{y}_{11}>|\cdot\|(<-\;|\xi^{y}_{22}>\xi^{y}_{11})\| \\ &=\frac{1}{4}|<\xi^{y}_{22}|\xi^{y}_{11}>|\cdot\|\xi^{y}_{22}\|\|\xi^{y}_{11}\| \\ &=\left||\lambda|^2-|\mu|^2\right|
\end{array}\]
which is invariant under unitary equivalence. So we find $\left||\lambda|^2-|\mu|^2|\right|=\left||\lambda'|^2-|\mu'|^2\right|$. This implies
\[\begin{array}{rcl}
  \left(|\lambda|=|\lambda'|\wedge |\mu|=|\mu'|\right) &\vee &\left(|\lambda|=|\mu'|\wedge |\mu|=|\lambda'|\right)
\end{array}\]
because of $|\lambda|^2+|\mu|^2=1=|\lambda'|^2+|\mu'|^2$. By definition of $V$ we have
\[\left||\lambda|-|\mu'|\right|\geq\left||\lambda_0|-|\mu_0|\right|-\left||\lambda|-|\lambda_0|\right|-\left||\mu'|-|\mu_0|\right|>\frac{\epsilon}{3}>0,\]
so that we can exclude the second case. Analogously, computing the invariants $\|\pi_{y}(U_{13}U_{22})\|$ and $\|\pi_{y}(U_{14}U_{22})\|$ gives
\[\begin{tabular}{rcl}
  $|\Re(\lambda\mu)|=|\Re(\lambda'\mu')|$&and&$|\Im(\lambda\mu)|=|\Im(\lambda'\mu')|$
\end{tabular}\]
and checking $\|\pi_{y}(U_{11}U_{42})\|$ and $\|\pi_{y}(U_{11}U_{32})\|$ shows
\[\begin{tabular}{rcl}
  $|\Re(\overline{\lambda}\mu)|=|\Re(\overline{\lambda'}\mu')|$&and&$|\Im(\overline{\lambda}\mu)|=|\Im(\overline{\lambda'}\mu')|$
\end{tabular}.\]
The last four equalities imply $\lambda\mu=\lambda'\mu'$ and $\overline{\lambda}\mu=\overline{\lambda'}\mu'$ by the choice of $V$. Together with $|\lambda|=|\lambda'|$ and $|\mu|=|\mu'|$ we find $(\lambda,\mu)=(\lambda',\mu')$ or $(\lambda,\mu)=(-\lambda',-\mu')$. In the second case we get $|\lambda-\lambda'|=2|\lambda|\geq 2|\lambda_0|-2|\lambda-\lambda_0|\geq \frac{4\epsilon}{3}$ contradicting $|\lambda-\lambda'|\leq|\lambda-\lambda_0|+|\lambda'-\lambda_0|<\frac{2\epsilon}{3}$ by the choice of $V$. It follows that $y=y'$.
\end{proof}

By now we have obtained enough information about $\prim(A_s(4))$ to show that it does not satisfy condition (2) of Theorem \ref{thm structure}. Hence we find:  

\begin{theorem}\label{thm A4 not sp}
The $C^*$-algebra $A_s(4)$ is not semiprojective.
\end{theorem}

\begin{proof}
Choose a point $x_0\in SU(2)\backslash S$ and a neighborhood $V$ of $x_0$ as in Proposition \ref{prop A4 inequivalent}. Since $SU(2)$ is a real 3-manifold, there is a neighborhood of $x_0$ contained in $V$ which is homeomorphic to $\mathbb{D}^3=\{x\in\mathbb{R}\colon\|x\|\leq 1\}$. The restriction of the Pauli representation $\pi$ to this neighborhood gives a $^*$-homomorphism $\varphi\colon A_s(4)\rightarrow\C(\mathbb{D}^3,\M_4)$ with the property that $\ev_x\circ\varphi$ and $\ev_y\circ\varphi$ are irreducible but not unitarily equivalent for all distinct $x,y\in\mathbb{D}^3$. The pointwise surjectivity of $\varphi$ given by Lemma \ref{lemma A4 irreducible} and a Stone-Weierstra{\ss} argument (\cite[Theorem 3.1]{Kap51}) show that $\varphi$ is in fact surjective. This implies that $\prim_4(A_s(4))$ contains a closed 3-dimensional subset and hence $\dim(\prim_4(A_s(4)))\geq 3$. As a consequence, $A_s(4)$ cannot be semiprojective because it is subhomogeneous by Proposition \ref{prop pauli rep} but fails to satisfy condition $(2)$ of Theorem \ref{thm structure}.
\end{proof}

It is not hard to show that semiprojectivity of $A_s(n)$ for some $n>4$ would force $A_s(4)$ to be semiprojective. Since we have just shown that this is not the case, we obtain:

\begin{corollary}\label{cor An not sp}
The $C^*$-algebras $A_s(n)$ are not semiprojective for $n\geq 4$.
\end{corollary}

\begin{proof}
For $n\geq 4$ there is a canonical surjection $\varrho_n\colon A_s(n)\rightarrow A_s(4)$ given by 
\[u^{(n)}_{ij}\mapsto\begin{cases}u^{(4)}_{ij} & \text{if}\; 1\leq i,j\leq 4\\ 1 & \text{if}\; i=j>4 \\ 0 & \text{otherwise}\end{cases}.\]
The kernel of $\varrho_n$ is generated by the finite set of projections $\left\{u_{ij}^{(n)}\colon\varrho_n\left(u_{ij}^{(n)}\right)=0\right\}$. It follows from \cite[Proposition 3]{Sor12}, which extends the idea of \cite[Proposition 5.19]{Neu00}, that semiprojectivity of $A_s(n)$ would imply semiprojectivity of $\varrho_n(A_s(n))=A_s(4)$. Since this is not the case by Theorem \ref{thm A4 not sp}, $A_s(n)$ cannot be semiprojective for all $n\geq 4$.
\end{proof}


\end{document}